\documentclass[a4paper,reqno,index]{amsart}
\usepackage{amsmath, amsfonts, amsthm, amssymb}
\usepackage[maxbibnames=9]{biblatex}
\usepackage{mathrsfs}
\usepackage{fancyhdr}
 
\usepackage{enumerate}   
\usepackage{blkarray}
\usepackage{mathrsfs}
\usepackage[title]{appendix}
\usepackage{amsthm}
\usepackage{xcolor}
\usepackage{graphicx}
\usepackage[normalem]{ulem}
\usepackage{morefloats}
\usepackage{float}
\usepackage{hyperref}
\usepackage{mathtools}

\newcommand{\df}[1]{{\textit{#1}}{\index{#1}}}

 \numberwithin{equation}{section}

\allowdisplaybreaks

\newcommand{\sz}{\mathfrak{s}}
\makeindex

\hypersetup{
    bookmarks=true,         
    unicode=false,          
    pdftoolbar=true,        
    pdfmenubar=true,        
    pdffitwindow=false,     
    pdfstartview={FitH},    
    pdftitle={My title},    
    pdfauthor={Author},     
    pdfsubject={Subject},   
    pdfcreator={Creator},   
    pdfproducer={Producer}, 
    pdfkeywords={keyword1, key2, key3}, 
    pdfnewwindow=true,      
    colorlinks=false,       
    linkcolor=green,          
    citecolor=green,        
    filecolor=green,      
    urlcolor=green,           
    urlbordercolor={1 1 1}  
}

\newtheorem{theorem}{Theorem}[section]
\newtheorem*{theorem*}{Theorem}
\newtheorem*{conjecture*}{Conjecture}
\newtheorem{lemma}[theorem]{Lemma}

\newcommand{\DD}{\mathbb{D}}

\newcommand{\CC}{\mathbb{C}}

\newcommand{\NN}{\mathbb{N}}

\newcommand{\vg}{{\mathsf{g}}}
\newcommand{\bzero}{\mathbf{0}}
\newcommand{\mt}{\vartheta}
\newcommand{\zw}{(z\overline{w})}

\newcommand{\interior}{\operatorname{interior}}
\newcommand{\kk}{f}

\newcommand{\wtk}{\widetilde{\kk}}
\newcommand{\Mult}{\operatorname{Mult}}

\newcommand{\bh}{B}
\newcommand{\hb}{h}

\newcommand{\tOmega}{D}

\newtheorem{proposition}[theorem]{Proposition}

\newtheorem{corollary}[theorem]{Corollary}
\theoremstyle{definition}
\newtheorem{definition}[theorem]{Definition}
\newtheorem{question}[theorem]{Question}
\newtheorem{example}[theorem]{Example}
\theoremstyle{remark}
\newtheorem{remark}[theorem]{Remark}

 \usepackage{etoolbox}

\begin{document}

\title[\small The complete Pick property for pairs]{\large The complete Pick property for pairs of kernels and Shimorin's factorization}

\author[\small MCCULLOUGH]{SCOTT MCCULLOUGH}
\address{DEPARTMENT OF MATHEMATICS, UNIVERSITY OF FLORIDA, GAINESVILLE, FL}
\email{sam@ufl.edu} 
\author[\small TSIKALAS]{GEORGIOS TSIKALAS}
\address{DEPARTMENT OF MATHEMATICS, VANDERBILT UNIVERSITY, NASHVILLE, TN}
\email{georgios.tsikalas@vanderbilt.edu}

\thanks{Tsikalas partially supported by National Science Foundation Grant DMS 2054199. Also funded by the Onassis Foundation - Scholarship ID: F ZR 061-1/2023-2024.}

\subjclass[2010]{46E22} 
\keywords{pairs of reproducing kernels, complete Pick property, complete Carathéodory property, interpolation}
\small
\begin{abstract}
    \small
       Let $(\mathcal{H}_k, \mathcal{H}_{\ell})$ be a pair of Hilbert function spaces with kernels $k, \ell$. In a 2005 paper, Shimorin showed that a certain factorization condition on $(k, \ell)$ yields a commutant lifting theorem for multipliers $\mathcal{H}_k\to\mathcal{H}_{\ell}$, thus unifying and extending previous results due to Ball-Trent-Vinnikov and Volberg-Treil. Our main result is a strong converse to Shimorin's theorem for a large class of holomorphic pairs $(k, \ell),$ which leads to a full characterization of the complete Pick property for such pairs. We also present a short alternative  proof of sufficiency for Shimorin's condition. Finally, we establish necessary conditions for abstract pairs $(k, \ell)$ to satisfy the complete Pick property, further generalizing Shimorin's work with proofs that are new even in the single-kernel case $k=\ell.$ Our  approach differs from Shimorin's in that we do not work with the Nevanlinna-Pick problem directly; instead, we are able to extract vital information for $(k, \ell)$ through Carath\'eodory-Fej\'er interpolation.

\end{abstract}
\maketitle

 \par \normalsize
 \section{Introduction} \label{INTROSEC} \normalsize
Given $n$ points $z_1, \dots, z_n$ in the unit disc {$\DD=\{|z|<1\}$} \index{$\DD$} in the complex plane and $n$ complex numbers $w_1, \dots, w_n,$ when does there exist a holomorphic function $\phi:\DD\to\overline{\DD}$ such that 
\[\phi(z_i)=w_i \hspace{0.1 cm} \text{ for } 1\le i\le n?\]
Pick's theorem  \cite{Pickoriginal} tells us that such a $\phi$ exists if and only if the $n\times n$ self-adjoint matrix 
\begin{equation}
 \bigg[\frac{1-w_i\overline{w_j}}{1-z_i\overline{z_j}}\bigg]
\end{equation}
is positive semi-definite, henceforth simply \df{positive} or \df{PsD}. This result has had a major, and continuing,  impact on both function theory and operator theory. The operator theory approach to interpolation, pioneered by Sarason \cite{Sarasongeneralized}, rests on interpreting Pick's Theorem in terms of a partially defined multiplier.
  \par
 This operator theoretic multiplier view of Sarason  was dramatically expanded upon in unpublished work of Agler by  viewing Pick’s theorem as a special property of multipliers of reproducing kernel Hilbert spaces. 
A reproducing kernel Hilbert space is a \df{complete Pick space} if it hosts a matrix-valued analogue of Pick's theorem (see Definition~\ref{CPdef} below).  Examples include the classical Hardy space $H^2$, standard
weighted Dirichlet spaces on the unit disc, the Sobolev space $W^2_1$ on the
unit interval and the Drury–Arveson space $H^2_d$ on the unit ball of $\mathbb{C}^d.$  A complete characterization of complete Pick spaces was achieved through work of the first author \cite{McCulloughlocal, McCulloughcarath} and Quiggin \cite{Quiggin}, while Agler and McCarthy \cite{Pickbook} proved that $H^2_d$ is universal among such spaces (see subsection~\ref{CPsubsec}). Since their inception, complete Pick spaces have  proven to be a natural venue for a number of function-theoretic operator theory topics such as Carleson measures \cite{Rochbergballs}, interpolating sequences \cite{MarshallSundberg,AHMRinterpol, chalmoukis2023simply}, invariant subspaces \cite{Invsub}, factorizations \cite{AHMRfact, freeouter}, weak product spaces \cite{weakPickprod,AHMRweak}, the column-row property \cite{ColRow} and operator models \cite{hyperrigidity,Pickchar}. The book \cite{Pickbook} is a standard reference (see also the surveys \cite{Hartzsurvey, shalit2014operator}). \par 
Here we will study interpolation for multipliers between spaces. Accordingly, let  $\mathcal{H}_k, \mathcal{H}_{\ell}$  denote reproducing kernel Hilbert  spaces on the non-empty set $X,$ with kernels $k$ and $\ell,$ respectively.  Thus, for instance $k:X\times X\to \CC$ and  $\mathcal{H}_k$ is a Hilbert space of functions on $X$ such that for each $y\in X,$ the function $k_y:X\to  \CC$
defined by $k_y(x)=k(x,y)$ reproduces functions $f\in \mathcal{H}_k$ at $y$ in the sense that $f(y)=\langle f,k_y\rangle.$ Given a positive integer $N$, let {$ \mathbb{M}_N$} \index{$\mathbb{M}_N$} denote the set of all $N\times N$ matrices with complex entries. A function $\Phi: X\to \mathbb{M}_N$ is a \df{multiplier} $\mathcal{H}_k\otimes\mathbb{C}^N\to\mathcal{H}_{\ell}\otimes\mathbb{C}^N$ if $f\in \mathcal{H}_k\otimes \CC^N$ implies $\Phi f\in \mathcal{H}_\ell\otimes \CC^N.$ In this case, the closed graph theorem implies $\Phi$ determines a bounded linear operator $M_\Phi$  \index{$M_\Phi$}  from $\mathcal{H}_k\otimes \CC^N$ to $\mathcal{H}_\ell\otimes \CC^N$ given by $M_\Phi f =\Phi f$ and we let {$\Mult(\mathcal{H}_k\otimes \CC^N,\mathcal{H}_\ell\otimes \CC^N)$}  \index{$\Mult(\mathcal{H}_k\otimes \CC^N,\mathcal{H}_\ell\otimes \CC^N)$} denote the set of these multipliers. In the case that $k=\ell$ and $N=1,$ we write $\Mult(\mathcal{H}_k).$ The \df{norm of a multiplier} refers to the norm of the operator $M_\Phi.$
 \par 
For $w\in X,$ letting $k_w=k(\cdot, w)\in \mathcal{H}_k$\index{$k_w$}, a routine computation shows if $\Phi\in  \Mult\big(\mathcal{H}_k\otimes \CC^N,\mathcal{H}_\ell\otimes \CC^N\big)$ and $h\in \CC^N,$ then $M_\Phi^* k_w\otimes h= k_w\otimes \Phi(w)^*h.$ It is well-known, and not difficult to verify,  that a function $\Phi: X\to \mathbb{M}_N$ is a multiplier  of norm at most one if and only if 
\[
  X\times X\ni (z,w)\mapsto \ell(z, w)I_{N\times N}-k(z, w)\Phi(z)\Phi(w)^* 
\]
is a (positive) kernel. Thus, given an  $n,$ points $z_1, \dots, z_n\in X$ and $n$ matrices $W_1, \dots, W_n\in \mathbb{M}_N$, a necessary condition for the existence of a multiplier $\Phi \in\Mult\big(\mathcal{H}_k\otimes\mathbb{C}^N, \mathcal{H}_{\ell}\otimes \mathbb{C}^N\big)$ of norm at most one that satisfies $\Phi(z_i)=W_i,$ for all $1\le i\le n,$ is that the block matrix 
\begin{equation}  \label{introinitial}
 \begin{bmatrix}\ell(z_i, z_j)I_{N\times N}-k(z_i, z_j)W_iW^*_j \end{bmatrix}_{i,j=1}^n
\end{equation}
 is positive.  Returning to Pick's Theorem, a bounded analytic function $\varphi:\DD\to \CC$ determines a multiplier of Hardy-Hilbert space $H^2(\DD).$ The space $H^2(\DD)$ is a reproducing kernel Hilbert space whose kernel is Szegő's kernel $\sz(z,u) = (1-z\overline{u})^{-1}$ for $z,u\in \DD;$ that is, if  $f\in H^2(\DD)$ and $u\in\DD,$ then, 
\[
 \langle f,\sz_u\rangle =f(u).
\]
The matrix-valued version of Pick's Theorem says, in the case $k=\ell=\sz,$ that the necessary condition above is also sufficient.

\begin{definition} \label{CPdef}
A pair $(k, \ell)$ of kernels on  $X$ is a \df{complete Pick pair}  if, for every positive integer $n,$ for every choice of points $z_1, \dots, z_n\in X,$ for every positive integer $N$ and for every choice of  matrices $W_1, \dots, W_n\in\mathbb{M}_N$  for  which  the matrix in \eqref{introinitial} is positive, there exists a multiplier $\Phi \in\Mult\big(\mathcal{H}_k\otimes\mathbb{C}^N, \mathcal{H}_{\ell}\otimes \mathbb{C}^N\big)$ of norm at most one that satisfies, 
 \begin{align*}
  \Phi(z_i)=W_i,
\end{align*}
for $i=1,2,\dots,n.$ 
\qed
\end{definition} 

 In the special case $k=\ell$, the kernel $k$ is known in the literature  as a \df{complete Pick kernel}. As the prototypical example, Szegő's kernel is a complete Pick kernel.  Since they play an important role in this paper, a further discussion of complete Pick kernels appears in Subsection~\ref{CPsubsec} below.  
The expressions \df{$(k, \ell)$ is a complete Pick pair}, \df{$(\mathcal{H}_k, \mathcal{H}_{\ell})$ is a complete Pick pair}, \df{$(k, \ell)$ has the complete Pick property} will be used interchangeably in the sequel. Also, for brevity, we will often use the terms \df{CP pair} and \df{CP property} instead. \par 
It is expected that many deep properties of complete Pick kernels have natural extensions in the setting of complete Pick pairs. In fact, such extensions have already been worked out for reproducing kernels with a \df{complete Pick factor} (see e.g. \cite{AHMRfact, InvPickfactor, Interpoltsik, Sarkarpaper} and \cite[Section 3.8]{ColRow}). These are kernels $\ell$ with the property that there exist a complete Pick kernel $s$ and a kernel $g$ such that \[\ell=s\cdot g,\] giving rise to the complete Pick pair $(s, \ell)$ (see \cite[Proposition 4.4]{AHMRinterpol}). Another class of examples is implicitly contained in work of Volberg and Treil \cite{Volbtreil}, who proved a commutant lifting theorem for certain pairs of function spaces on the unit disc. A consequence of their results is that if $\mathcal{H}_k$ and $\mathcal{H}_{\ell}$ are (normalized) Hilbert spaces of holomorphic functions on $\DD$ such that the shift $S: z\mapsto zf(z)$ is contractive on $\mathcal{H}_{\ell}$ and expansive on $\mathcal{H}_k,$ then $(\mathcal{H}_k, \mathcal{H}_{\ell})$ is a complete Pick pair. We note that the spaces $\mathcal{H}_k$ in this setting can be viewed as generalized \textit{de Branges-Rovnyak spaces}; see \cite{Higheroder, AleksandrovClark, AlemanBartosz} for recent activity concerning such spaces. \par 
A substantial contribution to the theory of complete Pick pairs can be found in the 2005 paper \cite{Shimorin} of Shimorin, who proved a commutant lifting theorem for pairs of spaces that unifies and extends analogous results found in \cite{BTVCommLift} and \cite{Volbtreil}. In particular, Shimorin was able to show that a certain factorization property of $k$ and $\ell$ is sufficient for $(k, \ell)$ to be a complete Pick pair. 
\begin{definition}  \label{strcertdef}
Assume $k, \ell$ are kernels on $X.$ A kernel $s$ on $X$ will be called a \df{strong Shimorin certificate for $(k, \ell)$} if it is a complete Pick kernel and also there exist kernels $h, g$ on $X$ with
\begin{equation} \label{intshimfact}
 k=(1-h)s \hspace{0.3 cm} \text{ and }  \hspace{0.3 cm} \ell=sg.  
\end{equation}
\end{definition}
\begin{theorem}[\cite{Shimorin}, Theorem 1.3]\label{mainShim}
  Suppose $(k, \ell)$ is a pair of kernels on the set $X.$ If there exists a strong
 Shimorin certificate for $(k,\ell),$  then $(k, \ell)$ is a complete Pick pair.   
\end{theorem} 
Setting $X=\DD$ and $s=\sz$ in \eqref{intshimfact} and choosing $\mathcal{H}_{\ell}$ to be a holomorphic space allows one to recover the Volberg-Treil class of pairs of kernels described above, while setting $h\equiv 0$ leads to pairs $(s, \ell)$ where $s$ is a complete Pick factor of $\ell.$ 
Shimorin actually worked with operator-valued pairs $(k, \ell)$. His proof of Theorem~\ref{mainShim} rested on recasting the Pick interpolation problem in this setting as a matrix completion problem, which then allows for the use of Parrott's Lemma, a strategy previously employed in \cite{Quiggin} and \cite{AmcCcompleteNP}, as well as in unpublished work of Agler,  in  the single-kernel setting. We also point out that in order to prove the full commutant lifting theorem for $(\mathcal{H}_k, \mathcal{H}_{\ell})$, Shimorin imposed an additional regularity condition on $k$ (see condition  ``(R)" in the statement of \cite[Theorem 1.1]{Shimorin}), which, roughly, ensures that $\text{Mult}(\mathcal{H}_k)$ shares sufficiently many elements with $\text{Mult}(\mathcal{H}_s)$. An alternative approach to the commutant lifting theorem in the important special case that $k=s$ is the Drury-Arveson kernel can be found in \cite{Sarkarpaper} (if $k=s$, then Shimorin's condition (R) is automatically satisfied; see the discussion in the second half of \cite[p. 139]{Shimorin}).

\par In Section~\ref{suffforCP}, we will give a new, short proof of Theorem~\ref{mainShim} that utilizes a version of Leech's Theorem valid in the context of a complete Pick kernel. Further, in Section~\ref{ShimnstrShim}, we will consider the notion of a \df{Shimorin certificate} for a pair $(k, \ell)$ (see Definition~\ref{mostgeneral}). Our definition is motivated by Shimorin's proof of Theorem~\ref{mainShim} and, in particular, the consideration of the precise conditions that have to be satisfied in order to extend a multiplier $\mathcal{H}_k\to\mathcal{H}_{\ell}$
on a set of points to one more point. Every pair that possesses a strong Shimorin certificate also has a Shimorin certificate, though the converse does not always hold (see Example~\ref{anykernel}). Still, it turns out that the existence of a Shimorin certificate for $(k, \ell)$ continues to be sufficient for the complete Pick property to hold, which leads us to a generalization of Theorem~\ref{mainShim} (see Theorem~\ref{generalsuff}). \par

\subsection{Diagonal complete Pick pairs}\
Before we state our first main result, we establish some notation and terminology.  A kernel $\kk(z,w)$ on a domain\footnote{Following the usual convention, \df{domain} means open,  connected and non-empty.} \index{domain} $\Omega\subseteq \mathbb{C}^\vg$ is a \df{holomorphic kernel} if $\kk:\Omega\times\Omega\to \CC$ is a kernel (a PsD function) that is holomorphic in $z$ and conjugate holomorphic in $w.$   Let $\mathbb{N}$ \index{$\mathbb{N}$} denote the set of all non-negative integers and fix an integer $\vg \ge 1$. Given $a\in\mathbb{N}^{\mathsf{g}}$ and $z\in\mathbb{C}^{\mathsf{g}},$ let \index{$\vert a \vert$} \index{$z^a$}
\[a=(a_1, \dots, a_{\vg}), \hspace{0.4 cm} |a|=|a_1|+\dots+|a_{\vg}|, \hspace{0.4 cm} z^a=z^{a_1}_1\cdots z_{\vg}^{a_{\vg}}.\]
 Also, let $\mathbf{0}=(0,0,\dots,0)\in \mathbb{N}^\vg.$ As a definition, assuming $\mathbf{0}\in\Omega,$  the holomorphic  kernel $\kk$ 
 is a \df{diagonal holomorphic kernel}  if there exists a sequence of positive numbers $\{\kk_a\}_{a\in\mathbb{N}^{\mathsf{g}}}$ such that the power series expansion of $\kk$ at $\mathbf{0}$ takes the form
\begin{equation} \label{serexp}
\kk(z, w)=\sum_{a\in \mathbb{N}^{\mathsf{g}}}\kk_az^a\overline{w}^a, \hspace*{0.3 cm} z, w\in\Omega.
\end{equation}
Finally, a diagonal holomorphic $k$ will be termed \df{normalized} if $k_{\mathbf{0}}=1$ in \eqref{serexp}. A further discussion of such kernels appears in  subsection~\ref{diaghomprelim}. \par

Our main result is a strong converse to Shimorin's Theorem~\ref{mainShim} for diagonal holomorphic kernels. 
\begin{theorem}\label{intromain}
 A  pair of diagonal holomorphic kernels $(k,\ell)$ on a domain  $\mathbf{0}\in \Omega$ is a complete Pick pair if and only if it possesses a strong Shimorin certificate. 

\end{theorem}

In the context of Theorem~\ref{intromain}, more is true. Not only does a complete Pick pair of diagonal holomorphic kernels have a strong Shimorin certificate, but it has a distinguished certificate $s$ that depends only upon $k.$ This kernel $s$ is, in a sense, {\it minimum} among all diagonal holomorphic certificates for $(k,\ell).$ Moreover, there are restrictions placed on the domain $\Omega$ in terms of the domains of convergence of the power series for $k$ and $\ell$ as
 discussed below.  These domain restrictions are  illustrated by  examples involving Bergman-type kernels in Section~\ref{mainmain} (Examples~\ref{finallyanex}, \ref{finallyanex2}) and Section~\ref{closerl}. 
   \par

Before stating Theorem~\ref{introextmain} below, two additional ingredients require an introduction. The first is a canonical domain for holomorphic diagonal kernels.  For  a  diagonal holomorphic kernel $\kk$ on a domain $\mathbf{0}\in\Omega\subseteq \CC^\vg$ with series expansion at the origin as  in equation~\eqref{serexp}, let $\Omega_\kk$ denote the domain of convergence (see \cite[Definition~2.3.11]{MR1871333}) of the series $\sum_a \kk_a x^{2a}.$  In particular, $\Omega_\kk$ is a domain\footnote{It is a logarithmically convex complete Reinhardt domain and is thus, in particular, star-like with respect to the origin.} and 
\begin{equation}
\label{Omegakdef}
 \Omega_\kk
 =\interior \mathcal{C}_\kk =\cup_{r>0} \{x: \sum_a \kk_a|y|^{2a} <\infty, \ \ \text{ for all } \|x-y\|<r\},
\end{equation}
 where $|x|=(|x_1|, \dots, |x|_{\vg})$ and
\[
 \mathcal{C}_\kk=\{x\in\CC^\vg: \sup \kk_a |x|^{2a} <\infty\}.
\]
  In this case, $\Omega\subseteq\Omega_\kk$ and $\kk$ extends to be a diagonal holomorphic kernel on $\Omega_\kk$; see Proposition~\ref{prop:Pringsheim}. 

The second ingredient is the notion of a master certificate. 

\begin{definition} \label{mastercert}
Let $k$ be a normalized diagonal holomorphic kernel on $\Omega\subseteq\mathbb{C}^{\mathsf{g}},$ so that 
\[
 k(z, w)=1+\sum_{a\in\mathbb{N}^{\mathsf{g}}\setminus\{0\}}k_az^a\overline{w}^a.
\]
Define $\mt_{\mathbf{0}}=0$ and $\mt_{e_j}=k_{e_j}$ for all $1\le j\le \vg$. Assuming $M\ge 1$ and $\mt_b$ has been defined for all $|b|\le M$, let, for $|b|=M+1$,
\[\mt_b=\max\Bigg\{0,k_b-\sum_{\substack{w+u=b, \\ w, u\neq \mathbf{0}}}\mt_w k_u\Bigg\}.\]
The \df{master certificate associated with $k$} is the formal power series in $\vg$ complex variables defined as\footnote{If not for the max above, we would have $\mt =1-\frac{1}{k}.$ Moreover, in the case that $k$ is a complete Pick kernel, $\mt=1-\frac{1}{k}.$} 
\[\mt(x)=\sum_{b\in\mathbb{N}^{\mathsf{g}}}\mt_bx^b.\]
Set
\begin{equation}
\label{OmegaTheta1}
  \Omega_\vartheta^1=  \{x\in\Omega_\vartheta: \sum_{a\ne \mathbf{0}} \vartheta_a |x|^{2a}<1\}. \qedhere
\end{equation}
\end{definition}

We can now state a refined version of Theorem~\ref{intromain}.  For convenience, we will work with normalized kernels. 

\begin{theorem} \label{introextmain}
 Let $(k, \ell)$ be a pair of normalized diagonal holomorphic kernels on a domain $\mathbf{0}\in\Omega \subseteq \mathbb{C}^{\mathsf{g}}$ and let $\mt$ denote the master certificate from  Definition~\ref{mastercert}. If $(k,\ell)$ is a complete Pick pair, then
\[
  \Omega  \subseteq \Omega_\ell \subseteq \Omega_\vartheta^1 \subseteq \Omega_k. 
\]
  Moreover,
   the following are equivalent:
 \begin{enumerate}[(i)]
     \item  $(k, \ell)$ is a complete Pick pair on $\Omega;$
     \item  $(k, \ell)$ has a Shimorin certificate on $\Omega;$
     \item  $(k, \ell)$ has a strong Shimorin certificate on $\Omega;$
     \item \label{i:introextmain:4}  $s(z, w)=\frac{1}{1-\mt(z\overline{w})}$ is a strong Shimorin certificate for $(k, \ell)$ on $\Omega;$ 
     \item[(v)] there exists a positive kernel $g$ on $\Omega$ such that
 \[
   \ell(z, w)=\frac{g(z, w)}{1-\mt(z\overline{w})}
\]
 on $\Omega.$
 \end{enumerate}
\end{theorem}

Theorem~\ref{introextmain} is expanded upon as Theorem~\ref{mainextmain} and
 proved in Section~\ref{mainmain}.

\subsection{Carath\'eodory  interpolation for diagonal pairs}
Our proof of  Theorem~\ref{mainextmain}, and thus Theorem~\ref{introextmain},  does not proceed directly from the  complete Pick property as Shimorin's proof in \cite[Theorem~1.3]{Shimorin} did. Instead we obtain our necessary and sufficient conditions through consideration of the \textit{Carath\'eodory interpolation problem} for pairs of (normalized) diagonal kernels. 
  \par
Recall that the classical Carath\'eodory interpolation problem for $H^{\infty}$ (also termed the \textit{Carath\'eodory-Fej\'er problem} in the literature), originally introduced and studied by  Carath\'eodory \cite{originalcarathproblem1, originalcarathproblem2}, can be phrased as follows: Given $c_0, c_1, \dots, c_n \in\mathbb{C},$ when do there exist complex numbers $c_{n+1}, c_{n+2}, \dots$ so that $\phi(z)=\sum_{i=0}^{\infty}c_iz^i$ is analytic and satisfies $\sup_{z\in\mathbb{D}}|\phi(z)|\le 1$? This problem has since been studied in many different settings (see \cite[p. 186]{Sarasongeneralized} for additional references, also \cite{Car1, Car2, Car3, Car4, Car5, Car6, Car7}).
 \par
Our  interest lies  in a matrix-valued version of Carath\'eodory interpolation that applies to pairs of kernels. Assume $(k, \ell)$ is a pair of (normalized) diagonal holomorphic kernels on  a domain $\Omega\subseteq \mathbb{C}^{\mathsf{g}}$ and let $F\subseteq\mathbb{N}^{\mathsf{g}}$ be a co-invariant set of indices (see Definition~\ref{coinvdef}). Given a positive integer $J$ and a collection of matrix coefficients $\{c_a : a\in F\}\subseteq\mathbb{M}_J$, when does there exist a collection
$\{c_a : a\in \mathbb{N}^{\mathsf{g}}\setminus F\}\subseteq\mathbb{M}_J$
such that the function
\begin{equation}\label{CCinterpolant}
 \Phi(z)=\sum_{a\in\mathbb{N}^{\mathsf{g}}}z^a\otimes c_a   
\end{equation}
is a contractive multiplier from $\mathcal{H}_k\otimes\mathbb{C}^J$ to $\mathcal{H}_{\ell}\otimes\mathbb{C}^J$? 
A necessary condition for the existence of such a multiplier is that the block upper-triangular matrix $C$ indexed by $F\times F$ with block $J\times J$ entries 
 \begin{equation}\label{data}
     C_{a, b}=\begin{cases}
     c_{b-a}^*\sqrt{\cfrac{k_a}{\ell_b}}, \hspace{0.5 cm} b\ge a, \\
      \mathbf{0},  \hspace{1.05 cm} \text{otherwise,}
 \end{cases}  \end{equation}
is a contraction, where $b\ge a$ means $b_j\ge a_j$ for $a=(a_1,\dots,a_\vg)$ and similarly for $b.$ If, for any $J\ge 1$ and any choice of co-invariant $F\subseteq\mathbb{N}^{\mathsf{g}}$ and matrices $\{c_a : a\in F\}\subseteq\mathbb{M}_J$, this necessary condition is also sufficient for the existence of a contractive multiplier $\Phi$ as in \eqref{CCinterpolant}, we say  $(k, \ell)$ is a \df{complete  Carath\'eodory pair}. In the special case $k=\ell,$ the kernel $k$ is known as a \textit{complete  Carath\'eodory kernel}. For a more extensive discussion of this definition and further background on the classical setting, see subsection~\ref{CCsubs}. As with the complete Pick property, the expressions \textit{$(k, \ell)$ is a complete Carath\'eodory pair}, \textit{$(\mathcal{H}_k, \mathcal{H}_{\ell})$ is a complete Carath\'eodory pair}, \textit{$(k, \ell)$ has the complete Carath\'eodory property} will be used interchangeably in the sequel. We will often replace them with the shorter versions \textit{CC pair} and \textit{CC property}. 
\index{complete Carath\'eodory pair} \index{complete Carath\'eodory property}
\index{CC property}
 \par 
 In \cite{McCulloughcarath}, the first author characterized those kernels that satisfy an abstract version of the complete Carath\'eodory property formulated in terms of a backwards shift operator; see \cite[Theorem 7.2]{McCulloughcarath}. This characterization was further extended in \cite{Mikethesis}, where it was shown that a single holomorphic kernel $k$ (satisfying a few very mild regularity assumptions) is CC if and only if it is CP; see \cite[Corollary 3.3]{Mikethesis}. The fact that this equivalence continues to hold for pairs of kernels will be one of the main components of our proof strategy for Theorem~\ref{introextmain}. First, we show directly that any CP pair of diagonal holomorphic kernels is also a CC pair (Theorem~\ref{directpassage}) and are thus led to  the following question. Does every diagonal CC pair $(k, \ell)$ possess a strong Shimorin certificate? We show that the answer is yes for diagonal kernels and, in fact, in that case we can always choose the $s$ from Theorem~\ref{introextmain}\eqref{i:introextmain:4} as our certificate (Corollaries~\ref{main:formal:CCP} and ~\ref{c:tlg}). The key observation here is that the inequalities obtained by looking at the positivity condition
 \[\text{``$s=(1-\mt)^{-1}$ is a factor of $\ell$"} \]
 coefficient-wise can all be encoded by appropriate choices of contractive block matrices $C$ as in \eqref{data}. We also point out that the factorization $k=(1-h)s$ follows immediately from the definition of $\mt$; see Proposition~\ref{1-k(1-mt)}. After obtaining that $s=(1-\mt)^{-1}$ is a strong Shimorin certificate for the CC pair $(k, \ell),$ an application of Theorem~\ref{mainShim} shows that $(k, \ell)$ is CP, concluding our argument.  In this indirect fashion we obtain the following corollary. 
\begin{theorem} \label{CPisCC}
 A pair of diagonal holomorphic kernels $(k, \ell)$ is a complete Pick pair if and only if it is a  complete Carath\'eodory pair. 
\end{theorem}  

 It is anticipated that, under fairly mild conditions on $k$ and $\ell,$  a limiting argument will prevail
 to show that generally complete Pick pairs are complete Carath\'eodory pairs. However, the proof of Theorem~\ref{CPisCC} offers little insight into the converse, as it passes through the (specific to diagonal pairs) proof of Theorem \ref{main:formal:CCP}.

 \par 
\subsection{The Bergman kernel as an example}
\label{s:negative}
Recall that the conclusion of Theorem~\ref{introextmain} can be reformulated as follows: for every CP pair of diagonal holomorphic kernels $(k, \ell)$, there exists a \textit{minimum} strong Shimorin certificate that only depends on $k$. Unfortunately, this strong canonicity result does not survive if we drop the assumption that $\ell$ is diagonal.  
\begin{theorem}\label{bergnotmin}
Let  $\mathfrak{b}$ denote the Bergman kernel on the unit disc $\mathbb{D}=\{z\in\mathbb{C}: |z|<1\}.$ There exists an  $0<r<1$, and a one-parameter family of CP kernels $\{s_{\lambda}\}_{\lambda\in\Lambda}$ 
with the following properties:
\begin{enumerate}[(i)]
    \item  For every $\lambda\in\Lambda,$ there exists a kernel $h_{\lambda}$ on $r\mathbb{D}$ such that 
\[\mathfrak{b}=(1-h_{\lambda})s_{\lambda};\]
\item  There does not exist a CP kernel $s$ on the disc  $r\mathbb{D}$ such that $\mathfrak{b}=(1-h)s$ for some kernel $h$ on $r\mathbb{D}$ and with the property that, for every $\lambda\in\Lambda,$ there exists a kernel $g_{\lambda}$ on $r\mathbb{D}$ such that 
\begin{equation}\label{namex}
 s_{\lambda}=sg_{\lambda}.   
\end{equation}
\end{enumerate}

 In particular, $(\mathfrak{b},s_\lambda)$ is a family of complete Pick pairs over $r\, \DD,$ but there does not exist a kernel $s$ that serves simultaneously  as a certificate
 for each pair.  
\end{theorem}  
Theorem~\ref{bergnotmin} further implies that, given a CP pair $(\mathfrak{b}, \ell)$, there will not, in general, exist a strong Shimorin certificate that is a factor of every other strong Shimorin certificate for $(\mathfrak{b}, \ell),$ even if $\ell$ is taken to be diagonal; see Corollary~\ref{nonmin}. Thus, the certificate $s$ from Theorem~\ref{introextmain}(iv) will, in general, only be minimum among \df{diagonal holomorphic} strong Shimorin certificates.
\par 
\subsection{Necessary conditions}
Despite the  negative results of subsection~\ref{s:negative}, we do provide new necessary conditions for a pair of abstract kernels to be a CP pair including an  extension of the following theorem of Shimorin.

\begin{theorem}[\cite{Shimorin}, Theorem 1.4]\label{Shimnec}
If $(k, \ell)$ is a complete Pick pair of kernels on $X,$ then there exists a kernel $g$ on $X$ such that $\ell=kg.$   
\end{theorem} 
   \par 

Given a kernel $k$ on a set $X$ and a  non-empty subset $Y\subseteq X$, let  $k^Y$ \index{$k^Y$} denote the reproducing kernel associated with the subspace
\[
 \{f\in\mathcal{H}_k : f(w)=0 \text{ for all }w\in Y\}
\]
 of $\mathcal{H}_k.$
If $Y=\{w\}$ is a singleton, we write $k^w$ in place of $k^{\{w\}}$. \index{$k^w$}  Recall (see also subsection~\ref{CPsubsec}) that  $k$ is a CP kernel if and only if, for every $w\in X,$ there exists a kernel $g_w$ on $X$ such that \begin{equation}\label{singleCPnec}
 k^w=kg_w.   
\end{equation}
Theorem~\ref{generalnecpos} below generalizes this result to pairs of kernels and can be seen as an extension of  \cite[Theorem~1.4]{Shimorin} quoted above as Theorem~\ref{Shimnec}.   Its proof proceeds directly from the complete Pick property. 
\begin{theorem}\label{generalnec}
Let $(k, \ell)$ be a CP pair of kernels on a set $X.$ If  $Y\subseteq X$ is finite, then $(k,\ell^Y)$ is also a CP pair and, consequently, there exists a kernel $g_Y$ on $X$ such that 
\begin{equation} \label{generalnecpos}
 \ell^Y=kg_Y.   
\end{equation}
\end{theorem}
 
 In the special case of a complete Pick kernel, where $\ell=k,$ the argument used in the proof  of Theorem~\ref{generalnec} (see Theorem~\ref{Schurpreserves}), while following the same general approach as other proofs (besides those mentioned already  another proof can be found in \cite{Greg}),  yields a new and somewhat simpler proof of the necessity of \eqref{singleCPnec} for the single kernel case.
\par 
As it turns out, conditions of the form~\eqref{generalnecpos} are far from sufficient for a pair $(k, \ell)$ to be CP, even in the diagonal holomorphic case. In fact, we will see that even replacing $\ell^Y$ by the kernel of any multiplier-invariant subspace of $\mathcal{H}_{\ell}$ will, in general, not be enough to guarantee sufficiency; see Proposition~\ref{necbutnotsuf} and the discussion preceding it. \par 
Finally, in Theorem~\ref{sometheorem}, the necessary and sufficient conditions of Theorem~\ref{introextmain} are formulated without reference  to the certificate $s(z, w)=(1-\mt(z\overline{w}))^{-1},$ which we know may not be available in the non-diagonal case. Even though these conditions are still specific to diagonal holomorphic kernels, we hope that they may provide a more practical framework for a future derivation of new necessary conditions for non-diagonal pairs to be CP. 
\par 
 \subsection{Reader's guide} The paper is structured as follows: in Section~\ref{PRELIMSEC}, we review some basic facts concerning reproducing kernels and the complete Carath\'eodory and Pick properties. In Section~\ref{ShimnstrShim}, we introduce the notion of a Shimorin certificate for a pair $(k, \ell)$ and explore some basic consequences of its existence. Our main results are Propositions~\ref{betweencerts} and \ref{Scertforhol}, which illuminate the relation between Shimorin and strong Shimorin certificates in different settings. Next, in Section~\ref{suffforCP}, we give our alternative proof of Theorem~\ref{mainShim}, which also yields a parametrization of all contractive multipliers $\mathcal{H}_k\to\mathcal{H}_{\ell}$; see Corollary~\ref{parametcorol}. In addition, we generalize Theorem~\ref{mainShim} by showing that the existence of Shimorin certificates suffices for the CP property; see Theorem~\ref{generalsuff}. Section~\ref{theCCsection} is exclusively  devoted to the characterization of the CC property for pairs of diagonal holomorphic kernels. First, we reformulate the CC property in terms of one-point extensions (Proposition~\ref{convenient}), and then show that the existence of a  strong Shimorin certificates is both necessary (Theorem~\ref{almostCCPtoShim}) and sufficient (Theorem~\ref{ShimtoCCP}) for a pair to have the CC property. We also present a direct proof that the CP property implies the CC property in this setting; see Theorem~\ref{directpassage}. In Section~\ref{mainmain}, we gather up our results so far to prove Theorem~\ref{mainextmain}, a refined version of Theorem~\ref{introextmain}. We also compute the master certificates associated with certain Bergman-type kernels; see Examples~\ref{finallyanex} and \ref{finallyanex2}. Further, in Section~\ref{generalnecsecti}, we give necessary conditions for general (not necessarily holomorphic) pairs to be CP, our main results being Theorem~\ref{generalnec} and Proposition~\ref{necbutnotsuf}. In Section~\ref{closerl}, we compute several non-diagonal certificates for CP pairs $(\mathfrak{b}, \ell)$, which are then used to demonstrate how Theorem~\ref{introextmain} can fail if $\ell$ is not taken to be diagonal (Theorem~\ref{bergnotmin} and Corollary~\ref{nonmin}). However, we also show that the domain $\Omega_{\mt}=\frac{1}{\sqrt{2}}\DD$ of the master certificate $\mt$  associated with $\mathfrak{b}$ continues to be, in a certain sense, maximal for CP pairs $(\mathfrak{b}, \ell)$ even if $\ell$ is not diagonal. Section~\ref{reform} contains the previously mentioned reformulation of Theorem~\ref{introextmain} and a brief discussion interpreting the new necessary conditions it contains; see Remarks~\ref{remark81}-\ref{remark83}. We close the paper with some open questions, all contained in Section~\ref{openq}.

\normalsize
 \section{Preliminaries} \label{PRELIMSEC} \normalsize
   Background and preliminary results needed in the remainder of the paper are collected here. Vector-valued reproducing kernel Hilbert spaces associated
    to scalar kernel is the subject of Subsection~\ref{s:vvk};  complete Pick kernels are exposited in Subsection~\ref{CPsubsec};  more details on diagonal holomorphic kernels and their domains are contained in Subsection~\ref{diaghomprelim}; and initial results on Carath\'eodory interpolation for pairs of kernesl are provided in Subsection~\ref{CCsubs}.
    
\subsection{Vector-valued kernels}
\label{s:vvk}
Let $X$ be a non-empty set. A function 
\[k: X\times X\to\mathbb{C}\]
 is a \df{positive kernel} or just a \df{kernel}, denoted $k\succeq 0,$  if it is positive semi-definite in the sense that the $n\times n$ complex matrix $[k(z_i, z_j)]$ is positive semi-definite for every choice of points $z_1, \dots, z_n\in X.$ Here we make the standing assumption that our kernels never vanish along the diagonal. The associated reproducing kernel Hilbert space will be denoted by $\mathcal{H}_k$ or $\mathcal{H}_k(X)$  \index{$\mathcal{H}_k$} \index{$\mathcal{H}_k(X)$} to emphasize the domain,  while the algebra of all functions $\phi: X\to\mathbb{C}$ that multiply $\mathcal{H}_k$ into itself will be denoted by $\Mult(\mathcal{H}_k)$. \par 
We will often have to work with vector-valued versions of these objects. Given a Hilbert space $\mathcal{E}$, we regard elements of
$\mathcal{H}_k\otimes \mathcal{E}$ as $\mathcal{E}$-valued functions on $X.$ It is not hard to see that the reproducing kernel of $\mathcal{H}_k\otimes \mathcal{E}$ is equal to $k\otimes I_{\mathcal{E}}$, where $I_{\mathcal{E}}$ is the identity operator on $\mathcal{E}$. If $\mathcal{F}$ is another Hilbert space and $\ell$ is another kernel on $X,$ we write $\Mult(\mathcal{H}_k\otimes\mathcal{E}, \mathcal{H}_{\ell}\otimes\mathcal{F})$ to denote the space of all $B(\mathcal{E}, \mathcal{F})$-valued functions on $X$ that multiply $\mathcal{H}_k\otimes \mathcal{E}$ into $\mathcal{H}_{\ell}\otimes \mathcal{F}$. A multiplier $\Phi\in \Mult(\mathcal{H}_k\otimes\mathcal{E}, \mathcal{H}_{\ell}\otimes\mathcal{F})$
is contractive \index{contractive multiplier} if the corresponding multiplication operator $M_{\Phi}: \mathcal{H}_k\otimes \mathcal{E}\to \mathcal{H}_{\ell}\otimes \mathcal{F}$ is contractive. The following characterization of contractive multipliers is standard (see e.g. \cite[Theorem 6.28]{PaulsenRagh}).

\begin{lemma} \label{multdef}
Let $\mathcal{H}_k, \mathcal{H}_{\ell}$ be  reproducing kernel Hilbert spaces on $X.$ Fix Hilbert spaces $\mathcal{E}, \mathcal{F}.$ A function $\Phi: X\to B(\mathcal{E}, \mathcal{F})$ is a contractive multiplier in $\Mult(\mathcal{H}_k\otimes\mathcal{E}, \mathcal{H}_{\ell}\otimes\mathcal{F})$ if and only if 
\[\ell(z, w)I_{\mathcal{F}}-k(z, w)\Phi(z)\Phi(w)^*\succeq 0 \hspace{0.2 cm} \text{ on } X\times X\]
if and only if the operator $T^*$ densely defined by 
\[T^*(\ell_w\otimes v)=k_w \otimes \Phi(w)^*v \]
extends to be a contractive operator from $\mathcal{H}_{\ell}\otimes\mathcal{F}$ to $\mathcal{H}_k\otimes\mathcal{E}$, in which case $T^*=M^*_{\Phi.}$
\end{lemma}

The following result will be used repeatedly in the sequel, usually without special mention. Its proof follows from the definition of positivity.
\begin{lemma} \label{resc}
Let $\mathcal{H}_k$ be a reproducing kernel Hilbert space on $X$. If  $\delta: X\to\mathbb{C}$ does not vanish, then  the function 
\[\widetilde{k}(z, w):=\delta(z)\overline{\delta(w)}k(z, w) \]
is a positive kernel on $X.$ Moreover, given any Hilbert spaces $\mathcal{E}, \mathcal{F},$ we have 
\[\Mult(\mathcal{H}_k\otimes \mathcal{E}, \mathcal{H}_k\otimes \mathcal{F})=\Mult(\mathcal{H}_{\widetilde{k}}\otimes \mathcal{E}, \mathcal{H}_{\widetilde{k}}\otimes \mathcal{F}) \]
isometrically.
\end{lemma}
We also record the following well-known maximum modulus principle for multipliers. For the readers convenience we provide a proof.
\begin{lemma}[Maximum Modulus for multipliers]\label{maxmod}
    Let $\mathcal{H}_k$ be a reproducing kernel Hilbert space on $X$ such that $k$ is non-vanishing. Given a Hilbert space $\mathcal{E}$, assume the row-valued multiplier $\Phi\in\Mult(\mathcal{H}_k\otimes\mathcal{E}, \mathcal{H}_k)$ is contractive. If there exists $z\in X$ such that $\|\Phi(z)\|=1,$ then $\Phi$ is constant.
     In particular, if $B$ is a positive kernel on $X$ and $(I-B(z,w))k(z,w)$ is a (positive) kernel, then either $|B(z,w)|<1$ for all $z,w\in X$
      of $B(z,w)=1$ for all $z,w\in X.$
\end{lemma}

\begin{proof}
 Observe,  for $u\in X,$ that the operator $\Phi(u)^*:\CC\to \mathcal{E}$ is naturally identified with an element of $\mathcal{E}$
 and if $v$ is also in $X,$ then 
 \[
  \langle k_u\otimes \Phi(u)^* , k_v\otimes \Phi(v)^* \rangle = \langle \Phi(u)^*,\Phi(v)^*\rangle \, k(u,v).
  \]
  Since  $M_\Phi^*$ is a contraction, the function $F:X\times X\to \CC$ defined by
\[
 F(u,v)=(1-\Phi(v)\Phi(u)^*)k(u,v) = (1-\langle \Phi(u)^*,\Phi(v)^*\rangle )k(u,v)
\]
 is a kernel by Lemma~\ref{multdef}. By assumption $F(z,z)=0.$ By positivity, $F(z,w)=0$ for all $w\in X.$ Thus  $\langle \Phi(z)^*,\Phi(w)^*\rangle =1.$
 Since also $\|\Phi(z)^*\|=1\ge \|\Phi(w)^*\|,$ we conclude $\Phi(w)=\Phi(z).$
 
  To prove the last statement, by standard reproducing kernel Hilbert space considerations, there exists an auxiliarly Hilbert space 
  $\mathscr{E}$ and a function $\Phi:X\to \mathscr{E}$ such that $B(z,w)=\Phi(z)\Phi(w)^*.$ By assumption $(I-\Phi(z)\Phi(w)^*)k(z,w)\succeq 0$
  and thus $\Phi$ is a contractive multiplier by Lemma~\ref{multdef}. Hence the last part of the lemma follows from the first part.
\end{proof}
Finally, we will also make repeated use of a simple formula for the reproducing kernels of zero-based invariant subspaces.
\begin{lemma}\label{l:compress:z}
Let $\mathcal{H}_k$ be a reproducing kernel Hilbert space on $X$ and fix a point $z\in X$ with $k(z, z)\neq 0.$ The reproducing kernel $k^z$ of the subspace $\{f\in\mathcal{H}_k : f(z)=0\}$ of $\mathcal{H}_k$ 
is given by 
\[k^z(w, v)=k(w, v)-\frac{k(w, z)k(z, v)}{k(z, z)}.\]
\end{lemma}

\subsection{The complete Pick property}\label{CPsubsec}
A reproducing kernel Hilbert space $\mathcal{H}_k$ on a set $X$ is said to be \df{irreducible} if $X$ cannot be partitioned into two non-empty disjoint sets $X_1, X_2$ so that $k(z_1, z_2)=0$ for all $z_1\in X_1$ and $z_2\in X_2.$ The kernel of an irreducible complete Pick space satisfies $k(z, w)\neq 0$ for all $z, w\in X,$ see \cite[Lemma 1.1]{Pickbook}.  By the Agler-McCarthy universality theorem \cite[Theorem 3.1]{Pickbook},  the space $\mathcal{H}_k$ is an irreducible complete Pick space if and only if there exist a function $\delta: X\to\mathbb{C}\setminus\{0\}$, an integer $1\le d\le \infty$ and a function $b: X\to\mathbb{B}_d$, where $\mathbb{B}_d$
denotes the open unit ball of $\mathbb{C}^d$, so that
\begin{equation}\label{univCP1}
k(z, w)=\frac{\delta(z)\overline{\delta(w)}}{1-\langle b(z), b(w)\rangle}, \hspace{0.2 cm} \text{ for all }z,w\in X.
\end{equation} 
Equivalently (see \cite{Pickbook}), $\mathcal{H}_k$ is an irreducible complete Pick space if and only if $k$ is non-vanishing and there exists $z\in X$ such that 
\begin{equation} \label{univCP2}
  \frac{k^z}{k}\succeq 0.  
\end{equation}
We note that, in view of Lemma~\ref{resc}, one may rescale any irreducible complete Pick kernel $k$ so that $\delta\equiv 1$ in \eqref{univCP1}. Also, given any $z_0\in X$, one may choose
\[\delta(z)=\frac{\sqrt{k(z_0, z_0)}}{k(z, z_0)},\]
so that $k(z, z_0)=1$  for all $z\in X.$ \par 
The study of general complete Pick kernels can be reduced to the study of irreducible ones. 

\begin{lemma}[{\cite[Theorem 1.1]{AmcCcompleteNP}}] \label{CPsplits}
Every complete Pick space can be decomposed as an orthogonal
direct sum of irreducible complete Pick spaces. Explicitly, if $s$ is a complete
Pick kernel on a nonempty set $X,$ then there is a unique partition $X=\cup X_i$
 such that for each $i$ the function $s|_{X_i\times X_i}$ is non-vanishing and 
  for each $i\ne j$ the function $s|_{X_i\times X_j}$ is identically $0.$
\end{lemma}

Complete Pick kernels actually satisfy a stronger positivity condition than \eqref{univCP2}. Given two kernels $\ell, k$ on a set $X,$ we will say that \df{$k$ is a factor of $\ell$} if there exists a kernel $g$ on $X$ such that $\ell=gk$.

\begin{lemma}[{\cite[Lemma 2.2]{InvPickfactor}}] \label{invfactors}
Let $\ell$ be a kernel with an irreducible complete Pick factor $s.$ If $M$ is a  $\Mult(\mathcal{H}_{\ell})$-invariant subspace of $\mathcal{H}_{\ell},$ and $\ell_M$ is the reproducing kernel for $M,$ then
\[\frac{\ell_M}{s}\succeq 0.\]
\end{lemma}

\subsection{Diagonal holomorphic kernels} \label{diaghomprelim}
 Given $z,w\in \CC^{\mathsf{g}},$ let \index{$z\overline{w}$}
\[ 
 z\overline{w}  =  (z_1\overline{w_1}, z_2\overline{w_2},  \cdots, z_{\mathsf{g}} \overline{w_{\mathsf{g}}}).
\]
Recall, from the introduction, that a kernel $\kk$ on a domain
 $\Omega\subseteq \CC^\vg$ containing $\mathbf{0}$ in its interior is 
 a \df{normalized diagonal holomorphic kernel} if its power series expansion at $\mathbf{0}$ has the form
\[
 \kk(z,w) =\sum_{a\in \NN^\vg}\kk_a (z\overline{w})^a
\]
 with $\kk_\mathbf{0}=1$ and $\kk_a>0$ for all $a.$ 
   \par
The diagonal holomorphic kernel $\kk$ has its canonical domain $\Omega_f$ as described in equation~\eqref{Omegakdef} and thus it gives rise to the reproducing kernel Hilbert spaces $\mathcal{H}_\kk(\Omega)$ and $\mathcal{H}_\kk(\Omega_f).$ Among other things, Lemma~\ref{restrbasic} below says these two spaces, as well as their spaces of multipliers, are canonically isometrically isomorphic.  Indeed, throughout this paper, we will often have to work with restrictions of holomorphic kernels to open subsets of their domain of convergence. Fortunately, no important  information is lost by doing so. 
\begin{lemma} \label{restrbasic}
Let $k$ be a holomorphic kernel on a domain $\Delta\subseteq\mathbb{C}^{\mathsf{g}}.$  If  $\Omega\subseteq\Delta$ is  a non-trivial open subset, then  the restriction  mapping 
\[\iota: f\mapsto f|_{\Omega} \]
maps $\mathcal{H}_k(\Delta)$ unitarily onto $\mathcal{H}_{k}(\Omega).$  Moreover, $\iota$ is an isometric isomorphism from $\Mult(\mathcal{H}_k(\Delta))$ onto $\Mult(\mathcal{H}_{k}(\Omega)).$
\end{lemma}
\begin{proof}
    This result follows from combining the Identity Principle with basic facts regarding restrictions of reproducing kernels; see \cite[Section 5.4]{PaulsenRagh}.
\end{proof}

The following proposition says more is true for a diagonal holomorphic kernel $\kk.$ Namely,  the domain  $\Omega_\kk$ is the maximum domain (containing $\mathbf{0}$) for $\kk.$  

\begin{proposition}
\label{prop:Pringsheim}
 Suppose $\kk$ is a holomorphic diagonal kernel and $\wtk$ is a holomorphic kernel  on a domain $\Omega\subseteq\CC^\vg.$ Thus,
it is assumed that $\mathbf{0}\in \Omega$ is open and connected and contains $\mathbf{0}$ and moreover that $\wtk:\Omega\times\Omega\to\CC$ is analytic 
 in the first coordinate  and conjugate analytic  in the second.  If $\kk(z,w)=\wtk(z,w)$ for 
 $z,w\in \Omega\cap \Omega_\kk,$ then $\Omega\subseteq \Omega_\kk.$
\end{proposition}

The proof will use the following variant of a theorem of Pringsheim. See for instance \cite[p. 15]{Boas}.
 For the record, we provide a proof, borrowing heavily from the arguments from \cite{Boas} in both the proof of Lemma~\ref{lem:Pringsheim}
  and Proposition~\ref{prop:Pringsheim} below. Recall $\mathbb{D}=\{z\in\mathbb{C}: |z|<1\}$ denotes the (open) unit disc in the complex plane.

\begin{lemma}
\label{lem:Pringsheim}
 Suppose 
\[
 f(\zeta)=\sum_a f_a \zeta^a
\]
 is a power series in $\vg$ complex variables and $f_a\ge 0$ for all $a.$ If the domain of convergence of $f$ contains the polydisc $\mathbb{D}^\vg$ and if there exists a neighborhood $N$ of $e=(1,1,\dots,1)$ and a holomorphic function $h:N\to \mathbb{C}$ such that $f(\zeta)=h(\zeta)$ for $\zeta\in N\cap \mathbb{D}^\vg,$ then there is a $t>1$ such that $\sum f_a (te)^{a}$ converges.
\end{lemma}

In the proofs of Lemma~\ref{lem:Pringsheim} and Proposition~\ref{prop:Pringsheim}, we let  $N_\eta(z)$ denote the usual Euclidean $\eta$-neighborhood of a point $z\in\CC^\vg,$ where $\eta>0$. As is customary, $a! = a_1! \, a_2! \, \cdots \, a_\vg !$ for $a=(a_1,\dots,a_\vg)\in\NN^\vg.$

\begin{proof}
There is an $1>\epsilon>0$ such that $N_{3\epsilon}((1-\frac{\epsilon}{\sqrt{\vg}})e) \subseteq N.$  For notational convenience, let $u=\frac{\epsilon}{\sqrt{\vg}}.$ Since $h$ is holomorphic on $N_{3\epsilon}((1-u)e),$  it has a power series expansion centered at $(1-u)e$ that converges in this neighborhood and in particular at  $(1+u)e.$ Letting $h_a$ denote the coefficients of this power series, it follows that 
\[
  h((1+u))e )=\sum_a  h_a \, (2u e)^a
\]
and this series converges. On the other hand, 
\[
  h_a = \frac{1}{a!}\kk_a((1-u)e)  = \sum_{b\ge a} \frac{b!}{a!\, (b-a)!} \kk_b  ((1-u)e)^{b-a}
 =  \sum_{b\ge a} \binom{b}{a}\,  \kk_b ((1-u)e)^{b-a},
\]
 where   $\kk^{(a)}$ is the corresponding partial derivative of the function $\kk.$  Hence, the series
\[
 \sum_a \bigg[\sum_{b\ge a} \binom{b}{a}\,  \kk_b ((1-u)e)^{b-a} \bigg] (2ue)^a
\]
 converges. Since everything in sight is non-negative, this sum converges after rearranging; e.g., using the binomial theorem,
\[
 \sum_b \sum_{a\le b} \binom{b}{a}\,  \kk_b ((1-u)e)^{b-a}\, (2u e)^a
 =\sum_b \kk_b ((1+u)e)^b
\]
 converges. Hence the conclusion of the lemma holds with $t=1+u.$
\end{proof}

Another ingredient in the proof of Proposition~\ref{prop:Pringsheim} is Abel's Lemma. 
 For a reference, see \cite[Proposition~2.3.4]{MR1871333}.

\begin{lemma}[Abel's Lemma]
\label{Abellemma}
  Given a  power series $\sum_a c_a x^a,$  if $0\ne \mathbf{y}\in \CC^\vg$  and there is a $C>0$ such that $|c_a y^a| \le C$ for all $a,$ then the series $\sum c_a x^a$ converges absolutely and uniformly on compact subsets of $\{(\zeta_1y_1,\dots,\zeta_\vg y_\vg): \zeta_j\in\mathbb{D}\}.$
\end{lemma}

\begin{proof}[Proof of Proposition~\ref{prop:Pringsheim}]
 Arguing by contradiction, suppose $\Omega\not \subseteq \Omega_\kk.$ In this case, by connectedness of $\Omega$ and since $\mathbf{0}\in \Omega\cap \Omega_\kk,$   there exists a point $p\in \Omega\setminus \Omega_\kk$ that lies in the boundary of $\Omega_\kk.$\footnote{If no such $p$ exists, then
$\Omega\setminus \Omega_f  = \Omega\, \cap \, [\operatorname{b}\Omega_f \cup \operatorname{e} \Omega_f] = [\Omega \cap \operatorname{e}\Omega_f],$ where $\operatorname{b}\Omega_f$ and $\operatorname{e}\Omega_f$ are the boundary and exterior of $\Omega_f.$}

 Since $ \mathbf{0}$ is in the interior of $\Omega_\kk$ (by a standing assumption),
  $p\ne \mathbf{0}.$ However, it is still possible that some of the coordinates of $p$ are $0.$  By renaming, we may assume $p=(p_1,\dots,p_a,0,\dots,0)\in \CC^\vg$ where $p_1,\dots,p_a\ne 0.$  Let $\eta=\min\{p_1,\dots,p_a\}>0.$   In what follows we use the max norm (metric)  $\|x\|_\infty =\max\{|x_j|:j\}$  on $\CC^\vg.$ Since $\Omega$ is open, there exists an $\epsilon<\frac{\eta}{2}$ such that $N^\infty_\epsilon(p)=\{x: \|x-p\|_\infty <\epsilon\}\subseteq \Omega.$  It follows that there is a point $q=(q_1,\dots,q_\vg) \in N^\infty_\epsilon(p)\cap \Omega_\kk$ such that $q_{a+1},\dots,q_\vg$ are all non-zero.  Note, at the same time, $|q_j|\ge \frac{\eta}{2}$ for $1\le j\le a$, since $|q_j-p_j|<\frac{\eta}{2}$ and $|p_j|\ge \eta$ for $1\le j\le a.$   Consider the point $z=(p_1,\dots,p_a,q_{a+1},\dots,q_\vg).$ Note $z\in \Omega,$ since  $\|z-p\|_\infty =\max\{|q_j|: a+1\le j\le \vg\} \le \|q-p\|_\infty<\epsilon.$ By 
 construction all the coordinates $z_1,\dots,z_\vg$ of $z$ are non-zero.  Arguing by  contradiction, suppose $z\in\Omega_\kk.$ In this case there is a $t>1$ such that $tz\in \Omega_\kk.$ By Lemma~\ref{Abellemma} applied to the power series $\sum \kk_a x^{2a},$ the set $\Omega_\kk$ contains the polydisc $D=\{t(\zeta_1z_1,\dots,\zeta_\vg z_\vg):  \zeta_j\in \mathbb{D}\}.$ Since none of the coordinates of $z$ are $0,$ this polydisc is an open set.  In particular, choosing $\zeta_j=\frac{1}{t}$ for $j=1,\dots,a$ and $0$ otherwise,  we find $p\in D\subseteq\Omega_\kk,$ which is a contradiction as $\Omega_\kk$ is open and $p$ is in its boundary.  
\par
At this point we have $q\in N_\epsilon(p)\cap \Omega_\kk$ and $z\in N_\epsilon(p)\setminus \Omega_\kk.$ 
 Hence $\{tq+(1-t)z: 0\le t\le 1\}\subseteq N_\epsilon(p)\subseteq \Omega$ and there is a $0< t\le 1$ such that $y=tq+(1-t)z$ is in the boundary of $\Omega_\kk.$
 Note that the coordinates of $y$ are given by $y_j=tq_j+(1-t)p_j$ for $1\le j\le a$ and $y_j=tq_j\ne0$ for $j>a.$  Suppose $y_j=0$ for some $j\le a.$
 Thus, $q_j = -\frac{1-t}{t}p_j$ and therefore, 
\[
 \frac{\eta}{2} \ge  \epsilon > |p_j-q_j| = \frac{|p_j|}{t} \ge \eta,
\]
a contradiction. Hence all the coordinates of $y_j$ are non-zero. 
\par

 Another application of Lemma~\ref{Abellemma} to $\sum \kk_a x^{2a}$ says that the polydisc $D=\{(\zeta_1 y_1,\dots ,\ \zeta_\vg y_\vg): \zeta_j\in \mathbb{D}\}$ lies in $\Omega_\kk.$  Let $g$ denote the power series, for $\zeta\in \CC^\vg,$
\[
 g(\zeta) =\sum_{a\in \NN^\vg} \kk_a|y_a|^{2a} \zeta^{2a}.
\]
In particular, as power series, 
\[
 g(\zeta)=\kk(\zeta y,\overline{\zeta}y)  = \sum \kk_a ((\zeta \, y)_a)^a \overline{((\overline{\zeta}\, y)_z)^a}.
\]
It follows that $\mathbb{D}^\vg$ is contained in the domain of convergence of $g$ and moreover, as functions, $g(\zeta)=\kk(\zeta y,\overline{\zeta}y)$  for $\zeta\in\mathbb{D}^\vg.$ Define $h$ by $h(\zeta)=\wtk(\zeta y,\overline{\zeta}y).$  Since $y\in \Omega,$ the function $h$ is defined in a neighborhod $N$ of $e=(1,\dots,1)$. Moreover, $g(\zeta)=h(\zeta)$ for $\zeta\in \mathbb{D}^\vg \cap N$ and $h$ is analytic. By Lemma~\ref{lem:Pringsheim}, it follows that there exists a $t>1$ such that  $\sum g_a (te)^{a}$ converges.  Thus $\sum_a \kk_a |ty|^{2a}$ converges. By Lemma~\ref{Abellemma}, $\Omega_{\kk}$ contains the polydisc $\{(\zeta_1 ty_1,\dots,\zeta_\vg t y_\vg): \zeta_j\in \mathbb{D}\}$ and, by choosing $\zeta_j=\frac{1}{t},$  we have arrived at the contradiction $y\in\Omega_\kk.$
\end{proof}

The following lemma is used in the proof of Proposition~\ref{l:tlg}.

\begin{lemma}
 \label{Omegafsup1}
    Let  $f=\sum\limits_{|a|>0} f_a z^a\overline{w}^a.$ If $f_a\ge 0$ for all $|a|>0$ and $\Omega_f\ne \emptyset,$ then 
 \[
   \Omega_f^1 =\{x\in \Omega_f: \sum f_a|x|^{2a} <1\}
 \]
is non-empty, open and connected; that is, a domain in $\CC^\vg.$ 
\end{lemma}

\begin{proof}
  Let $g(\zeta)=\sum_{|a|>0} f_a \zeta^{2a}.$ The domain of convergence of the power series $g$ is,
   by definition, $\Omega_f.$ By assumption, $\Omega_f$ contains a neighborhood of the
    origin. The domain of convergence of a power series 
    is open and star-like with respect to the origin. Since $g$ defines a holomorphic function on $\Omega_f,$ the set $\Omega_f^1$
 is open and contains $\mathbf{0}.$ Moreover, $\Omega_f^1$ is the intersection
  of two sets that are star-like with respect to $\mathbf{0},$ namely $\Omega_f$ (\cite[Proposition 2.3.15]{MR1871333})
   and $\{x\in\CC^\vg: \sum f_a|x|^{2a} <1\}.$ Hence, $\Omega_f^1$ is 
   star-like with respect to $\mathbf{0}$ too. In particular, it is connected and thus a domain.
\end{proof}

Informed by Proposition~\ref{prop:Pringsheim}, we conclude this subsection with the following observation.
A normalized holomorphic diagonal kernel  $\kk$ and its  corresponding Hilbert space $\mathcal{H}_\kk$ can also
be understood as follows. A  point  $z\in\Omega_\kk$ gives rise to a vector $E_z=(\sqrt{\kk_a}\overline{z}^a)$ in the Hilbert space $\ell^2(\NN^{\mathsf{g}}).$ Thus we obtain a mapping  $E:\Omega_\kk \to \ell^2(\NN^\vg)$  and,  if $z,w\in \Omega_\kk,$ then 
\begin{equation} \label{anotherlab}
 \langle E_w,E_z\rangle = \sum_{a\in\NN^{\mathsf{g}}} \kk_a z^a \overline{w^a}
 =\kk(z,w).
\end{equation} 
From standard facts about reproducing kernels, the monomials $z^a$ form
an orthogonal basis for $\mathcal{H}_f$  and
\[
 \langle z^a,z^a \rangle =\|z^a\|^2 = \frac{1}{k_a}.
\]

\subsection{The complete Carath\'eodory property} \label{CCsubs}
 As we saw in Section~\ref{INTROSEC}, the classical Carath\'eodory problem asks, given $c_0,\dots,c_n\in \CC,$ does there exist
 an analytic function $g:\DD\to\CC$ such that the function
\[
 f=\sum_{j=0}^n \overline{c}_j z^j + z^{n+1}g
\]
 satisfies $|f(z)|\le 1$ for all $z\in \DD?$  From an  operator-theoretic viewpoint,
 there is a natural necessary condition that, when combined with Parrott's Lemma,
 can be seen to be sufficient. A function $f:\DD\to\DD$ determines an operator $M_f:H^2(\DD)\to H^2(\DD)$
 defined by $M_fh=fh$ for $h\in H^2(\DD)$ and 
 $\|M_f\|\le 1.$  The subspace $M_n$ of $H^2(\DD)$ spanned by the orthonormal
 set  $\{1,z,\dots,z^n\}$ is invariant
 for $M_f^*$ and moreover, the matrix representation of $M_f^*$ with respect to this basis is
\[
 C_n =\begin{pmatrix} c_0 & c_1 & c_2 & \dots & c_{n-1} & c_n \\ 0 & c_0 & c_1 & \dots & c_{n-2} & c_{n-1}
 \\ 0 & 0 & c_1 &\dots & c_{n-3} & c_{n-2} \\ \vdots &  \vdots & \vdots & \ddots & \vdots &\vdots
 \\ 0 & 0&0& \dots &0&c_0 \end{pmatrix}.
\]
 Thus, $\|C_n\|\le 1$ is a necessary condition for a solution to the Carathéodory  problem.
 While not the original proof, 
 sufficiency can be proved using the following version of the Parrott Lemma (see, for instance, \cite[Lemma~B1]{Pickbook}).

\begin{lemma}
\label{l:parrotf}
 Suppose $p,q$ and $m,n$  are positive integers such that $p+q=m+n,$
\begin{enumerate}
 \item[(a)] $A\in M_{p,m};$
 \item[(b)] $C\in M_{q,m};$
 \item[(c)] $D\in M_{q,n};$
\end{enumerate}
such that
\[
 \big \| \begin{pmatrix} A \\ C \end{pmatrix}\big \|, \ \  \big \| \begin{pmatrix} C& D \end{pmatrix} \big \| \le 1,
\]
 then there exists a matrix $B\in M_{p,n}$ such that
\[
 \big \| \begin{pmatrix} A& B \\ C  & D\end{pmatrix} \big \| \le 1.
\]
\end{lemma}

With $c_{n+1}$ to be determined, consider the following partition of the matrix $C_{n+1},$
\[
 C_{n+1} = \left ( \begin{array}{cccccc|c}
       c_0 & c_1 & c_2 & \dots & c_{n-1} & c_n & c_{n+1}  \\
   \hline
        0 & c_0 & c_1 & \dots & c_{n-2} & c_{n-1} & c_n
 \\ 0 & 0 & c_1 &\dots & c_{n-3} & c_{n-2} & c_{n-1}
  \\ \vdots &  \vdots & \vdots & \ddots & \vdots &\vdots  & \vdots 
 \\ 0 & 0&0& \dots &0&c_0 & c_1 \\ 0&0&0 &\dots 0 &0  & 0 &  c_0
    \end{array}\right ).
\] 
 Since $\|C_n\|\le 1,$ 
\[
\big \|  \begin{pmatrix} C_n \\ 0 \end{pmatrix} \big \|, \ \ \big \| 
 \begin{pmatrix} 0 & C_n \end{pmatrix} \big \| \, \le 1.
\]
 Thus, by Parrott's Lemma there is a choice of $c_{n+1}$ such that
 $\|C_{n+1}\|\le 1.$  An induction argument now produces $c_m$ such
 that $\|C_m\|\le 1,$ for all $m$, which in turn implies, with $f=\sum_{j=0}^{m} \overline{c}_jz^j,$
 that the restriction of $M_f^*$ to the subspaces $M_{m}$ is a contraction. Thus, we obtain an $f=\sum_{j=0}^\infty \overline{c_j}z^j$ such that   $M_f^*$ defines an
 operator on $H^2(\DD)$ of norm at most one. It is also not hard to see that $M_f^* \sz_\lambda = \overline{f(\lambda)}\,  \sz_\lambda,$ where $\sz$ is Szego's kernel
  and $\lambda\in \DD.$ 
 Hence $\|f(\lambda)\|\le 1$ for all $\lambda\in \DD$, as desired.
 Note that the argument goes through unchanged if the $c_j$'s are 
 replaced by $J\times J$ matrices for any positive integer $J.$
 
 We highlight the fact that this argument proceeds in two parts. 
 There is the {\it matrix completion problem}: Starting with the contraction  $C_n,$
 find $c_m$ for $m>n$ such that $C_m$ is a contraction for each $m.$ And there
 is the {\it function-theoretic operator theory interpretation}: 
 the assumption that each $C_m$ is a contraction
 implies that the power series $\sum \overline{c}_j z^j$ defines a function 
 $f:\DD\to \DD.$ Equivalently, $f:\DD\to \CC$ and $M_f^*:H^2(\DD)\to H^2(\DD)$
 is a contraction. \par 
 In the remainder of this section, we consider the complete Carath\'eodory property for a pair of kernels $(k,\ell),$
 where $k$ and $\ell$ are (normalized) diagonal holomorphic. First, we will restate the definition given in the introduction. 
 Given $a, b\in\mathbb{N}^{\vg}, $ we will write $a\le b$ to mean $a_i\le b_i$ for $i\in\{1, \dots, \vg\}.$
  \begin{definition} \label{coinvdef}
 A non-empty set $F\subseteq \mathbb{N}^\vg$  is \df{co-invariant} 
 if  $\{a\in\NN^\vg: a\le b\}\subseteq F$ for each $b\in F.$ 
 \qed
\end{definition}
Thus, if $F$ is  co-invariant, then $F$ contains $(0, \dots, 0)=\bzero.$ 

\begin{definition}\label{CCdef} A pair $(k, \ell)$  of (normalized) diagonal holomorphic kernels 
is
a \df{complete Carath\'eodory pair} if, for any positive integer $J$, any finite co-invariant  set of indices $\varnothing \ne F\subseteq\mathbb{N}^{\mathsf{g}}$ and any collection of matrix coefficients $\{c_a : a\in F\}\subseteq\mathbb{M}_J$, the positivity of 
 the block upper-triangular matrix $C$ indexed by $F\times F$ with block $J\times J$ entries 
 \begin{equation}\label{data1}
     C_{a, b}=\begin{cases}
     c_{b-a}\sqrt{\cfrac{k_a}{\ell_b}}, \hspace{0.5 cm} b\ge a, \\
      \mathbf{0},  \hspace{1.05 cm} \text{otherwise,}
 \end{cases}  \end{equation}
is equivalent to the existence of a collection
$\{c_a : a\in \mathbb{N}^{\mathsf{g}}\setminus F\}\subseteq\mathbb{M}_J$
such that the function
\begin{equation}\label{CCinterpolantag}
 \Phi(z)=\sum_{a\in\mathbb{N}^{\mathsf{g}}}z^a\otimes c_a^*   
\end{equation}
is a contractive multiplier from $\mathcal{H}_k\otimes\mathbb{C}^J$ to $\mathcal{H}_{\ell}\otimes\mathbb{C}^J.$   In the special case $k=\ell,$ the kernel $k$ is known as a \df{complete Carath\'eodory kernel}.
\qed
\end{definition}
Just as in the classical Carath\'eodory  problem above, it is possible to reformulate Definition~\ref{CCdef} so that it does not involve  function theory.

\begin{proposition}
\label{d:ccp}
A pair of diagonal holomorphic  kernels  $(k, \ell)$ is a complete  Carath\'eodory pair if and only if, 
  if for each  $J\in\mathbb{N},$ each finite co-invariant $\varnothing \ne F\subseteq\mathbb{N}^{\vg}$ and each collection $\{c_a  : a\in F\}\subseteq \mathbb{M}_J$ such that the block upper-triangular matrix $C$ indexed by $F\times F$ with block $J\times J$ entries as in \eqref{data1}
is a contraction, there exists a collection $\{c_a  :  a\in\mathbb{N}^{\vg}\setminus F\}$ such that the (infinite) block matrix $\mathscr{C}$ indexed by $\mathbb{N}^{\vg}\times\mathbb{N}^{\vg}$ with $\mathscr{C}_{a, b}$ entries
\begin{equation}\label{fullext}
    \mathscr{C}_{a, b}=\begin{cases}
     c_{b-a}\sqrt{\cfrac{k_a}{\ell_b}}, \hspace{0.5 cm} b\ge a, \\
      \mathbf{0},  \hspace{1.05 cm} \text{otherwise,}
 \end{cases} \end{equation}
 is a contraction.
\end{proposition}
\begin{remark}
\label{r:ka<=la}
 Assume $k, \ell$ are normalized. By considering $F=\{\mathbf{0}\}$ and choosing $\{c_{\mathbf{0}}=1\}$ we obtain the $1\times 1$
 matrix $\Big(c_{\mathbf{0}} \sqrt{\frac{k_{\mathbf{0}}}{\ell_{\mathbf{0}}}}\Big)=1,$ which is a contraction. Hence,
 assuming that $(k,\ell)$ is a CC pair, there exists
 $c_a$ for $a>0$ such that the matrix $\mathscr{C}$ in \eqref{fullext}
 is a contraction. In particular, its diagonal entries satisfy
 $1 \ge \Big|c_{\mathbf{0}} \sqrt{\frac{k_a}{\ell_a}}\Big| = \sqrt{\frac{k_a}{\ell_a}}$ and thus
 $k_a\le \ell_a,$ an inequality that follows from Theorem~\ref{Shimnec}. 
\qed
\end{remark}
The remainder of this subsection will be devoted to a proof sketch of Proposition~\ref{d:ccp}. We begin with a function-theoretic interpretation of the contractivity of \eqref{data1}.  Let 
\[e_a(z)=\sqrt{k_a}z^a, \hspace{0.2 cm} f_a=\sqrt{\ell_a} z^a, \hspace{0.2 cm} a\in\mathbb{N}^{\vg},\]
so that $\{e_a\}$ and $\{f_a\}$ form orthonormal bases for $\mathcal{H}_k$ and $\mathcal{H}_{\ell}$, respectively. Fix $J\in\mathbb{N}$ and a  co-invariant $F\subseteq\mathbb{N}^{\vg}$ and let $\mathcal{H}_{k,F}$ and $\mathcal{H}_{\ell,F}$ denote the subspaces of $\mathcal{H}_k$ and $\mathcal{H}_\ell$  spanned by $\{e_a : a\in F\}$ and $\{f_a : a\in F\}$, respectively. Further, assume the collection $\{c_a : a\in\mathbb{N}^{\vg}\}\subseteq \mathbb{M}_J$ is such that the function
\[
\Phi(z)=\sum_{a\in\mathbb{N}^{\vg}}z^a\otimes c_a^*
\]
is a multiplier $\mathcal{H}_k \otimes\mathbb{C}^J \to\mathcal{H}_{\ell}\otimes\mathbb{C}^J$. A short computation reveals that 
\begin{equation} \label{adjexp}
 M^*_{\Phi}(f_a\otimes h)=\sum_{u\le a}\sqrt{\frac{k_u}{\ell_a}}e_u\otimes c_{a-u}h ,\hspace{0.2 cm} \text{ for all } a\in\mathbb{N}^{\vg} \text{ and }h\in\mathbb{C}^J.   
\end{equation} 
Thus, $M_{\Phi}^*$ maps $\mathcal{H}_{\ell,F}$ into $\mathcal{H}_{k,F}.$ Moreover, the matrix representation of  the restriction of $M_\Phi^*$ to $\mathcal{H}_{\ell,F}$ is given by the matrix of equation~\eqref{data1}.  Since $M_\Phi^*$ is, by assumption,
 a contraction,  so is $C.$  
 \par 
Next, we will show that if the $\mathbb{N}^{\vg}\times\mathbb{N}^{\vg} $ matrix $\mathscr{C}$ given by \eqref{fullext} 
is a contraction, then the function 
\begin{equation} \label{formalPhidef}
 \Phi(z)=\sum_{a\in\mathbb{N}^{\vg}}z^a\otimes c_a^*    
\end{equation}
is a contractive multiplier $\mathcal{H}_k \otimes\mathbb{C}^J \to\mathcal{H}_{\ell}\otimes\mathbb{C}^J$ (the converse follows immediately from our previous argument). First, observe that, since $\mathscr{C}$ is a contraction, the $\ell^2$-norm of the first row ($a=\mathbf{0}$) is bounded. Thus,
\begin{equation} \label{Phidef}
    \sum_{b\in\mathbb{N}^{\vg}}\frac{\|c_b\|^2}{\ell_b}<\infty.
\end{equation}
If $k, \ell$ are defined on $\Omega,$ then,  by Proposition~\ref{prop:Pringsheim},  $\Omega\subseteq \Omega_{\ell}$, where $\Omega_{\ell}$ is defined as in \eqref{Omegakdef}. Thus, for every $z\in \Omega$, \eqref{Phidef}  and the Cauchy-Schwarz inequality gives us
\[\|\Phi(z)\|^2\le \bigg(\sum_{a\in\mathbb{N}^{\vg}}\|c_a\|\cdot|z|^a\bigg)^2\le \bigg(\sum_{a\in\mathbb{N}^{\vg}}\ell_a|z|^{2a}\bigg)\bigg(\sum_{a\in\mathbb{N}^{\vg}}\frac{\|c_a\|^2}{\ell_a}\bigg)<\infty.  \]
So, $\Phi$ does actually define a function on $\Omega$. Now, arguing as earlier in this proof, we see that $\|\mathscr{C}\|\le 1$ is equivalent to 
\[
f_a\otimes h\mapsto \sum_{u\le a}\sqrt{\frac{k_u}{\ell_a}}e_u\otimes c_{a-u}h 
\]
defining a contractive operator from $\mathcal{H}_{\ell}\otimes\mathbb{C}^J$ to $\mathcal{H}_{k}\otimes\mathbb{C}^J$. Denote that operator by $T$. It is then not hard to verify that 
\[T(\ell_w\otimes v)=k_w\otimes \Phi(w)^*v, \]
for all $w\in\Omega$ and $v\in \mathbb{C}^J.$ In view of Lemma~\ref{multdef}, $M^*_{\Phi}=T$ is a contraction and $\Phi$ is a contractive multiplier, as desired.

\normalsize

\section{Shimorin and Strong Shimorin Certificates} \label{ShimnstrShim}
\normalsize
In this section, we will introduce the notion of a Shimorin certificate  for a pair of abstract (not necessarily holomorphic) kernels. Our motivation behind Definition~\ref{mostgeneral} was to come up with the weakest set of conditions that guarantee the CP property for a general pair $(k, \ell)$.  Further, even though having a Shimorin certificate will turn out to be weaker than having a strong Shimorin certificate, we will see that the two properties coincide in many interesting cases. For convenience, we will make the assumption that all pairs $(k, \ell) $ consist of kernels that are non-vanishing along the diagonal.

\subsection{Shimorin certificates}
In this subsection we define a Shimorin certificate and collect some consequences. Recall that 
\[k^z(w, v)=k(w, v)-\cfrac{k(w, z)k(z, v)}{k(z, z)}.\]

\begin{definition} \label{mostgeneral}
Assume $k, \ell$ are kernels on the non-empty set $X.$ 
A family of kernels $\{p[z]\}_{z\in X}$ on $X$ is a \df{Shimorin certificate} for the pair $(k, \ell)$ if,
 for each $z\in X,$ 
\begin{equation}\label{Shimcertconds}
 \ell^z\succeq p[z]\ell \hspace*{0.3 cm} \text{ and } \hspace*{0.3 cm}  k^z\preceq p[z]k.
\end{equation}
\end{definition}
\begin{remark} A strong Shimorin certificate is a CP kernel $s$ on $X$, while a Shimorin certificate is a family of kernels $\{p[z]\}$ all defined on $X$ and indexed by $X$.
\end{remark}

\begin{proposition}\label{newcerttt}

Let $k, \ell$ be kernels on $X.$  If $(k, \ell)$ has a strong Shimorin certificate, then it also has a Shimorin certificate. 
\end{proposition}
\begin{proof}
    Suppose $(k, \ell)$ has a strong Shimorin certificate $s.$ Thus, there exist kernels $h$ and $g$ on $X$ such  that 
\[
 k=(1-h)s,  \hspace*{0.7 cm} \ell=gs.
\]
By Lemma~\ref{CPsplits}, there exists a decomposition $X=\cup_{i\in I}X_i$ such that $s_i:=s|_{X_i\times X_i}$ never vanishes, 
 $k(z,w)=\ell(z,w)=s(z,w)=0$ for $i\ne j$ and $z\in X_i$ and $w\in X_j,$ and,  for each $i,$ either $k|_{X_i\times X_i}$  is non-vanishing or identically $0.$ 
 Now, set $\ell_i=\ell|_{X_i\times X_i}$ and $g_i=g|_{X_i\times X_i}$ and observe that, for any $i\in I$ and $z\in X_i$, we have 
  \begin{equation}\label{altcert}
      \ell^z_i\succeq \frac{s_i^z}{s_i}\ell_i  \hspace*{0.4 cm}  \text{ and }     \hspace*{0.4 cm}  k^z_i\preceq \frac{s_i^z}{s_i} k_i.
  \end{equation}
Indeed, after a little computation the first inequality is seen to be equivalent to
\[
 \frac{g_i(w,z){g_i(z, v)}}{g_i(z, z)}\preceq g_i(w, v), \hspace{0.3 cm} w, v\in X_i,
\]
which holds because $g$ is positive. The second inequality in \eqref{altcert} is trivial if $1-h \equiv 0.$  If  $1-h$  is non-vanishing on $X_i\times X_i,$ then, setting $t_i=(1-h |_{X_i\times X_i})^{-1}=s_i/k_i,$ the second inequality becomes
\[
 \frac{t_i^{z}}{t_i}\succeq 0,
\]
which holds for all $z\in X_i$ because $t$ is a CP kernel. We now define a Shimorin certificate for $(k, \ell)$. Given $z\in X_i$, we set
\[
 p[z](w, v)=\begin{cases}
 \cfrac{s_i^z(w, v)}{s_i(w, v)}, \hspace{0.38 cm} \text{ if } w, v\in X_i,     \\
 1, \hspace{1.45 cm} \text{ if } w, v\in X_j \text{ with } j\neq i, \\
 0,\hspace{1.48 cm}   \text{ otherwise.}
\end{cases}  
 \]
Combining \eqref{altcert} with the observation that $k^z|_{X_i\times X_i}=k|_{X_i\times X_i}$ and $\ell^z|_{X_i\times X_i}=\ell|_{X_i\times X_i}$ whenever $z\in X_j$ with $i\neq j$, we conclude that $\{p[z]\}_{z\in X}$ is a Shimorin certificate for $(k, \ell)$ and the proof is complete.
\end{proof}

The converse of Proposition \ref{newcerttt} does not hold, as the following example shows. 
\begin{example} \label{anykernel}
Given an arbitrary kernel $k$, we can always find a second kernel $\ell$ with the property that $(k, \ell)$ has a Shimorin certificate. Indeed, choose  $\ell$ to be a kernel on $X$ with $\ell(z, w)=0$ whenever $z\neq w.$ Now, for any $t\in X,$ set
\[p[t](z, w)=\begin{cases}
    0 \hspace*{0.6 cm} \text{ if } z=t \text{ or } w=t, \\
    1, \hspace*{0.65 cm} \text{ otherwise}.
\end{cases} \]
Since $\ell^t =p[t]\ell$ for all $t$ and 
\[
 p[t](z,w)k(z,w)-k^t(z,w)= p[t](z,w) \frac{k(z,t)\, k(t,w)}{k(t,t)} \succeq 0,
\]
 $\{p[t]\}_{t\in X}$ is a Shimorin certificate for $(k, \ell)$.  

 It is now easy to construct an example of a pair $(k,\ell)$ with a Shimorin certificate, but without a strong Shimorin certificate. Let $X=\{1,2,3\}$ and let 
\[
 k=\begin{pmatrix} 1 & 1 & 0\\1&2&1\\0&1 & 2\end{pmatrix},
\]
so that $k(a,b)$ is the $(a,b)$ entry of this matrix. Note $k$ is positive definite. If $B$ and $s$ are positive kernels and $k=(1-B)s,$ then either $B(1,3)=1$ or $s(1,3)=0.$ If $s$ is a complete Pick kernel and $s(1, 3)=0$, then, by Lemma~\ref{CPsplits}, either $s(2,3)=0$ or $s(1,2)=0.$ Either way one obtains a contradiction, since $k(1,2)\ne 0\ne k(2,3).$ If $B(1,3)=1,$ then as $B$ is positive and $B(1,1),B(3,3)\le 1$ (since $k(1,1),k(3,3),s(1,1),s(3,3)\ge0$), we obtain $B(1, 1)=1$ and so $k(1,1)=0,$ a contradiction. Thus, choosing an $\ell$ such that $(k,\ell)$ has a Shimorin certificate, which is possible from the discussion above, gives the desired example.
\end{example}

We close this subsection by recording  some basic restrictions that are imposed on a pair of a kernels by the existence of a Shimorin certificate.
\newcounter{zerosl} 
\newcounter{zerosll}
\newcounter{zeroslll}
\begin{lemma}
 \label{l:zeroslemma}
  Suppose $k,\ell,p$  are kernels on $X$ and $v,w,z\in X$ are distinct.
    If  $\ell^z\succeq p\ell,$ then 
\begin{enumerate}[(i)]
 \item   $p(w,w)\le 1$ for all $w\in X$ and thus, by positivity, $|p(w,v)|\le 1$
 for all $w,v\in X;$
 \item \label{i:pz:2} if $p(w,w)=1,$ then $\ell(z,w)=0;$
 \item \label{i:pz:3}  if $p(w,w)=1,$ 
 then either $\ell(w,v)=0$ or $p(v,v)=1=p(w,v).$
\setcounter{zerosl}{\value{enumi}}
\end{enumerate}
 
 If $k^z\preceq pk,$ then 
\begin{enumerate}[(i)]
\setcounter{enumi}{\value{zerosl}}
 \item \label{i:pz:4} if  $k(z,w)=0,$ then $p(w,w)=1;$
 \item \label{i:pz:5} if  $k(z,w)=0,$  
 then either $k(w,v)=0$ or $p(v,v)=1=p(w,v).$
\setcounter{zerosll}{\value{enumi}}
\end{enumerate}
If  $\ell^z\succeq p\ell$ and $k^z\preceq pk$ and if $k(z,w)=0$,  then 
\begin{enumerate}[(i)]
 \setcounter{enumi}{\value{zerosll}}
 \item  \label{i:pz:6} $\ell(z,w)=0.$
 \setcounter{zeroslll}{\value{enumi}}
 \end{enumerate}
   If $(k,\ell)$ has a Shimorin certificate and $k(z,w)=0,$  then 
 \begin{enumerate}[(i)]
 \setcounter{enumi}{\value{zeroslll}}
 \item \label{i:pz:7} at least one of the following holds:
\begin{enumerate}[(a)]
    \item \label{i:zerosl:a}  $\ell(z, v)=\ell(w, v)=0$;
    \item \label{i:zerosl:b} either $k(z, v)=0$ or $k(w, v)=0$.
\end{enumerate}
\end{enumerate}
\end{lemma}
\begin{proof}
 By definition, $\ell^z\succeq p\ell$ means 
\[
 X\times X\ni (x,y) \mapsto  \ell(x,y) -\frac{\ell(x,z) \ell(z,y)}{\ell(z,z)} - p(x,y)\ell(x,y)
\]
 is a kernel (PsD). In particular, $\ell(w,w)(1-p(w,w)) \ge \frac{|\ell(z,w)|^2}{\ell(z,z)}\ge 0.$
 Thus, $p(w,w)\le 1.$  Moreover, if $p(w,w)=1,$ then necessarily, $\ell(z,w)=0.$  
  Assuming $p(w,w)=1,$ from what has already been proved $\ell(z,w)=0$ and thus
 the positivity condition $\ell^z\succeq p\ell$ gives
\[\begin{bmatrix}
\ell^z(w, w) & \ell^z(w, v) \\
\ell^z(v, w) & \ell^z(v, v)
\end{bmatrix}=\begin{bmatrix}
\ell(w, w) & \ell(w, v) \\
\ell(v, w) & \ell^z(v, v)
\end{bmatrix}
\succeq \begin{bmatrix}
 \ell(w, w) & p(w, v)\ell(w, v)   \\
 p(v, w) \ell(v, w) & p(v, v)\ell(v, v)
\end{bmatrix}.\]
Hence
\[ \begin{bmatrix}
    0 & \big(-p(w, v)+1\big)\ell(w, v) \\
\big(-p(v, w)+1\big)\ell(v, w)  & -p(v, v)\ell(v, v)+\ell^z(v, v)
\end{bmatrix}      \succeq 0.\]
Thus, either $\ell(w,v)=0$ or $p(w,v)=1.$ In the later case, by positivity, $p(v,v)=1,$  
concluding the proof of the first three items of the lemma. Items~\eqref{i:pz:4} and \eqref{i:pz:5} are proved in an entirely analogous manner.

Assume $\ell^z \succeq p\ell$ and $k^z \preceq pk.$ If $k(z,w)=0,$ then, by item~\eqref{i:pz:4}, $p(w,w)=1$ and thus, by item~\eqref{i:pz:2},  $\ell(z,w)=0,$ proving item~\eqref{i:pz:6}.

Now suppose $(k,\ell)$ has a Shimorin certificate. In particular, there exist kernels
 $p[z]$ and $p[w]$ such that $\ell^z \succeq p[z]\, \ell$ and $k^z\preceq p[z]\, k$
 as well as $\ell^w \succeq p[w]\, \ell$ and $k^w \preceq p[w]\, k.$ 
If $k(z,v)\ne 0 \ne k(w,v),$ then, by  item~\eqref{i:pz:5}, $p[z](v,v)=p[w](v, v)=1.$ 
Thus, by item~\eqref{i:pz:2} (twice), $\ell(z,v)=0=\ell(w,v)$ and the proof of item~\eqref{i:pz:7}, and thus the lemma, is complete.
\end{proof}
{ We point out that the conclusions of items (vi)-(vii) of Lemma~\ref{l:zeroslemma} continue to hold under the weaker (in view of Theorem~\ref{generalsuff}) assumptions that $(k, \ell)$ have the CP property and $k(z, w)=0$; see Proposition~\ref{CPsplit}.}

\subsection{Shimorin versus strong Shimorin certificates}
In this section we consider two sufficient conditions under which the existence of a Shimorin certificate implies the existence of a strong Shimorin certificate. The first is a general condition that applies in a number of cases of interest. The second includes the assumption that $k$ and $\ell$ are holomorphic kernels
on a domain $\Omega\subseteq \CC^\vg.$

 \begin{proposition} \label{betweencerts}
   Let $k, \ell$ be kernels on $X.$  If $\ell$ is non-vanishing, then $(k, \ell)$ has a Shimorin certificate if and only if $(k,\ell)$  has a strong Shimorin certificate if and only if there exists a kernel $p$ on $X$ and a point $z_0\in X$ with
    \begin{equation}\label{onecert}
        \ell^{z_0}\succeq p\ell  \hspace*{0.4 cm}  \text{ and }     \hspace*{0.4 cm}  k^{z_0}\preceq pk.
    \end{equation}
\end{proposition}

The proof of  Proposition~\ref{betweencerts} will be based on the following two lemmas.

\begin{lemma}\label{CPPsplits}
If $k,\ell$ are kernels over the set $X$ and $s$ is a strong Shimorin certificate for  $(k,\ell),$  then,
 with respect to the unique partition $X=\cup X_i$ for $s$ from Lemma~\ref{CPsplits}, 
\begin{enumerate}[(i)]
 \item $k(z,w)=\ell(z,w)=s(z,w)=0$ whenever $z\in X_i$ and $w\in X_j$ with $i\neq j;$
 \item for each $i,$ the function $s|_{X_i\times X_i}$ is non-vanishing;
 \item \label{i:cpp:3} for each $i,$ either $k|_{X_i\times X_i}$ is non-vanishing or identically $0.$
\end{enumerate}
\end{lemma}
\begin{proof}
 Since $s$ is a strong Shimorin certificate for $(k,\ell),$ there exist kernels $h$ and $g$ on $X$ such  that 
\[
 k=(1-h)s,  \hspace*{0.7 cm} \ell=gs.
\]
By Lemma~\ref{CPsplits}, there exists a unique decomposition $X=\cup_{i\in I}X_i$ such that $s_i:=s|_{X_i\times X_i}$ never vanishes and 
$s|_{X_a\times X_b}\equiv 0$  whenever $a\ne b,$ which also yields $k|_{X_a\times X_b}\equiv 0$ and $\ell|_{X_a\times X_b}\equiv 0$. 
Since $s_i$ never vanishes, Lemma~\ref{maxmod} tells us that $1-h$ is either non-vanishing on $X_i\times X_i$ or identically zero. As $k=(1-h)s,$ we conclude that $k$ splits into a collection of kernels $k_i=k|_{X_i\times X_i}$ that are non-vanishing. 
\end{proof}

\begin{lemma}
 \label{shim=cert}
 Suppose $k,\ell,p$ are kernels over the nonempty set $X.$
 If $\ell^{z_0}\succeq p\ell$ and $k^{z_0}\preceq pk$ for some $z_0\in X$ and $p(x,x)<1$ for all $x\in X,$ then
\[
 s(w,v)=\frac{k(w, z_0)k(z_0, v)}{k(z_0, z_0)}\frac{1}{1-p(w, v)}
\]
 is a strong Shimorin certificate for $(k,\ell).$
\end{lemma} 

\begin{proof}
 By positivity, $|p(w,v)|<1$ for all $w,v\in X$ and it is evident from the discussion at the outset of subsection~\ref{CPsubsec} that $s$ is a complete Pick kernel.  From $k^{z_0}\preceq pk,$ there exists a kernel $C$ on $X$ such that
\[k(w,v)-\frac{k(w, z_0)k(z_0, v)}{k(z_0, z_0)}+C(w, v)=p(w, v)k(w, v),\]
and therefore,
\[ k(w, v)=\frac{k(w, z_0)k(z_0, v)}{k(z_0, z_0)}\frac{\Big(1-\frac{k(z_0, z_0)}{k(w, z_0)k(z_0, v)}C(w, v)\Big)}{1-p(w, v)},\]
for any $w, v\in X$. Similarly, there exists a kernel $D$ such that, for all $w,v\in X,$ 
\[\ell(w, v)=\frac{\frac{\ell(w, z_0)\ell(z_0, v)}{\ell(z_0, z_0)}+D(w, v)}{1-p(w, v)}. \]
Defining $s$ as in the statement of the proposition and setting 
\[
 \begin{split} h(w, v)& =\frac{k(z_0, z_0)}{k(w, z_0)k(z_0, v)}C(w, v)
\\ g(w, v) & =\frac{k(z_0, z_0)}{k(w, z_0)k(z_0, v)}\bigg(\frac{\ell(w, z_0)\ell(z_0, v)}{\ell(z_0, z_0)}+D(w, v)\bigg),
\end{split}
\]
 we see that $k=(1-h)s$ and $\ell=sg$, as desired.
\end{proof}

\begin{proof}[Proof of Proposition~\ref{betweencerts}]

 Since we have already shown that having a strong Shimorin certificate implies the existence of a Shimorin certificate (which, in turn, clearly implies \eqref{onecert}), all that remains is to show that, under the assumption that $\ell$ never vanishes, \eqref{onecert} leads to a strong Shimorin certificate, a conclusion that  follows immediately from Lemma~\ref{shim=cert} after noting that $\ell$ non-vanishing implies $p(x,x)<1$, for all $x$, by Lemma~\ref{l:zeroslemma}~\eqref{i:pz:2}.
\end{proof}

The notions of Shimorin and strong Shimorin certificate also coincide for pairs of holomorphic kernels over  domains $\Omega\subseteq\mathbb{C}^d$, even if $\ell$ has zeros. To prove this claim, we need a few preliminary lemmas.
   \par  
Assume $(k, \ell)$ is a pair of kernels on $X$ with a Shimorin certificate $\{p[z]\}_{z\in X}$. For $w\in X$, let
\begin{equation}\label{certdecomp}
X^w_0:=\{z\in X: p[w](z, z)<1\}, \hspace*{0.4 cm} X^w_1:=\{z\in X: p[w](z, z)=1\}.    
\end{equation}
\begin{lemma}\label{easy}
For every $w\in X,$  the sets $X^w_0, X^w_1$ partition $X.$ Further, $w\in X^w_0,$ so $X^w_0$ is always non-empty.
\end{lemma}
\begin{proof}
    The proof is straightforward and is omitted.
\end{proof}
\begin{remark}
   As Remark~\ref{anykernel} shows, it could  happen that $X^w_0=\{w\}$ for every $w$. On the other hand, it could also happen that $X^w_0=X$ for every $w$ (this coincides with the existence of a strong Shimorin certificate $s$ that does not vanish - see Lemma~\ref{shim=cert}). 
\end{remark}
  
\begin{lemma} \label{ellsplit}
   Given  $w\in X,$ if   $z\in X^w_0$ and  $v\in X^w_1$, then  $\ell(z, v)=0.$
\end{lemma}
\begin{proof}
By definition, $p[w](v, v)=1.$  Further, since $p[w](z, z)<1,$ we must have $|p[w](z, v)|<1$. Lemma \ref{l:zeroslemma} item~\eqref{i:pz:3} then implies $\ell(z, v)=0$.
\end{proof}

The following lemma is well-known. We include a proof for the reader's convenience. 
\begin{lemma}
 \label{toomany0}
   Let $\Omega$ denote a domain in $\CC^\vg.$ If $f_1,\dots,f_n:\Omega\to\CC$ are non-zero and holomorphic, then  $\Omega \ne \cup Z(f_j),$  where $Z(f_j)$ denotes the zero set of $f_j$ (in $\Omega$).
\end{lemma}

\begin{proof}
 Since the $f_j$'s are non-zero and holomorphic, the sets  $U_j = \Omega\setminus Z(f_j)$ are open and dense in $\Omega.$  Hence, so is 
 $\cap_{j=1}^n U_j.$ In particular, $\cap_{j=1}^n U_j\ne\varnothing.$ Equivalently,
 $\cup_{j=1}^n  Z(f_j)\ne \Omega.$
\end{proof}

The following is our main result in this subsection.

\begin{proposition} \label{Scertforhol}
Assuming  $(k, \ell)$ are non-zero holomorphic kernels on the (connected) domain $\Omega\subseteq\mathbb{C}^{\vg},$ the pair  $(k, \ell)$ has a Shimorin certificate if and only if it has a strong Shimorin certificate,  in which case $k$ is  non-vanishing.
\end{proposition}

\begin{proof}
 We first show the existence of a Shimorin certificate implies $k$ is non-vanishing. Assuming $k$ has a Shimorin certificate and  arguing by contradiction, suppose $z,w\in \Omega$ and $k(z,w)=0$ and consider the sets 
   \begin{equation*}
\begin{split}
\Omega_1&=\{v\in\Omega : k_z(v)=k(v,z)=0\}, \\
\Omega_2&=\{v\in\Omega : k_w(v)=k(v,w)=0\}, \\
\Omega_3&=\{v\in\Omega : \ell_z(v)= \ell(v, z)=0\}.
\end{split}
\end{equation*}
Since the kernels $k,\ell$ are holomorphic in their first argument, the sets  $\Omega_i$ are the zero sets for  (non-zero) holomorphic functions on $\Omega$  and, by  Lemma~\ref{l:zeroslemma}~\eqref{i:pz:7}  it follows that $\cup_i\Omega_i=\Omega,$ contradicting the conclusion of Lemma~\ref{toomany0}.
  \par
Since the existence of a strong Shimorin certificate implies the existence of a Shimorin certificate (Proposition~\ref{newcerttt}), what remains to be proved is that the existence of a Shimorin certificate implies that of a strong Shimorin certificate. In view of Lemma~\ref{shim=cert}, it suffices to show there is a point $z\in X$ and a (PsD) kernel $p$ on $X$ such that $\ell^z\succeq p\ell$ and $k^z\preceq pk$  and with  $p(x,x)<1$ for all $x\in X.$
To this end, fix $w\in \Omega$ and consider the decomposition $\Omega=\Omega^w_0\cup \Omega^w_1$ as in \eqref{certdecomp}. Arguing by contradiction, suppose  $\Omega^w_1$ is non-empty and choose $v\in \Omega^w_1.$ Lemma~\ref{ellsplit} then tells us that $\Omega^w_0$ is contained in the zero set of $\ell_v,$ while $\Omega^w_1$ is contained in the zero set of $\ell_w,$ both of which are non-zero holomorphic functions $\ell_v,\ell_w:\Omega\to\CC.$  Since the union of these two sets is equal to all of $\Omega,$ Lemma~\ref{toomany0} gives a contradiction. Thus, $\Omega^w_1$ is empty and $\Omega=\Omega^w_0$ and consequently,  $p[w](z,z)<1$ for  all $z\in\Omega.$  Hence, by Lemma~\ref{shim=cert}, $(k,\ell)$ has a strong Shimorin certificate.
\end{proof}

\normalsize
\section{Sufficient Conditions for the CP property}\label{suffforCP}
\normalsize
This section contains a new proof of Theorem~\ref{mainShim}, one based upon a general version of the Leech factorization Theorem. It also contains a generalization of Theorem~\ref{mainShim}, Theorem~\ref{generalsuff}, proved using  a one-step extension argument  that replaces the  strong Shimorin certificate assumption of Theorem~\ref{mainShim} with the  weaker assumption of the existence of a Shimorin certificate.

\subsection{Theorem~\ref{mainShim}}

 This subsection begins with our Leech-based proof of Theorem~\ref{mainShim} and concludes with a corollary that generalizes 
 a result of \cite{AHMRinterpol}.

\begin{proof}[Proof of Theorem~\ref{mainShim}]
Let $s$ denote a strong  certificate for $(k,\ell).$ Thus,  $s$ is a CP kernel and there exist kernels $h$ and $g$ on $X$ such that $k=(1-h)s$ and $\ell=sg$.
From Lemma~\ref{CPPsplits}, there exists a decomposition $X=\cup_{i\in I}X_i$ such that $k(z,w)=s(z,w)=\ell(z,w)=0$ whenever $z\in X_i$ and $w\in X_j$ with $i\neq j.$ Further,  $ s|_{X_i\times X_i}$ is non-vanishing and $k|_{X_i\times X_i}$ is non-vanishing. 
 Since $(k, \ell)$ has the CP property if and only if every restriction $(k|_{X_i\times X_i}, \ell|_{X_i\times X_i})$ does, 
 we  assume, without loss of generality, that both $k$ and $s$ are non-vanishing on $X.$ In particular, $|h(z, z)|<1$
 for $z\in X.$ \par 
There exist Hilbert spaces $\mathcal{L}_1, \mathcal{L}_2$ and functions $\Psi: X\to \mathcal{B}(\mathcal{L}_1, \mathbb{C}),$ and  $G: X\to\mathcal{B}(\mathcal{L}_2, \mathbb{C})$ such that $h(z,w)=\Psi(z)\Psi(w)^*$ and $g(z,w)=G(z)G(w)^*$. Since $(1-\Psi\Psi^*)s=k\succeq 0$, it follows that  $\Psi\in\Mult(\mathcal{H}_s\otimes \mathcal{L}_1, \mathcal{H}_s)$ is contractive. Further, $\|\Psi(z)^*\|^2=|h(z, z)| <1.$ 
\par
Now,  assume $z_1, \dots, z_n\in X$  and $W_1, \dots, W_n$ are $N\times N$ matrices that satisfy
 \begin{equation}\label{otherinitial}
\big[\ell(z_i, z_j)I_{N\times N}-k(z_i, z_j)W_i W_j^*\big]_{i, j=1}^n\succeq 0.
 \end{equation}
Rewriting \eqref{otherinitial} as 
\[ g(z_i, z_j)s(z_i, z_j)I_{N\times N}+h(z_i, z_j)s(z_i, z_j)W_iW_j^*-s(z_i, z_j)W_iW_j^*\succeq 0,\]
gives 
\[
 \big(U_i U_j^*-W_iW_j^* \big)s(z_i, z_j)\succeq 0,
\]
where
 \[
 U_j =\begin{bmatrix} I_{N\times N}\otimes G(z_i) & W_i \otimes \Psi(z_i) \end{bmatrix}\in\mathcal{B}\big((\mathbb{C}^N\otimes\mathcal{L}_2) \oplus  (\mathbb{C}^N\otimes\mathcal{L}_1), \mathbb{C}^N\big),
\]
 for all $i.$ Applying the generalized version of Leech's Theorem (\cite[Theorem 8.57]{Pickbook}) produces a contractive multiplier 
\[
\Phi=\begin{bmatrix}
  \Phi_1 \\ \Phi_2  
\end{bmatrix}\in \Mult\big(\mathcal{H}_s\otimes\mathbb{C}^N, \mathcal{H}_s \otimes \big((\mathbb{C}^N\otimes\mathcal{L}_2) \oplus  (\mathbb{C}^N\otimes\mathcal{L}_1)\big)\big)
\]
 such that { $W_i=U_i\Phi(z_i),$} for all $i.$ This last equality can be rewritten as
\begin{equation}\label{WivPhi}
\begin{split}
W_i&=\big(I_{N\times N}\otimes G(z_i)\big)\Phi_1(z_i)+\big(W_i \otimes \Psi(z_i)\big)\Phi_2(z_i)  \\
&=\big(I_{N\times N}\otimes G(z_i)\big)\Phi_1(z_i)+
 W_i \big(I_{N\times N} \otimes \Psi(z_i)\big)\Phi_2(z_i),
\end{split}
\end{equation}
for all $i.$
\par
 Since $\Phi$ is a contractive multiplier, $\|\Phi_2(z)\|\le 1$ for all $z.$ Since also  $\|\Psi(z)^*\|<1,$ it follows that, for
 each $z,$ the $N\times  N$ matrix $(I_{N\times N}\otimes \Psi(z))\Phi_2(z)$ is a strict contraction and therefore 
 $I_{N\times N}-(I_{N\times N}\otimes \Psi(z))\Phi_2(z)$ is pointwise invertible. Define $R:X\to \mathbb{M}_N$ by
\[
R(z)=\big(I_{N\times N}\otimes G(z)\big) \Phi_1(z)\big(I_{N\times N}-\big(I_{N\times N}\otimes \Psi(z)\big) \Phi_2(z)\big)^{-1}.
\]
 Solving equation~\ref{WivPhi} for $W_i$ gives 
\[
 W_i=R(z_i)
\]
for all $i.$  Defining 
$H:X\to \mathcal{B}\big((\mathbb{C}^N\otimes\mathcal{L}_2) \oplus  (\mathbb{C}^N\otimes\mathcal{L}_1), \mathbb{C}^N\big)$
by
\[
 H(z)=\begin{bmatrix} I_{N\times N}\otimes G(z) & R(z) \otimes \Psi(z) \end{bmatrix},
\]
 it follows from the definition of $R$ that $R(z)=H(z) \Phi(z),$ for all $z\in X$,
 {and, from the definition of $U_i,$ that $H(z_i)=U_i.$}
Hence 
\begin{equation*}
\begin{split}
    \ell(z, w)I_{N\times N}-k(z, w)R(z)R(w)^*&=\big(H(z) H(w)^*-R(z)R(w)^* \big)s(z, w) \\
    &=H(z)\big((I-\Phi(z)\Phi(w)^* )s(z, w)\big)H(w)^*\succeq 0,
\end{split}
\end{equation*}
since $(I-\Phi(z)\Phi(w)^* )s(z, w)\succeq 0$. Thus, $R\in\Mult(\mathcal{H}_k \otimes\mathbb{C}^N, \mathcal{H}_{\ell} \otimes\mathbb{C}^N)$ is an interpolating contractive multiplier and our proof is complete. 
\end{proof}

As a consequence of the above proof, we obtain the following parametrization of multipliers between spaces with a strong Shimorin certificate, which generalizes \cite[Proposition 4.10]{AHMRinterpol}.
\begin{corollary} \label{parametcorol}
Let $k, s, \ell, \Psi, G$ be as in the previous proof. A function $R: X\to \mathbb{M}_N$ is a contractive multiplier from $\mathcal{H}_k\otimes\mathbb{C}^N$ to $\mathcal{H}_{\ell}\otimes\mathbb{C}^N$ if and only if there exists  a contractive multiplier \[\Phi=\begin{bmatrix}
  \Phi_1 \\ \Phi_2  
\end{bmatrix}\in \Mult\big(\mathcal{H}_s\otimes\mathbb{C}^N, \mathcal{H}_s \otimes \big((\mathbb{C}^N\otimes\mathcal{L}_2) \oplus  (\mathbb{C}^N\otimes\mathcal{L}_1)\big)\big)\] satisfying 
$$R(z)=\big(I_{N\times N}\otimes G(z)\big) \Phi_1(z)\big(I_{N\times N}-\big(I_{N\times N}\otimes \Psi(z)\big) \Phi_2(z)\big)^{-1},$$ for all $z\in X.$ 
\end{corollary}
\begin{remark}
A standard modification of the above argument can be used to generalize \cite[Theorem 8.57]{Pickbook} to pairs of kernels possessing a strong Shimorin certificate. Such a generalization has already been obtained by Shimorin as a corollary of his commutant lifting theorem; see \cite[Corollary 2.3]{Shimorin}. 
\end{remark}
\begin{remark}
We will shortly prove Theorem~\ref{generalsuff} (sufficiency for Shimorin certificates), which (in view of Proposition~\ref{newcerttt}) contains Theorem~\ref{mainShim} as a special case. However, we still chose to include the Leech factorization proof of Theorem~\ref{mainShim} because it is of independent interest and  also yields Corollary~\ref{parametcorol}. 
\end{remark}

\subsection{Shimorin Certificates and the CP property for pairs}
Next, we will briefly investigate how the CP property for $(k, \ell)$ places restrictions on the linear independence of kernel functions before turning to generalizing Theorem~\ref{mainShim} to pairs $(k,\ell)$ that have a Shimorin certificate. We again assume that all pairs $(k, \ell) $ consist of kernels that are non-vanishing along the diagonal.

In the case of a single CP kernel $k$ on a set $X,$ it is known 
 (see \cite[Lemma~7.5]{Pickbook}) that if $k_z, k_w$ are linearly independent for every choice of
distinct points $z,w\in X$  then the same must hold for any finite collection $\{k_{z_1}, \dots, k_{z_n}\},$ where $z_1, \dots, z_n\in X$ are $n$ distinct points. Moving to the two-kernel setting, Remark~\ref{anykernel} tells us that, even if, for each $n$ and distinct point $z_1,\dots,z_n\in X$  the collection $\{\ell_{z_1}, \dots, \ell_{z_n}\}$ is  linearly independent,  we cannot draw any conclusions concerning $k_{z_1}, \dots, k_{z_n}$ without further knowledge of $(k, \ell)$. But what if $z_1,\dots,z_n\in X$ are distinct, but $\{\ell_{z_1}, \dots, \ell_{z_n}\}$ is linearly dependent? 
\begin{lemma} \label{deplemma}
Assume $(k, \ell)$ is a CP pair on $X$ and  $n\ge 2.$  Further, assume   $z_1, \dots, z_n$ are distinct points in $X,$ the vectors  $\ell_{z_1}, \dots, \ell_{z_{n-1}}$ are linearly independent and  there exist scalars $c_1, \dots c_{n-1}\in\mathbb{C}$ such that 
\begin{equation}\label{elldep}
  \ell_{z_n}=\sum_{i=1}^{n-1}c_i\ell_{z_i}.  
\end{equation}
If $1\le j\le n-1$ and  $c_j\neq 0$,  then $k_{z_n}, k_{z_j}$ are linearly dependent. 
\end{lemma}
\begin{proof}
Assume that there exist scalars $c_1, \dots c_{n-1}$, not all of which can be zero, such that \eqref{elldep} holds. Given 
 $e_1,\dots,e_{n-1}\in \mathbb{C},$ define \[R: \underset{1\le i\le n-1}{\text{span}}\{\ell_{z_i}\}\to \underset{1\le i\le n-1}{\text{span}}\{k_{z_i}\}\]
by $R\ell_{z_i}=e_i k_{z_i}$ (where we suppress the dependence of $R$ on $e$). 
By the independence of $\ell_{z_1},\dots,\ell_{z_{n-1}},$  there exists $\delta>0$ such that $R$ is a contraction whenever $|e_i|<\delta$ for all $i.$ Thus, for any such choice of $e_1, \dots, e_{n-1},$ the CP property of $(k, \ell)$ implies the  existence of $a\in\mathbb{C}$ such that the extension 
\[R_a: \underset{1\le i\le n}{\text{span}}\{\ell_{z_i}\}\to \underset{1\le i\le n}{\text{span}}\{k_{z_i}\}\]
of $R$ determined by $R_a\ell_{z_n}=ak_{z_n}$ is a contraction. Thus, \eqref{elldep} gives us 
\begin{equation}\label{kdep}
\sum_{i=1}^{n-1}e_ic_ik_{z_i}=R\bigg(\sum_{i=1}^{n-1}c_i\ell_{z_i}\bigg)=R\big(\ell_{z_n}\big)\in\text{span}\{k_{z_n}\},    
\end{equation}
for any $e_1, \dots, e_{n-1}$ with modulus less than $\delta.$ Now, assume $c_j\neq 0$ and choose $e_j\neq 0$ sufficiently small with $e_i=0$ whenever $i\neq j.$ The equality \eqref{kdep} implies $k_{z_j}\in\text{span}\{k_{z_n}\},$ as desired.
\end{proof}

We will conclude this section by showing that the existence of a Shimorin certificate is sufficient to guarantee the CP property. The proof employs  the standard one-point-extension argument in combination with Parrott's Lemma, which were also the main ingredients in Shimorin's proof of Theorem~\ref{mainShim}.

\begin{theorem} \label{generalsuff}
If  $(k, \ell)$ is a pair of kernels that possesses a Shimorin certificate, then  $(k, \ell)$ has the CP property. 
\end{theorem}
\begin{proof} We will first prove the theorem under the assumption that for each $n$ and distinct points $z_1,\dots,z_n,$  the kernel  functions $\ell_{z_1}, \dots, \ell_{z_n},$ are  linearly independent. 
   \par 
 Assume $z_1, \dots, z_n$ are given distinct points in $X$ and $W_1, \dots, W_n$ are $N\times N$ matrices that satisfy
 \begin{equation}\label{initial}
\big[\ell(z_i, z_j)I_{N\times N}-k(z_i, z_j)W_i W_j^*\big]_{i, j=1}^n\succeq 0.
 \end{equation}
We will first show that, for any $z_{n+1}\in X\setminus\{z_1,\dots,z_n\},$ there exists a matrix $W_{n+1}$ so that 
\begin{equation}\label{final}
\big[\ell(z_i, z_j)I_{N\times N}-k(z_i, z_j)W_i W_j^*\big]_{i, j=1}^{n+1}\succeq 0.
 \end{equation}
Let $\{u^{\alpha}\}_{\alpha=1}^N$ be a basis for $\mathbb{C}^N$ and set 
\[\mathcal{M}^k_n=\text{span}\{k_{z_i}\otimes u^{\alpha} : 1\le i\le n, 1\le\alpha\le N\}, \]
\[\mathcal{M}^{\ell}_n=\text{span}\{\ell_{z_i}\otimes u^{\alpha} : 1\le i\le n, 1\le\alpha\le N\}.\]
Define $R: \mathcal{M}_n^{\ell}\to\mathcal{M}_n^{k}$ by
\begin{equation}\label{finiteadj}
R: \ell_{z_i}\otimes u^{\alpha}\mapsto k_{z_i}\otimes W^*_iu^{\alpha}    
\end{equation}
and extend linearly. The operator $R$ is a contraction if and only if for every choice of scalars $\{a^{\alpha}_i\}$ we have 
\[\big\langle\big(I-R^*R\big)\sum_{j, \beta}a^{\beta}_j\ell_{z_j}\otimes u^{\beta}, \sum_{i, \alpha}a^{\alpha}_i\ell_{z_i}\otimes u^{\alpha}\big\rangle\] \[=\sum_{i, j, \alpha, \beta}a^{\beta}_j\overline{a}^{\alpha}_i\Big( \ell(z_i, z_j)\langle u^{\beta}, u^{\alpha}\rangle-k(z_i, z_j)\big\langle W_iW^*_ju^{\beta}, u^{\alpha}\big\rangle\Big)\ge 0,\]
which is equivalent to \eqref{initial}. Thus, $\|R\|\le 1$. \par
Now, by the independence assumption on $\ell_{z_1},\dots,\ell_{z_{n+1}},$  for each choice of $W=W_{n+1}$ there is  an extension $R_W$ of $R$ on $\mathcal{M}_{n+1}^{\ell}$ uniquely determined by
\begin{equation} \label{onepointR}
R_W: \ell_{z_{n+1}}\otimes u^{\alpha}\mapsto k_{z_{n+1}}\otimes W^*u^{\alpha}.    
\end{equation}
When does there exist $W$ so that $\|R_W\|\le 1$ (which is equivalent to \eqref{final})? The answer is given by the following lemma, the proof of which rests on an application of Parrott's Lemma to an appropriate matrix decomposition of $R_W.$ The single-kernel version of this argument is well-known (see e.g. \cite{Quiggin}, \cite{AmcCcompleteNP}, \cite{Pickbook} and the unpublished work of Agler), and there are essentially no new wrinkles in this generalized two-kernel version. Thus, we omit the details. 

\begin{lemma}\label{boringonepoint}
 In the above setting, there exists $W$ so that $\|R_W\|\le 1$ if and only if 
\begin{equation}
    \label{e:boringonepoint}
\big[\ell^{z_{n+1}}(z_i, z_j) I_{N\times N}-k^{z_{n+1}}(z_i, z_j)W_iW_j^* \big]_{i, j=1}^n\succeq 0.
\end{equation}
\end{lemma} 
Now, let $J$ denote the matrix all of whose entries are $1$ with respect to $\{u^{\alpha}\}$. By assumption, there exists a Shimorin certificate $\{p[x]\}_{x\in X}$ for $(k, \ell).$ Thus, 
\begin{equation*}
\begin{split}
&\ell^{z_{n+1}}(z_i, z_j) I_{N\times N}-k^{z_{n+1}}(z_i, z_j)W_iW_j^*  \\ 
=&I_{N\times N} \cdot \big[\ell^{z_{n+1}}(z_i, z_j)\otimes J\big]-W_iW_j^*\cdot \big[k^{z_{n+1}}(z_i, z_j)\otimes J\big] \\
\succeq &I_{N\times N}\cdot\big[p[z_{n+1}](z_i, z_j)\ell(z_i, z_j)\otimes J\big]
-W_iW_j^*\cdot \big[p[z_{n+1}](z_i, z_j)k(z_i, z_j)\otimes J\big] \\ 
=&\big[\ell(z_i, z_j)I_{N\times N}-k(z_i, z_j)W_i W_j^*\big]\cdot\big[p[z_{n+1}](z_i, z_j)\otimes J\big],
\end{split}
\end{equation*}
which is positive, being the Schur product of \eqref{initial} with a positive matrix. Thus, Lemma~\ref{boringonepoint} tells us that, whenever $z_1, \dots, z_n\in X$ and $W_1, \dots, W_n$ are $N\times N$ matrices that satisfy \eqref{initial} and $z_{n+1}$ is any distinct point in $X,$
there exists a matrix $W_{n+1}$ so that \eqref{final} holds. In other words, we have shown that one can always extend a multiplier defined on a finite subset of $X$ to any other point without increasing the norm. We can now employ either transfinite induction (as in \cite[Section 3]{Quiggin}) or a compactness argument (as in \cite[Proposition 2.9]{Serra})
to obtain a contractive multiplier satisfying the initial conditions. We omit the details.
      \par 
 Finally, we prove the general case, where it is possible to have linearly dependent kernel functions. As before, assume $z_1, \dots, z_n$ are given distinct  points in $X$ and $W_1, \dots, W_n$ are $N\times N$ matrices that satisfy \eqref{initial}. Without loss of generality we assume that $\{\ell_{z_1}, \dots, \ell_{z_m}\}$ is a maximal linearly independent subset of $\{\ell_{z_1}, \dots, \ell_{z_n}\}$.
By Zorn's lemma, there exists  a maximal subset $Y\subseteq X$ with the properties that no finite subcollection of $\{\ell_v : v\in Y\}$ is linearly dependent and  $\{z_1, \dots, z_m\}\subseteq Y.$ Let $\widetilde{k}=k|_{Y\times Y}$ and $\widetilde{\ell}=\ell|_{Y\times Y}$. Since $(k, \ell)$ has the CP property, the same must be true for $(\widetilde{k}, \widetilde{\ell})$. Note that \eqref{initial} continues to hold if we restrict to the points $\{z_1, \dots, z_m\}$. Thus, in view of what has already been proved, there exists $\widetilde{\Phi}: Y\to \mathbb{M}_N$ such that $\widetilde{\Phi}$ is a contractive multiplier $\mathcal{H}_{\widetilde{k}}\otimes \mathbb{C}^N\to\mathcal{H}_{\widetilde{\ell}}\otimes \mathbb{C}^N$ that satisfies $\widetilde{\Phi}(z_i)=W_i$ for $1\le i\le m.$ \par
Since the span of $\{\ell_x:x\in X\}$ is dense in  $\mathcal{H}_\ell$, by construction, the span of $\{\ell_y:y\in Y\}$ is dense in $\mathcal{H}_\ell$
and, in fact, the restriction map $\mathcal{H}_{\ell}\to\mathcal{H}_{\widetilde{\ell}}$ is unitary.
Thus $\widetilde{\Phi}$ induces a contractive 
 map $M^*:\mathcal{H}_\ell\otimes \mathbb{C}^N\to \mathcal{H}_k\otimes \mathbb{C}^N$ determined by $M^* \ell_y\otimes u= k_y\otimes \widetilde{\Phi}(y)^*u$ for $y\in Y$ and $u\in \mathbb{C}^N.$  Now, we argue that the function  $\widetilde{\Phi}:Y\to \mathbb{M}_N$ has a unique extension to a function $\Phi:X\to \mathbb{M}_N$ such that  $M^*=M_\Phi^*.$ The function  $\Phi$ so extended satisfies $\Phi(z_j)=W_j$ for all $j$ and thus solves the original interpolation problem. Indeed, given $x\in X\setminus Y,$ we know (by definition of $Y$) that there exist unique points $v_1, \dots ,v_p\in Y$ and unique non-zero scalars $c_1, \dots, c_p$ such that 
 \begin{equation*}\label{ellextend}
     \ell_x=\sum_{i=1}^pc_i\ell_{v_i}.
 \end{equation*}
   In view of Lemma~\ref{deplemma}, there exist, for each $1\le i\le p,$  (non-zero) scalars $d_1, \dots, d_p$ such that $k_{v_i}=d_i k_x.$ Thus,
\[
 M^* \ell_x\otimes u = M^* \bigg(\sum_{i=1}^p c_i \ell_{v_i}\otimes u\bigg) = \sum_{i=1}^p c_i k_{v_i}\otimes \widetilde{\Phi}(v_i)^* u
 = k_x \otimes \bigg[\sum_{i=1}^p c_i d_i \widetilde{\Phi}(v_i)^* u\bigg].
\]
Hence, setting $\Phi^*(x)=\sum_{i=1}^p c_i d_i \widetilde{\Phi}(v_i)^*,$ we have $M^* = M_{\Phi}^*$ and this condition uniquely determines
 $\Phi.$  By construction, $\Phi(z_j)=W_j$ for $1\le j\le m.$  While not as transparent, it is also true that $\Phi(z_j)=W_j$
 for $j>m.$ Indeed, applying the argument above to $j>m=p$ and $\ell_{z_j} =\sum_i c_i \ell_{z_i}$ yields
\[
 k_{z_j}\otimes  \Phi^*(z_j)u = M_\Phi^* \big(\ell_{z_j}\otimes u\big) = k_{z_{j}} \otimes \sum c_i d_i W_i^* u.
\]
 On the other hand, the  matrices $W_j$ for $j>m$ are determined by the $W_j$ for $j\le m.$ To verify this claim,
let $\mathcal{E}$ denote the span of $\{\ell_j:1\le j\le n\}$ and $\mathcal{F}$ the span of $\{k_j:1\le j\le n\}.$ 
The inequality of equation~\ref{initial} is equivalent to the statement that the  mapping $T^*: \mathcal{E}\otimes \CC^N\to \mathcal{F}\otimes  \CC^N$ determined by $T^* \big(\ell_{z_j} \otimes u\big)= k_{z_j}\otimes W_j^*u$ is well-defined and a contraction. In particular,
 arguing as above with the same notations,
\[
 k_{z_j}\otimes W_j^* u =   T^* \big(\ell_{z_j}\otimes u\big)=  \sum_{i=1}^m c_i k_{z_i} \otimes  W_i^* u
  = k_{z_j}\otimes \sum c_id_i W_i^* u.
\]
Thus $W_j^* = \sum_{i=1}^m  c_id_i W_i^* =\Phi(z_j)^*.$
\end{proof}

\normalsize
\section{The CC Property for Diagonal Holomorphic Pairs} \label{theCCsection}
\normalsize
Throughout this section, we work  exclusively  with diagonal holomorphic kernels $k, \ell.$ Unless otherwise noted, they are assumed normalized.
\subsection{One-step extensions}
 Proposition~\ref{convenient} below, whose proof depends upon the following lemma, 
 interprets Proposition~\ref{d:ccp} in terms of one-step extensions.

Recall, for $a=(a_1,\dots,a_\vg)\in \mathbb{N}^\vg,$ the \df{length} of $a$ is $|a|=\sum |a_j|.$ \index{$\vert a\vert$}
\begin{lemma}\label{compatib}
    Assume $\varnothing \ne F\subseteq\mathbb{N}^{\vg}$ is finite and  co-invariant and $d\in\mathbb{N}^{\vg}$ satisfies
\begin{enumerate}[(i)]
    \item  $d\notin F;$
    \item  if $a\in\mathbb{N}^{\vg}$ and $|a|<|d|,$ then $a\in F.$
\end{enumerate}
Setting  $F^+:=(F\cup\{d\})\setminus\{\mathbf{0}\},$ the set $F^+\cup \{\mathbf{0}\}$ is co-invariant  and if  $a,b\in F^+$ and $a\le b,$ then  $b-a\in F.$ 
\end{lemma}

\begin{proof}
    If $a, b\in F$, we get the desired conclusion because  $F$ is co-invariant. If $b=d,$ then, since $|a|\ge 1,$ we must have $|b-a|<|b|=|d|,$ hence $b-a\in F$ in view of our hypothesis for $d.$ Finally, if $a=d,$ then we cannot have $b\in F$ (since $d\le b$ would imply that $d\in F$, by co-invariance), and so we must have $b-a=d-d=\textbf{0}.$ \par 
    Now, to check co-invariance, assume $b\in F^+$ and let $a\in \mathbb{N}^{\vg}$ satisfy $a\le b$. If $b\in F,$ then we obtain $a\in F$ since  $F$ is co-invariant.  Thus, either $a\in F^+$ or $a=\{\mathbf{0}\}$. Further, if $b=d,$ then  either  $a=d$ (hence $a\in F^+$) or $|a|<|d|,$ which implies $a\in  F^+\cup\{\mathbf{0}\}$ in view of our hypothesis for $d.$
\end{proof}

\begin{proposition}\label{convenient}
A pair   $(k, \ell)$  of diagonal holomorphic kernels
 has the complete Carath\'eodory property  if and only if for each $J\in\mathbb{N},$ each finite co-invariant $F\subseteq\mathbb{N}^{\vg}$ and each  collection $\{c_a  : a\in F\}\subseteq \mathbb{M}_J$ such that the block (upper-triangular) matrix $C$ indexed by $F\times F$ with block $J\times J$ entries,
\begin{equation}
  \label{e:Cab}
   C_{a, b}=\begin{cases}
     c_{b-a}\sqrt{\cfrac{k_a}{\ell_b}}, \hspace{0.5 cm} b\ge a, \\
      \mathbf{0},  \hspace{1.05 cm} \text{otherwise,}
 \end{cases}
\end{equation}
is a contraction and for every $d\in\mathbb{N}^{\vg}$ such that 
\begin{enumerate}[(i)]
    \item \label{i:convenient:i}  $d\notin F;$ and 
    \item \label{i:convenient:ii} if $a\in\mathbb{N}^{\vg}$ and $|a|<|d|,$ then $a\in F,$
\end{enumerate}
the matrix $C^+$ indexed by $F^+\times F^+,$ where $F^+=(F\cup\{d\})\setminus\{\mathbf{0}\}$, and given by 
\begin{equation}\label{onepoint}
    C^+_{a, b}=\begin{cases}
     c_{b-a}\sqrt{\cfrac{k_a}{\ell_b}}, \hspace{0.5 cm} b\ge a, \\
      \mathbf{0},  \hspace{1.05 cm} \text{otherwise,}
 \end{cases} \end{equation}
 is also a contraction.
\end{proposition}

\begin{remark}
    The condition of, universally, passing from a contraction $C$ to a contraction $C^+$
 is the \df{one-step extension property}. 
\end{remark}

\begin{remark}\rm
\label{r:more-one-step}
 Observe Lemma~\ref{compatib} is implicitly used in defining  $C^+.$ 
 \end{remark} 
 
\begin{remark}
 The existence of a $d$ satisfying the conditions of items~\eqref{i:convenient:i}
 and~\eqref{i:convenient:ii} of Proposition~\ref{convenient}
 is not in doubt so long as $F\ne \NN^\vg.$  Indeed,
 let $m=\min \{ |f|: f\notin F\}$ and choose $d$ such that $d\notin F$ and $|d|=m.$
 In particular, by Lemma~\ref{compatib}, the set $G=F\cup\{d\}$ is co-invariant.
\end{remark}

\begin{proof}[Proof of Proposition~\ref{convenient}]
 Suppose $(k,\ell)$ has the CC  property and fix  a co-invariant $F$ and  $d\in\mathbb{N}^{\vg}\setminus F$ 
 is such that $F\cup\{d\}$ is also co-invariant. Let $F^+=(F\cup\{d\})\setminus\{\mathbf{0}\}$ 
 and assume $\{c_a:a\in F\}$ is such that the corresponding
 matrix $C$ in equation~\eqref{e:Cab} is a contraction. By assumption, there exists $c_d$
 such that the matrix  indexed by $(F\cup\{d\})\times (F\cup\{d\})$
\[
 \widetilde{C}_{a,b} = \begin{cases}
     c_{b-a}\sqrt{\cfrac{k_a}{\ell_b}}, \hspace{0.5 cm} b\ge a, \\
      \mathbf{0},  \hspace{1.05 cm} \text{otherwise.}
 \end{cases}
\]
 is also a contraction.  By considering the submatrix indexed by $F^+\times F^+,$ 
 we find that the matrix $C^{+}$ (indexed by $F^+\times F^+$),
\begin{equation} \label{anothonepoint}
     C^{+}_{a, b} = \begin{cases}
     c_{b-a}\sqrt{\cfrac{k_a}{\ell_b}}, \hspace{0.5 cm} b\ge a>0, \\
      \mathbf{0},  \hspace{1.7 cm} \text{otherwise,}
 \end{cases}
\end{equation}
 is a contraction.

 It remains to prove, if $(k,\ell)$ satisfies the one-step extension property as in the statement of the proposition,
 then $(k,\ell)$ has the CC property. Accordingly, 
 suppose  $F$ is co-invariant, $d$ satisfying the conditions of items~\eqref{i:convenient:i} and \eqref{i:convenient:ii} of the 
 proposition and $C$ as in equation~\eqref{e:Cab} is $C$ is a contraction. 
 Thus, by the one-step-extension assumption, $C^+$ is also a contraction.
 With $c_d$ to be determined, consider the enlarged $(F\cup\{d\})\times (F\cup\{d\})$ matrix
 \[
 \widetilde{C}_{a, b}=\begin{cases}
     c_{b-a}\sqrt{\cfrac{k_a}{\ell_b}}, \hspace{0.5 cm} b\ge a, \\
      \mathbf{0},  \hspace{1.05 cm} \text{otherwise.}
 \end{cases} 
\] 
 Partitioning 
\[
 \widetilde{C} = \begin{pmatrix} A & X \\   B & D\end{pmatrix},
\]
 where the scalar entry $X$ is to be determined and
\[
 \begin{pmatrix} A\\ B\end{pmatrix} =  \begin{pmatrix} \widetilde{C}_{a,b} \end{pmatrix}_{b\ne d}
    =\begin{pmatrix} C \\ \bzero \end{pmatrix},
\]
 (where the partitionings on the left and right are not the same) and
\[
 \begin{pmatrix} B & D \end{pmatrix} = \begin{pmatrix} \widetilde{C}_{a,b} \end{pmatrix}_{a\ne d}
  =\begin{pmatrix} \bzero & C^+ \end{pmatrix},
\]
 an application of Parrott's Lemma produces an $X$ such that $\widetilde{C}$ is a contraction.
 Setting  $c_d=\sqrt{\ell_d}X$  shows there is a choice of $c_d$ for which $\widetilde{C}$ is a  contraction.

 Since, by Lemma~\ref{compatib},  $F\cup \{d\}$ is also co-invariant, we may now proceed by induction; first, we extend
   $F$ finitely many (possibly zero) times so that it contains every $a\in\mathbb{N}^{\vg}$ with $|a|\le 1,$ then we perform another finite number of extensions to include every $a\in\mathbb{N}^{\vg}$ with $|a|\le 2,$ and so on. Indeed, choosing the new point $d\in\mathbb{N}^{\vg}$ so that it satisfies items~\eqref{i:convenient:i} and~\eqref{i:convenient:ii} at each step guarantees that the extended index set $F\cup\{d\}$ will always be co-invariant.  By induction we obtain $\{c_a\in \mathbb{M}_J: a\notin F\}$ so that 
the infinite matrix $\mathscr{C}$ given by \eqref{fullext} is  a contraction as well, in the sense that each finite submatrix is a contraction,  which is what the CC property requires.
\end{proof}

\subsection{Sufficient conditions for the CC property}

For the purposes of this section, we will temporarily recast Definition~\ref{strcertdef} in terms of formal power series. Domain consideration issues will not trouble us until Section~\ref{mainmain}. 
 \par

Given diagonal holomorphic kernels $(k,\ell)$, 
  a \df{formal Shimorin certificate} for $(k,\ell)$ is a formal power series 
\[
   t(x)=\sum_{|a|>0} t_a x^a,
\]
  where $t_a \ge 0$ for all $a$, for which there exist 
\[
 g(x)=\sum_{a\in\mathbb{N}^{\vg}}g_a x^a \hspace{0.2 cm} \text{ and } \hspace{0.2 cm} \hb(x)=\sum_{a\in\mathbb{N}^{\vg}}h_a x^a
\]
  such that $g_a, \hb_a\ \ge 0$ for all $a$ and 
\[
 \ell(1-t)=g, \ \ \ 1-k(1-t)=\hb,
\]
 in the sense of formal power series. These 
 conditions are equivalent to $g_{\mathbf{0}}=\ell_{\mathbf{0}}$ and $\hb_{\mathbf{0}}=1-k_{\mathbf{0}}$ and, for $|a|>0,$
\begin{equation}
 \label{lfact}
 \ell_a=g_a+\sum_{0<u\le a}t_u\ell_{a-u} 
\end{equation}
and 
\begin{equation} 
 \label{kfact}
  k_a+\hb_a=\sum_{0<u\le a}t_u k_{a-u}.
\end{equation}
\par 
 We will now show that existence of a formal Shimorin certificate guarantees the CC property. Even though this result is subsumed by Theorem~\ref{mainextmain} (to be proved in the next section), we have chosen to include the proof as it requires only a short detour and might be of independent interest. 

\begin{theorem}\label{ShimtoCCP}
 Suppose the kernels  $k(z, w)=1+\sum_{|a|\ge 1}k_az^a\overline{w}^a$ and $\ell(z, w)=1+\sum_{|a|\ge 1}\ell_az^a\overline{w}^a$ are diagonal and  holomorphic. 
 If there is a formal Shimorin certificate for $(k,\ell),$ then $(k,\ell)$ has the CC property. 
   \end{theorem}  
Before turning to the proof of Theorem~\ref{ShimtoCCP}, we  record
the following lemma.

\begin{lemma}\label{shifted}
Assume $F\subseteq\mathbb{N}^{\vg}$ is co-invariant. If  the block matrix $C=(C_{a, b})_{a, b\in F}$
with entries from $\mathbb{M}_J$  is positive, then, for each  $v\in\mathbb{N}^{\vg}$, the matrix  $\widetilde{C}=(\widetilde{C}_{a, b})_{a, b\in F}$
 defined by
\[
  \widetilde{C}_{a, b}=\begin{cases}
    C_{a-v, b-v}, \hspace{0.3 cm} a, b\ge v \\
    \mathbf{0}, \hspace{1.1 cm} \text{otherwise},
 \end{cases}
\]
is positive. 
\end{lemma}

\begin{remark}\rm
 The statement of Lemma~\ref{shifted} has used the assumption $F$ is coinvariant to guarantee 
  that $a,b\ge v$ implies $a-v,b-v\in F.$
\end{remark}

\begin{proof}
  Since $C$ is positive, there exists $M\ge 1$ and $G: F\to \mathbb{M}_{J\times M}$ such that $C_{a, b}=G(a)G(b)^*$ for all $a, b\in F.$ Define $\widetilde{G}: F\to \mathbb{M}_{J\times M}$ by setting $\widetilde{G}(a)=G(a-v)$ if $a\ge v$ and $G(a)=\mathbf{0}$ otherwise. Then, $\widetilde{C}_{a, b}=\widetilde{G}(a)\widetilde{G}(b)^*,$ for all $a, b\in F$, which implies that $\widetilde{C}$ is positive. 
\end{proof}

\begin{proof}[Proof of Theorem~\ref{ShimtoCCP}]
 By assumption, there exist power series $t,g,h$ satisfying the conditions of equations~\eqref{kfact} and~\eqref{lfact}.
We will prove that $(k,\ell)$ has the CC property as per Proposition~\ref{convenient}. Suppose $J\in\mathbb{N},$  the set  $F\subseteq\mathbb{N}^{\vg}$ is co-invariant,  the collection of coefficients $\{c_a  : a\in F\}\subseteq \mathbb{M}_J$ is such that the block matrix $C=(c_{a, b})_{a, b\in F}$ defined as in \eqref{data} is a contraction and  $d\in\mathbb{N}^{\vg}$ satisfies
\begin{enumerate}[(i)]
    \item $d\notin F;$
    \item if $a\in\mathbb{N}^{\vg}$ and $|a|<|d|,$ then $a\in F.$
\end{enumerate}
It suffices to show 
 that the matrix $C^+$ indexed by $F^+\times F^+$ (where $F^+=(F\cup\{d\})\setminus\{\mathbf{0}\}$) and defined as in \eqref{onepoint} is a contraction.

Now, $C$ being a contraction implies that $X:=I-C^*C$ is positive. Letting $I_J$ denote the $J\times J$ identity matrix, we obtain, for every $a, b\in F,$
\[X_{a, b}=\delta_{a, b}I_J-(C^*C)_{a, b} \]
\[=\delta_{a, b}I_J-\sum_{u\in F}(C^*)_{a, u}c_{u, b} \]
\[=\delta_{a, b}I_J-\sum_{u\in F}(c_{u, a})^*c_{u, b} \]
\[=\delta_{a, b}I_J-\sum_{u\in F}c^*_{a-u}c_{b-u}\frac{k_u}{\sqrt{\ell_a\ell_b}} \]
\[=\delta_{a, b}I_J-\sum_{u\le a,b }c^*_{a-u}c_{b-u}\frac{k_u}{\sqrt{\ell_a\ell_b}}, \]
where the last equality holds because $F$ is co-invariant. Taking, for convenience, the Schur product of $X$ with the dyad $(\sqrt{\ell_a}\sqrt{\ell_b})_{a, b\in F},$ we obtain that the matrix with entries 
\begin{equation}\label{X}
    \sqrt{\ell_a\ell_b}X_{a, b}=\ell_a\delta_{a, b}I_J-\sum_{u\le a,b }k_uc^*_{a-u}c_{b-u},
    \end{equation}
$a, b\in F,$ is positive. Our goal now is to show that $X^+:=I-(C^+)^*C^+$ is positive. Arguing as before, $X^+$ is positive if and only if the matrix with entries 
\[\sqrt{\ell_a\ell_b}X^+_{a, b}=\ell_a\delta_{a, b}I_J-\sum_{u\in F^+ }k_uc^*_{a-u}c_{b-u} \]
\begin{equation}\label{Xplus}
    =\ell_a\delta_{a, b}I_J-\sum_{0<u\le a,b }k_uc^*_{a-u}c_{b-u},
    \end{equation}
$a, b\in F^+$, is positive. In view of \eqref{kfact}, \eqref{lfact}, 
\[ \begin{split}
&\sqrt{\ell_a \ell_b} X^+_{a, b}
\\ &=g_a\delta_{a, b}I_J+\sum_{0<v\le a}t_v\ell_{a-v}\delta_{a, b}I_J-\sum_{0<u\le a, b}\sum_{0<v\le u}t_vk_{u-v}c^*_{a-u}c_{b-u}+\sum_{0<u\le a, b}h_uc^*_{a-u}c_{b-u} 
\end{split}\]
\[=g_a\delta_{a, b}I_J+\sum_{0<v\le a}t_v\ell_{a-v}\delta_{a, b}I_J-\sum_{0<v\le a, b}\sum_{v\le u\le a, b}t_vk_{u-v}c^*_{a-u}c_{b-u}+\sum_{0<u\le a, b}h_uc^*_{a-u}c_{b-u} \]
\[=g_a\delta_{a, b}I_J+\sum_{0<u\le a, b}h_uc^*_{a-u}c_{b-u}+\sum_{0<v\le a}t_v\Big(\ell_{a-v}\delta_{a, b}I_J-\sum_{v\le u\le a, b}k_{u-v}c^*_{a-u}c_{b-u}\Big) \]
\[=g_a\delta_{a, b}I_J+\sum_{0<u\le a, b}h_uc^*_{a-u}c_{b-u}+\sum_{0<v\le a}t_v\Big(\ell_{a-v}\delta_{a, b}I_J-\sum_{w\le a-v, b-v}k_{w}c^*_{a-v-w}c_{b-v-w}\Big) \]
\[=g_a\delta_{a, b}I_J+\sum_{0<u\le a, b}h_uc^*_{a-u}c_{b-u}+\sum_{0<v\le a}t_v\sqrt{\ell_{a-v}\ell_{b-v}}X_{a-v, b-v}, \]
for all $a, b\in F^+,$ where for the last equality we have used 
\eqref{X} (note that, since $v>0,$ both $a-v, b-v\in F$ assuming they are defined).  Thus, we obtain 
\begin{equation}\label{decomp}
X^+=G+H+\sum_{v\in F^+}t_v X_v,    
\end{equation}
where the $\mathbb{M}_J$-block matrices $G=(G_{a, b}), H=(H_{a, b}),$ and $X_v=\big((X_v)_{a, b}\big)$ are indexed by $F^+\times F^+$ and defined as follows:
\begin{enumerate}[(i)]
    \item  $G_{a, b}=\frac{g_a}{\ell_a}\delta_{a, b}I_J$;
    \item  $H_{a, b}=\sum_{0<u\le a, b}\frac{h_u}{\sqrt{\ell_a\ell_b}}c^*_{a-u}c_{b-u}$;
    \item $(X_v)_{a, b}=X_{a-v, b-v}$, for any $v\in F^+,$ where $X_{a-v, b-v}=\mathbf{0}$  if either $v\nleq a$ or $v\nleq b$,
\end{enumerate}
for all $a, b\in F^+.$ Clearly, $G$ is positive, while a short computation reveals that $H=(C_h)^*C_h,$ where 
\[
   (C_h)_{a, b}=\frac{\sqrt{h_a}}{\sqrt{\ell_b}}c_{b-a}, 
\]
for all $a, b\in F^+.$ Thus, $H$ must be positive as well. Finally, every $X_v$ is positive because of Lemma~\ref{shifted} and the fact that $X$ was positive to begin with. Hence, \eqref{decomp} allows us to conclude that $X^+$ is positive, as desired. 
\end{proof}

\subsection{Existence of a master certificate and a necessary condition for the CP property}
\label{yesmaster}
Let $k$ be a normalized diagonal holomorphic kernel with coefficients $\{k_a\}_{a\in\mathbb{N}^{\vg}}.$ We recall Definition~\ref{mastercert} from the introduction: the \df{master certificate associated with $k$} is the formal power series in $\vg$ complex variables defined as 
\[
 \mt(x)=\sum_{b\in\mathbb{N}^{\mathsf{g}}}\mt_bx^b,
\]
where $\mt_{\mathbf{0}}=0$, $\mt_{e_j}=k_{e_j}$ for all $j$ (recall that $e_j\in \NN^\vg$ denotes the element with $1$ in the $j$-th entry and $0$
 elsewhere) and
\begin{equation}
\label{def:mtb}
  \mt_b=\max\Bigg\{0,k_b-\sum_{\substack{w+u=b, \\ w, u\neq \mathbf{0}}}\mt_w k_u\Bigg\},
\end{equation}
for all $b\in\NN^{\vg}$ with $|b|>1.$
 In this subsection, we show that, for any diagonal holomorphic $\ell,$ if $(k, \ell)$ has the CC property, then $\mt$ is a formal Shimorin certificate for $(k, \ell).$ Thus, in combination with our results from the previous subsection, we  obtain a complete characterization of those pairs of diagonal holomorphic kernels possessing the CC property. Recall the standing assumption that $k$ and $\ell$ are normalized. 
             \par 
   Two things are immediate from the construction of $\mt$. First, $k_b\ge \mt_b\ge 0$ for all $b$  and  second
\begin{equation}
\label{e:why-t-works}
 k_b-\sum_{\substack{w+u=b, \\ w\neq \mathbf{0}}}  k_u \mt_w\le 0
\end{equation}
 for all $\mathbf{0}\ne b\in \NN^\vg,$ so that, as a formal power series, the coefficients of $h=1-k(1-\mt)$ are all  non-negative. These observations are
 summarized in the following theorem.
 
 \begin{theorem}
 \label{1-k(1-mt)}
  The master certificate $\theta$ associated to  a normalized diagonal holomorphic kernel $k$
  satisfies $k_b\ge \mt_b$ for all $b$ and the coefficients of the formal power series
  $h=1-k(1-\mt)$ are all non-negative.
 \end{theorem}
  
   \par

\begin{theorem}
\label{almostCCPtoShim}
 Suppose $\ell$ is a normalized diagonal holomorphic kernel.  If $(k,\ell)$ 
 has the CC property, then the coefficients of the formal power series $g=\ell(1-\mt)$ are all non-negative;
  that is,  $h_0=0$ and 
\[
g_a=\ell_a - \sum_{\substack{u\le a, \\ u\neq \mathbf{0}}}\mt_u \ell_{a-u} \ge 0
\]
 for $ \mathbf{0} \ne a\in \NN^\vg.$
 \end{theorem}

The proof of  Theorem~\ref{almostCCPtoShim} is postponed in favor of 
the following key result.

\begin{theorem}
 \label{main:formal:CCP}
   A pair of diagonal holomorphic kernels $(k,\ell)$ has the CC property if and only if
   the master certificate $\mt$ for $k$ is a formal Shimorin certificate for $(k,\ell).$
 \end{theorem}
 
 \begin{proof}
   Suppose $(k,\ell)$ has the CP property. By Theorems~\ref{1-k(1-mt)} and \ref{almostCCPtoShim},
   the formal power series $h=1-k(1-\mt)$ and $g=\ell(1-\mt)$ have non-negative coefficients.  
   Hence, $\mt$ is a formal certificate for $(k,\ell).$ 
   
   The converse is Theorem~\ref{ShimtoCCP}.
 \end{proof}

The proof of  Theorem~\ref{almostCCPtoShim} consumes the remainder
of this subsection. It 
 uses the Lemmas~\ref{p:multi:need-} and~\ref{l:someare0}
 immediately below.

\begin{lemma}
\label{p:multi:need-}
Suppose $(k, \ell)$ is a normalized diagonal holomorphic CC pair. Suppose further, $\varnothing \ne F\subseteq \mathbb{N}^{\vg}$ is finite,   $d\in\mathbb{N}^{\vg}\setminus F$ and both $F$ and  $F^+=F\cup\{d\}$ are co-invariant. 
If  $\{v_a: a\in F \}$ is a set of non-negative real numbers such that
\begin{equation}
\label{e:multi:necu-}
 \ell_a \ge \sum_{u\le a} v_{a-u} k_{u} \ge 0,
\end{equation}
 for each $a\in F,$   then for all $a\in F^+,$  
\begin{equation}
 \label{e:multi:necu+}
 \ell_{a} \ge \sum_{u< a} v_{u} k_{a-u} = \sum_{0<u\le a} v_{a-u}k_u.
\end{equation}
\end{lemma}

\begin{proof} 
  Let $c_a$ denote the $F\times F$ matrix  with 
 $(a,\mathbf{0})$ entry $\sqrt{v_{a}}$ and also let $E_{1, 1}$ denote the $F\times F$ matrix with $1$ in the $(\mathbf{0}, \mathbf{0})$ entry and $0$ everywhere else. Note 
\[
 c_a^*c_b =\begin{cases}  \mathbf{0} & \mbox{ if } a\ne b \\
                         v_aE_{1, 1} & \mbox{ if } a=b. \end{cases}
\]
  The conditions of equation~\eqref{e:multi:necu-}
 imply the corresponding  $F\times F$ block matrix $C=C(\{c_a\})$ as in \eqref{e:Cab} is 
 a contraction, since
\[
 (C^*C)_{a,b} = \sum_{u\le a,b} c_{a-u}^* c_{b-u} \frac{k_u}{\sqrt{\ell_a\ell_b}}
  = \begin{cases}  \Big(\sum_{u \le a} v_{a-u} \frac{k_u}{\ell_a}\Big)E_{1,1}  & \hspace{0.12 cm} a=b 
    \\ \mathbf{0} & \mbox{ otherwise}.  \end{cases}
\]

In view of the discussion preceding the proof of Proposition~\ref{convenient},  the $F^+\times F^+$ matrix $C^{+}$ as in \eqref{anothonepoint} is a contraction.
 Moreover, for  $a,b\in F^+$ 
\[
\begin{split}
 ((C^+)^* \, C^+)_{a,b} & = 
 \begin{cases} \Big(\sum \{ c_{a-u}^* c_{a-u} \frac{k_{u}}{\ell_{a}}  \mid F_N^+ \ni u \le a \}\Big)E_{1, 1}
     & \mbox{ if } a =b  
  \\ \mathbf{0} & \mbox{ otherwise}.  \end{cases}
 \\ & = \begin{cases}  \Big(\sum_{u<a}  v_{u} \frac{k_{a-u}}{\ell_{a}}\Big)E_{1, 1} & \mbox{ if } a= b
    \\ \mathbf{0} & \mbox{ otherwise}, \end{cases}
\end{split}
\]
 from which the inequalities of equation~\eqref{e:multi:necu+} follow.
\end{proof}


\begin{lemma}
 \label{l:someare0}
 Given normalized diagonal holomorphic $k,\ell$ in $\vg$ variables, 
 $N\in\NN$ and  $d\in \mathbb{N}^{\vg}$ with $|d|=N+1$ and $S\subseteq \{a\le d\}$, define  $v_a$ for $a\le d$
 recursively as follows. Let  $v_{\mathbf{0}}=0$ if $0\in S,$ and  $v_{\mathbf{0}}=1$  
 if $0\notin S$ and, assuming $0\le M\le N$ and $v_u$ have been defined for $|u|\le M$ and $u\le d,$  let
\[
  v_a =\begin{cases}  0 &  a\in S
   \\ \ell_a -  \sum_{u< a} v_u k_{a-u}  & a\notin S\end{cases}
\]
for  $|a|=M+1$ and $a\le d.$ If $(k,\ell)$ is a CC pair, then 
\begin{enumerate}[(i)]
 \item \label{i:someare0:1} $v_a\ge 0$ for all $a\le d,$ and
 \item \label{i:someare0:2}  $\ell_a - \sum_{u\le a} v_uk_{a-u} \ge 0$ for all $a\le d.$
\end{enumerate}
\end{lemma}

\begin{proof} Let $d, S$ and $v_a$ be as above. We will proceed by induction on $|a|.$ By definition, $v_{\mathbf{0}}\ge 0$ and $\ell_{\mathbf{0}}-v_{\mathbf{0}}k_{\mathbf{0}}=1-v_{\mathbf{0}}\ge 0,$ so the result holds if $|a|=0.$ Now, let $0\le M\le N$ and assume that the conditions of both item~\eqref{i:someare0:1} and item~\eqref{i:someare0:2} hold whenever $|a|\le M$ and $a\le d.$ We will show that they continue to hold if $|a|=M+1$ and $a\le d$. \par 
Fix an arbitrary $a\in\mathbb{N}^{\vg}$ with $|a|=M+1$ and $a\le d.$
Set $F=\{u\in\mathbb{N}^{\vg} : |u|\le M \text{ and } u\le d\}$. Clearly, both $F$ and $F\cup\{a\}$ are co-invariant. Set $F^+=(F\cup\{a\})\setminus\{\mathbf{0}\}$. The collection of non-negative numbers $\{v_u : u\in F\}$ satisfies \eqref{e:multi:necu-} by our inductive hypothesis. By Lemma~\ref{p:multi:need-}, it follows that 
\[
 \ell_a -\sum_{u<a} v_u k_{a-u} \ge 0.
\]
 Hence $v_a\ge 0$ so that  \eqref{i:someare0:1} holds for $a.$ Moreover, 
 if $a\notin S,$ then $v_a=\ell_a -\sum_{u<a} v_u k_{a-u}$ and
\[
 \ell_a - \sum_{u\le a} v_u k_{a-u} = \ell_a- \sum_{u<a} v_u k_{a-u} - v_a =0.
\]
On the other hand, if $a\in S$, then $v_a=0,$ in which case
\[
  \ell_a - \sum_{u\le a} v_u k_{a-u} = \ell_a- \sum_{u<a} v_u k_{a-u} \ge 0.
\] 
Thus, in either case, the inequality of item~\eqref{i:someare0:2} holds for $a,$ completing
 a proof by induction.
\end{proof}

\begin{proof}[Proof of Theorem~\ref{almostCCPtoShim}]
Our goal is to show 
\begin{equation}\label{dasgoal}
\ell_d - 
  \sum_{\mathbf{0}<u\le d} \mt_u \ell_{d-u} \ge 0
\end{equation}
for all non-zero $d\in\mathbb{N}^{\vg}.$ 
\par
 Assume first that $|d|=1$. 
 Choosing $F=\{\mathbf{0}\}$, $v_{\mathbf{0}}=1$ and $d=e_j$ in Lemma~\ref{p:multi:need-}, 
 it follows that 
\[
 \ell_{e_j} \ge k_{e_j} = \mt_{e_j}.
\]
 Thus,
\[
 (\ell(1-\mt))_{e_j} = \ell_{e_j} - \mt_{e_j} \ge 0
\]
 for $1\le j\le \vg.$  \par 
Now, assume $|d|\ge 2.$ Put $S=\{a\le d: \mt_{d-a}=0\}$  and define $v_a$ recursively as in Lemma~\ref{l:someare0}, for every $a\le d.$ We have $v_a\ge 0$ by construction. Set 
\[\alpha_a=\begin{cases} 0,  \hspace{0.6 cm} \mt_{d-a}=0\\
1, \hspace{0.6 cm}  \mt_{d-a}>0.
\end{cases}
\]
Thus, $v_{\mathbf{0}}=\alpha_{\mathbf{0}}$  and $\alpha_a \mt_{d-a}=\mt_{d-a}$  as well as,
\[
 v_a=\alpha_a\bigg(\ell_a-\sum_{u<a}v_uk_{a-u}\bigg),
\]
for all $|a|\ge 1$. In particular, $v_{\mathbf{0}}\theta_d=\theta_d.$ Moreover, 
\[
  \mt_{d-a}=\alpha_a\bigg(k_{d-a}-\sum_{\mathbf{0}<u<d-a}k_u\mt_{d-a-u} \bigg) 
\]
 assuming $|d-a|\ge 2$, while  $\mt_{e_j}=a_{d-e_j}k_{e_j}$ for all $j$, since $\theta_{e_j}=k_{e_j}>0$ and hence $a_{d-e_j}=1$ (recall also that $\mt_{\mathbf{0}}=0=a_d\cdot 0$). \par 
To prove \eqref{dasgoal}, we will show,
 by induction, that for all $d,$
\begin{equation}\label{interr}
 \sum_{v<d} v_a k_{d-a} = \sum_{\mathbf{0}<u\le d} \theta_u \ell_{d-u}.
\end{equation} 
First, observe that 
\begin{align} \label{neceq1}
\sum_{a<d}v_ak_{d-a}  =&\sum_{a+e_j=d}\alpha_a\bigg(\ell_a-\sum_{b<a}v_bk_{a-b} \bigg)k_{e_j}+\sum_{\substack{b<d, \\ |b|\le |d|-2}}v_bk_{d-b} \notag \\ 
=&\sum_{a+e_j=d}\alpha_a\ell_ak_{e_j}-\sum_{a+e_j=d}\sum_{b<a}\alpha_av_bk_{a-b}k_{e_j}+\sum_{\substack{b<d, \\ |b|\le |d|-3}}v_bk_{d-b}+\sum_{\substack{b<d, \\ |b|=|d|-2}}v_bk_{d-b} \notag \\ 
=&\sum_{a+e_j=d}\mt_{e_j}\ell_a-\sum_{a+e_j=d}\sum_{b<a}v_bk_{d-b-e_j}\mt_{e_j}+\sum_{\substack{b<d,\\ |b|\le |d|-3}}v_b k_{d-b}+\sum_{\substack{b<d, \\ |b|=|d|-2}}v_bk_{d-b} \notag \\ 
=&\sum_{a+e_j=d}\mt_{e_j}\ell_a-\sum_{a+e_j=d}\sum_{b+e_i=a}v_bk_{e_i}\mt_{e_j} -\sum_{a+e_j=d}\sum_{\substack{b<a, \\ |b|\le |a|-2=|d|-3}}v_bk_{d-b-e_j}\mt_{e_j}\notag \\ &\qquad+\sum_{\substack{b<d, \\ |b|\le |d|-3}}v_b k_{d-b}+\sum_{\substack{b<d, \\ |b|=|d|-2}}v_bk_{d-b}.
\end{align}
Summing the second and fifth terms on the right hand side in  \eqref{neceq1} gives
\begin{align} \label{neceq2}
&\sum_{\substack{b<d,\\ |b|=|d|-2}}v_bk_{d-b}-\sum_{a+e_j=d}\sum_{b+e_i=a}v_b k_{e_i}\mt_{e_j} \notag \\
=&\sum_{\substack{b<d, \\|b|=|d|-2}}v_bk_{d-b}-\sum_{\substack{b<d, \\ |b|=|d|-2}}v_b\sum_{\substack{b+e_i+e_j=\vg, \\ 1\le i, j\le \vg}}k_{e_i}\mt_{e_j} \notag \\
=&\sum_{\substack{b<d, \\ |b|=|d|-2}}v_b\bigg(k_{d-b}-\sum_{\substack{e_i+e_j=d-b, \\ 1\le i, j\le d}}k_{e_i}\mt_{e_j}\bigg) \notag \\
=&\sum_{\substack{b<d, \\ |b|=|d|-2}}\alpha_bv_b\bigg(k_{d-b}-\sum_{\substack{e_i+e_j=d-b, \\ 1\le i, j\le \vg}}k_{e_i}\mt_{e_j}\bigg) \notag \\
=&\sum_{\substack{b<d, \\ |b|=|d|-2}}v_b\mt_{d-b},
\end{align}
and summing the third and fourth terms on the right hand side in \eqref{neceq1} gives 
\begin{align} \label{neceq3}
&\sum_{\substack{b<d, \\ |b|\le |d|-3}}v_bk_{d-b}-\sum_{a+e_j=d}\sum_{\substack{b<a, \\ |b|\le |d|-3}}v_bk_{d-b-e_j}\mt_{e_j}  \notag \\
=&\sum_{\substack{b<d, \\ |b|\le |d|-3}}v_bk_{d-b}-\sum_{\substack{b<d,\\ |b|\le |d|-3}}\sum_{\substack{d-b-e_j>\mathbf{0},\\ 1\le i\le \vg }}v_bk_{d-b-e_j}\mt_{e_j} \notag \\
=&\sum_{\substack{b<d, \\ |b|\le |d|-3}}v_b\bigg(k_{d-b}-\sum_{\substack{d-b-e_j>\mathbf{0},\\ 1\le i\le \vg}}k_{d-b-e_j}\mt_{e_j}\bigg)
\notag \\ 
=&\sum_{\substack{b<d, \\ |b|\le |d|-3}}v_b\bigg(k_{d-b}-\sum_{\substack{d-b-w>\mathbf{0},\\ |w|=1}}k_{d-b-w}\mt_{w}\bigg).
\end{align}
Combining \eqref{neceq1}-\eqref{neceq2}-\eqref{neceq3}, we obtain
\[\sum_{a<d}v_ak_{d-a}\]
\begin{equation}\label{neceq4}
=\sum_{a+e_j=d}\mt_{e_j}\ell_a+\sum_{\substack{b<d, \\ |b|=|d|-2}}v_b\mt_{d-b}+\sum_{\substack{b<d,\\ |b|\le |d|-3}}v_b\bigg(k_{d-b}-\sum_{\substack{|w|=1, \\ w<d-b}}k_{d-b-w}\mt_w \bigg).
\end{equation}
Notice that if $|d|=2,$ then \eqref{neceq4} becomes (recall that $k_{\mathbf{0}}=\ell_{\mathbf{0}}=1$)
\begin{equation} \label{neceq5}
 \sum_{a<d}v_ak_{d-a}=\sum_{a+e_j=d}\mt_{e_j}\ell_a =\sum_{\mathbf{0}<u\le d}\mt_u \ell_{d-u}.
\end{equation}
Now, assume $|d|\ge 3.$ In order to arrive at $\eqref{interr}$, we will show that, for every $2\le m\le |d|-1$,  
\begin{align} \label{neceq6}
\sum_{a<d}v_ak_{d-a}=&\sum_{\substack{u<d,\\ 0<|u|\le m-1}}\mt_u\ell_{d-u}+\sum_{\substack{b<d,\\ |b|=|d|-m}}v_b\mt_{d-b} \notag\\ \quad &+\sum_{\substack{b<d, \\ |b|\le |d|-1-m}} v_b\bigg(k_{d-b}-\sum_{\substack{u< d-b, \\ 0<|u|<m}}\mt_uk_{d-b-u} \bigg).
\end{align}
In view of \eqref{neceq4},  we see that \eqref{neceq6} holds for $m=2.$ If $|d|=3,$ we are done. Thus we further assume $|d|\ge 4$.
We move to the inductive step. Accordingly, suppose  \eqref{neceq6} holds for some fixed $2\le m\le |d|-2$. We will show that \eqref{neceq6}  also holds for $m+1.$ First, observe that  
  \begin{align}\label{neceq7}   
 \sum_{\substack{b<d, \\ |b|=|d|-m}}v_b\mt_{d-b} 
 &=\sum_{\substack{b<d,\\ |b|=|d|-m}}\mt_{d-b}\alpha_b\bigg(\ell_b-\sum_{u<b}v_uk_{b-u} \bigg) \notag \\
&=\sum_{\substack{b<d,\\ |b|=|d|-m}}\alpha_b\mt_{d-b}\ell_b-\sum_{\substack{b<d,\\ |b|=|d|-m}}\sum_{u<b}\alpha_b\mt_{d-b}v_uk_{b-u} \notag \\
&=\sum_{\substack{b<d,\\ |b|=|d|-m}}\mt_{d-b}\ell_b-\sum_{\substack{b<d,\\ |b|=|d|-m}}\sum_{u<b}\mt_{d-b}v_uk_{b-u} \notag \\ 
&=\sum_{\substack{u<d,\\ |u|=m}}\mt_{u}\ell_{d-u}-\sum_{\substack{w<d,\\ |w|=m}}\sum_{u<d-w}\mt_{w}v_uk_{d-w-u} \notag \\ 
&=\sum_{\substack{u<d,\\ |u|=m}}\mt_{u}\ell_{d-u}-\sum_{\substack{u<d, \\ |u|\le |d|-1-m}}\sum_{\substack{w<d-u, \\ |w|=m}}\mt_{w}v_uk_{d-w-u}.
\end{align}  
Combining \eqref{neceq7} with \eqref{neceq6} (which holds for $m$ by our inductive hypothesis) yields 
\begin{align} \label{neceq8}
 \sum_{a<d}v_ak_{d-a}  
=&\sum_{\substack{u<d,\\ 0<|u|\le m-1}}\mt_u\ell_{d-u}+\sum_{\substack{u<d, \\ |u|=m}}\mt_u\ell_{d-u}-\sum_{\substack{u<d, \\ |u|\le |d|-1-m}}\sum_{\substack{w<d-u, \\ |w|=m}}v_u\mt_wk_{d-w-u} \notag \\
\quad &+\sum_{\substack{b<d,\\ |b|\le |d|-1-m}}v_b\bigg(k_{d-b}-\sum_{\substack{0<|u|<m, \\ u<d-b}}\mt_u k_{d-b-u}\bigg)  \notag \\
=&\sum_{\substack{u<d, \\ 0<|u|\le m}}\mt_u\ell_{d-u}-\sum_{\substack{u<d, \\ |u|\le |d|-1-m}}\sum_{\substack{w<d-u, \\ |w|=m}}v_u\mt_wk_{d-w-u} \notag \\ & +\sum_{\substack{b<d,\\ |b|\le |d|-2-m}}v_b\bigg(k_{d-b}-\sum_{\substack{u<d-b,\\ 0<|u|<m+1}}\mt_uk_{d-b-u} \bigg) \notag \\ 
&+\sum_{\substack{b<d, \\ |b|=|d|-1-m}}v_b\bigg(k_{d-b}-\sum_{\substack{u<d-b, \\ 0<|u|<m}}\mt_uk_{d-b-u} \bigg)
\notag \\ &+\sum_{\substack{b<d, \\ |b|\le |d|-2-m}}\sum_{\substack{u<d-b, \\ |u|=m}}v_b\mt_uk_{d-b-u}.
\end{align}  But since $|b|=|d|-1-m$ and $u<d-b$ imply $|u|<m+1,$ we have
\begin{align} 
&\sum_{\substack{b<d, \\ |b|=|d|-1-m}}v_b\bigg(k_{d-b}-\sum_{\substack{u<d-b, \\ 0<|u|<m}}\mt_uk_{d-b-u} \bigg) \notag \\  \pagebreak 
=&\sum_{\substack{b<d, \\ |b|=|d|-1-m}}v_b\bigg(k_{d-b}-\sum_{\mathbf{0}<u<d-b}\mt_uk_{d-b-u}\bigg)+\sum_{\substack{b<d, \\ |b|=|d|-1-m}}\sum_{\substack{u<d-b, \\ |u|=m}}v_b\mt_uk_{d-b-u} \notag \\
=&\sum_{\substack{b<d, \\ |b|=|d|-1-m}}v_b\mt_{d-b}+\sum_{\substack{b<d,\\ |b|=|d|-1-m}}\sum_{\substack{u<d-b, \\ |u|=m}}v_b\mt_uk_{d-b-u}.\notag
\end{align}
Thus, \eqref{neceq8} becomes
\begin{align}  
& \sum_{a<d}v_ak_{d-a}  \notag \\ 
=&\sum_{\substack{u<d, \\ 0<|u|\le m}}\mt_u\ell_{d-u}+\sum_{\substack{b<d, \\ |b|=|d|-1-m}}v_b\mt_{d-b}-\sum_{\substack{u<d, \\ |u|\le |d|-1-m}}\sum_{\substack{w<d-u, \\ |w|=m}}v_u\mt_wk_{d-w-u}\notag \\ &+\sum_{\substack{b<d,\\ |b|\le |d|-2-m}}v_b\bigg(k_{d-b}-\sum_{\substack{u<d-b,\\ 0<|u|<m+1}}\mt_uk_{d-b-u} \bigg)+\sum_{\substack{b<d,\\ |b|=|d|-1-m}}\sum_{\substack{u<d-b, \\ |u|=m}}v_b\mt_uk_{d-b-u}
\notag \\ &+\sum_{\substack{b<d, \\ |b|\le |d|-2-m}}\sum_{\substack{u<d-b, \\ |u|=m}}v_b\mt_uk_{d-b-u} \notag \\
=&\sum_{\substack{u<d,\\ 0<|u|\le m}}\mt_u\ell_{d-u}+\sum_{\substack{b<d,\\ |b|=|d|-1-m}}v_b\mt_{d-b} \notag \\ &+\sum_{\substack{b<d, \\ |b|\le |d|-2-m}} v_b\bigg(k_{d-b}-\sum_{\substack{u< d-b, \\ 0<|u|<m+1}}\mt_uk_{d-b-u} \bigg) \notag,
\end{align}
 concluding the proof of the inductive step. Thus, \eqref{neceq6} holds for every $2\le m\le |d|-1.$ Choosing $m=|d|-1$ then yields (as the 3rd term disappears)
\begin{align}
\sum_{a<d}v_ak_{d-a}&=\sum_{\mathbf{0}<u<d}\mt_u\ell_{d-u}+\sum_{|b|=0}v_b\mt_{d-b} \notag \\ 
&=\sum_{\mathbf{0}<u<d}\mt_u\ell_{d-u}+v_{\mathbf{0}}\mt_d \notag \\ &=\sum_{\mathbf{0}<u<d}\mt_u\ell_{d-u}+\mt_d \notag \\ 
&=\sum_{\mathbf{0}<u\le d}\mt_u\ell_{d-u},\notag 
\end{align}
and $\eqref{interr}$ is proved. We may now conclude
\[\ell_d - \sum_{\substack{u\le d, \\ u\neq \mathbf{0}}}\mt_u \ell_{d-u}=\ell_d-\sum_{a<d}v_ak_{d-a}=v_d\ge 0, 
\]
so \eqref{dasgoal} holds.
\end{proof}

\subsection{CP implies CC} \label{directCPgivesCC}

The last ingredient that will be needed for our proof of Theorem~\ref{introextmain} (contained in Section~\ref{mainmain}) is a direct passage from the CP to the CC property. Establishing such a passage will be our main objective for this subsection.  The idea is to apply the complete Pick pair assumption to
tuples of points near $\mathbf{0}$ and then take a limit - letting these points tend to $\mathbf{0}$.  Before proceeding we highlight an ingredient
in the proof, Proposition~\ref{autopolyn}.  Namely, 
for a $(k,\ell)$ pair with $k$ normalized, multiplication by a monomial $z^b$ defines a bounded multiplier from $\mathcal{H}_k$ to $\mathcal{H_{\ell}}$ 
and moreover, as pointed out by an anonymous referee, in fact the mutiplier norm of $z^b$ coincides with its Hilbert space norm in $\mathcal{H}_\ell.$
  \par 
We will require the following elementary Hilbert space lemma.
\begin{lemma}
\label{l:Pconverge:1}
 Let $H$ denote a Hilbert space and assume $M\subset H$ is a non-trivial finite-dimensional subspace. 
 Further, suppose $\gamma:[0,\delta]\to H$ is a continuous function  such that $\gamma(t)$ does not lie in $M$, for any $t.$  Then, letting 
  $P_t$ denote the projection onto $\textup{span}\big(M\cup\{\gamma(t)\}\big)$, we have $P_t\to P_0$ in operator norm as $t\to 0.$
\end{lemma}
\begin{proof}[Proof sketch]
Fix an orthonormal basis  $\{v_1, \dots, v_m\}$ for $M$ and apply Gram-Schmidt to the basis
$\{v_1,\dots,v_m,\gamma(t)\}$ to obtain the orthonormal basis $\{v_1,\dots,v_m,\rho(t)\}$ 
for $\textup{span}\big(M\cup\{\gamma(t)\}\big)$, noting that $\rho(t)$ depends continuously on $t.$
Since, for $f\in H,$ 
\[
 (P_t-P_0)f = \langle f,\rho(t)\rangle (\rho(t)-\rho(0)) + \langle f, \rho(t)-\rho(0)\rangle \rho(0)
\]
 it follows that $\|P_t-P_0\|\le 2 \|\rho(t)-\rho(0)\|.$
\end{proof}

We will work with a lexicographic order on $\mathbb{N}^{\vg}$. Given $a, b\in\mathbb{N}^{\vg},$ the expression $a\prec b$ means that either $|a|<|b|$ or $|a|=|b|$ and there exists $0\le r\le \vg-1$ such that 
\[a_1=b_1\]
\[\vdots \]
\[a_rr=b_r \]
\[a_{r+1}<b_{r+1},\]
where $r=0$ simply means $a_1<b_1.$ Clearly, $(\mathbb{N}^{\vg}, \prec)$ is totally ordered. Set $\mathbb{N}^{\vg}=\{a^0, a^1, \dots\}$ accordingly.

Now, let $\ell=1+\sum_{a\in\mathbb{N}^{\vg}, \\ |a|>0}\ell_a(z\overline{w})^a$ be a normalized diagonal holomorphic kernel on some domain $\Omega\subseteq\mathbb{C}^{\vg}$ containing $\mathbf{0}$. Set $p_{a^j}(z)=z^{a^j},$ for all $j$. For $m\ge 0$ and $\Lambda \subseteq\Omega\setminus\{\mathbf{0}\}$, let $M_{m,\Lambda}$ and $M_{\Lambda}$
denote the spans of $\{p_{a^j}:0\le j\le m\}\cup \{\ell_{\lambda}: \lambda\in\Lambda\} $ and $\{\ell_{\lambda}: \lambda\in\Lambda\}$, respectively, with 
  $P_{m,\Lambda}$ and $P_{\Lambda}$ the associated projections. We will also write $M^0_{m,\Lambda}$ and $M^0_{\Lambda}$  for the spans of $\{p_{a^j}:1\le j\le m+1\}\cup \{\ell_{\lambda}-1: \lambda\in\Lambda\}$ and  $\{\ell_{\lambda}-1: \lambda\in\Lambda\}$, respectively,  with $P^0_{m, \Lambda}$ and $P^0_{\Lambda}$ denoting the associated projections.
 For what is to follow, it should be kept in mind that, since all coefficients $\ell_a$ are non-zero, the polynomials are  contained in $\mathcal{H}_{\ell}$. Hence, there exist no (finite) collections of linearly dependent kernel functions, and this continues to be the case even if we have vanishing derivatives up to a certain order.

 \begin{lemma}
 \label{l:Pconverge:2} Fix $m\ge -1$. With notation as above, for any normalized diagonal holomorphic kernel $\ell$  and any finite $\Lambda \subseteq\Omega\setminus\{\mathbf{0}\},$ there exists a continuous map $v: (0, \delta] \to \Omega$ such that, with $\Lambda_t=\Lambda\cup \{v(t)\}$, 
 we have \[P_{m,\Lambda_t}\to P_{m+1,\Lambda}\hspace{0.3 cm} \text{ and }\hspace{0.3 cm} P^0_{m,\Lambda_t}\to P^0_{m+1,\Lambda}\] in operator norm as $t\to 0$, where $P_{-1, \Lambda_t}\equiv P_{\Lambda_t}$ and $P^0_{-1, \Lambda_t}\equiv P^0_{\Lambda_t}$. 
 \end{lemma}

 \begin{proof} We assume $m\ge 0$ (the case $m=-1$ can be treated analogously).  Set $b=a^{m+1}, c=a^{m+2}$ and choose $r_j =|c|^{\vg-j} +N$, where $N> \sum (b_j+c_j) |c|^{\vg-j}$. Put $r=(r_1, \dots, r_{\vg}).$ Observe that,  for each $\rho>m+2$ and $|a^{\rho}|=|c|,$ the inequality $c\prec a^{\rho}$ yields
\[ 
  \sum c_j r_j = \sum c_j |c|^{\vg-j} \, + \, N \, |c|<  \sum a^\rho_j |c|^{\vg-j} \, + \, N \, |c|= \sum a^{\rho}_j r_j.    
\]
Our choice of $N$ guarantees that the same inequality continues to hold when $|c|<|a^{\rho}|.$ One can similarly show $\langle r, b\rangle < \langle r, c\rangle $, and thus we obtain
 \begin{equation} \label{expdiffrev}
  \langle r, b\rangle  <\langle r, c\rangle < \langle r, a^\rho \rangle, 
   \end{equation}
for all $\rho>m+2.$

 Now, fix $\delta>0$ sufficiently small and, for $0<t\le \delta,$ set  $v=v(t)=(t^{r_1}, \dots, t^{r_{\vg}})\in\Omega$. Define $\gamma: [0, \delta]\to \mathcal{H}_{\ell}$ and $\gamma^0: [0, \delta]\to \mathcal{H}_{\ell}$ as 
\[ 
\gamma(t)=\begin{cases}
\cfrac{\ell_{v}-\sum_{j=0}^m\ell_{a^j}p_{a^j}v^{a^j}}{\sqrt{\ell(v, v)-\sum_{j=0}^m\ell_{a^j}v^{2a^j}}}, \hspace{0.6 cm} 0<t\le \delta, \\ 
\sqrt{\ell_{a^{m+1}}}p_{a^{m+1}},  \hspace{2.3 cm} t=0,
\end{cases}
\]
and 
\[ 
\gamma^0(t)=\begin{cases}
\cfrac{\ell_{v}-\sum_{j=0}^{m+1}\ell_{a^j}p_{a^j}v^{a^j}}{\sqrt{\ell(v, v)-\sum_{j=0}^{m+1}\ell_{a^j}v^{2a^j}}}, \hspace{0.6 cm} 0<t\le \delta, \\ 
\sqrt{\ell_{a^{m+2}}}p_{a^{m+2}},  \hspace{2.3 cm} t=0.
\end{cases}
\]
Clearly, both $\gamma$ and $\gamma^0$ are continuous in $(0, \delta]$. We will show that they are continuous at $0$ as well. Since the proofs are essentially identical, we will only examine $\gamma$. For $0< t \le \delta,$ we have
\begin{align} \label{someratio}
\gamma(t)=\cfrac{\sum_{j=m+1}^{\infty}\ell_{a^j}p_{a^j}v^{a^j}}{\sqrt{\sum_{j=m+1}^{\infty}\ell_{a^j}v^{2a^j}}}
=\cfrac{\cfrac{\sum_{j=m+1}^{\infty}\ell_{a^j}p_{a^j}v^{a^j}}{v^{a^{m+1}}}}{\cfrac{\sqrt{\sum_{j=m+1}^{\infty}\ell_{a^j}v^{2a^j}}}{v^{a^{m+1}}}}. 
\end{align}
Now, \eqref{expdiffrev} tells us that 
\[\frac{v^{a^j}}{v^{a^{m+1}}}=\frac{t^{\langle r, a^j\rangle}}{t^{\langle r, a^{m+1}\rangle}}=t^{\langle r, a^j\rangle-\langle r, a^{m+1}\rangle}\]
has a positive exponent, for every $j>m+1.$ Thus, by elementary power series arguments, we have 
\begin{equation} \label{anotherlimit1}
  \cfrac{\sqrt{\sum_{j=m+1}^{\infty}\ell_{a^j}v^{2a^j}}}{v^{a^{m+1}}}\longrightarrow \sqrt{\ell_{a^{m+1}}} \hspace{0.3 cm} \text{ as } t\to 0.  
\end{equation} 
Further, we have
\begin{align*}
 \Bigg|\Bigg|\frac{\sum\limits_{j\ge m+1}\ell_{a^{j}} p_{a^j}v^{a^{j}}}{v^{a^{m+1}}}-\ell_{a^{m+1}}p_{a^{m+1}}    \Bigg|\Bigg|^2_{\mathcal{H}_{\ell}}=\Bigg|\Bigg|\frac{\sum\limits_{j\ge m+2}\ell_{a^{j}} p_{a^j}v^{a_{j}}}{v^{a^{m+1}}}   \Bigg|\Bigg|^2_{\mathcal{H}_{\ell}} =\frac{\sum\limits_{j\ge m+2}\ell_{a^{j}} v^{2a^{j}}}{v^{2a^{m+1}}}, \label{anotherlim}  
\end{align*}
which converges to $0$ as $t\to 0$, as we already saw. Thus, 
\begin{equation} \label{anotherlimit2}
\cfrac{\sum_{j=m+1}^{\infty}\ell_{a^j}p_{a^j}v^{a^j}}{v^{a^{m+1}}} \longrightarrow \ell_{a^{m+1}}p_{a^{m+1}} \hspace{0.3 cm} \text{ as } t\to 0 ,     
\end{equation}
where convergence is taken in the norm of $\mathcal{H}_{\ell}$. Combining \eqref{someratio} with \eqref{anotherlimit1}-\eqref{anotherlimit2}, we obtain that $\gamma$ is continuous at $0$.
   \par 
Now, let $\Lambda\subseteq\Omega\setminus \{\mathbf{0}\}$ be finite and set $\Lambda_t=\Lambda\cup \{v(t)\}.$ It is easily verified that, for all $t\in (0,\delta],$
\begin{equation*}
M_{m, \Lambda_t}=\textup{span}\big(M_{m, \Lambda}\cup\{\gamma(t)\}\big), 
\end{equation*} 
while $\textup{span}\big(M_{m, \Lambda}\cup\{\gamma(0)\}\big)=M_{m+1, \Lambda}$. Since we also have $\gamma(t)\notin M_{m, \Lambda}$ (for $\delta$ sufficiently small), Lemma \ref{l:Pconverge:1} implies $P_{m, \Lambda_t}\to P_{m+1, \Lambda}.$ Similarly, since
$M^0_{m, \Lambda_t}=\textup{span}\big(M^0_{m, \Lambda}\cup\{\gamma^0(t)\}\big)$ and $\textup{span}\big(M^0_{m, \Lambda}\cup\{\gamma^0(0)\}\big)=M^0_{m+1, \Lambda}$, we obtain $P^0_{m, \Lambda_t}\to P^0_{m+1, \Lambda},$ as desired.
 \end{proof}
 
   Let $F_n=\{a^0,\dots,a^n\}.$ In particular, $F_n$ is co-invariant in view of the lexicographic ordering.   Let \[L^n(z,w) = \sum_{a\notin F_n} \ell_a z^a\overline{w}^a.\]
 \begin{lemma}\label{l:Lnkpos}
    If $(k,\ell)$ is a CP pair of diagonal holomorphic kernels, then \[\frac{L^n}{k}\succeq 0.\]
 \end{lemma}
 
 \begin{proof} 
   Fix a set $\Lambda =\{\lambda_j:0\le j\le n\}$ of distinct points near $\mathbf{0}$ with $\lambda_0=\mathbf{0}.$  Abusing notation slightly, we will write $P_{m, \Lambda}\ell$ to denote the reproducing kernel associated with the subspace $M_{m, \Lambda},$ which is obtained by applying the projection $P_{m, \Lambda}$ to the kernel functions $\ell_w.$ By Theorem \ref{generalnec}, \[\frac{(I-P_{0, \{\lambda_1, \dots, \lambda_n\}})\ell}{k}=\frac{\ell^{\Lambda}}{k}\succeq 0.\]  
We will now show, by induction on $m$, that 
\begin{equation} \label{indpos}
 \frac{(I-P_{m, \{\lambda_{m+1}, \dots, \lambda_n\}})\ell}{k} \succeq 0,  
\end{equation}
for all $0\le m\le n$ and any choice of distinct nonzero points $\lambda_{m+1}, \dots, \lambda_n$ near $\mathbf{0}.$ Assume that it holds for some fixed $0\le m\le n-1$ and let $\lambda_{m+2}, \dots, \lambda_n\in\Omega$ be distinct, non-zero. By Lemma \ref{l:Pconverge:2}, there exists $v: (0, \delta]\to\Omega$ continuous such that
\[P_{m,\{v(t), \lambda_{m+2}, \dots, \lambda_n\}}\to P_{m+1, \{\lambda_{m+2}, \dots, \lambda_n\}} \]
   in operator norm as $t\to 0.$ By our inductive hypothesis, 
   \[ \frac{(I-P_{m, \{v(t), \lambda_{m+2}, \dots, \lambda_n\}})\ell}{k}\succeq 0, 
   \]
 for all $t\in (0,\delta].$
   Since pointwise limits of PsD kernels are PSD kernels, we may let $t$ tend to $0$  to obtain
   \[ \frac{(I-P_{m+1, \{\lambda_{m+2}, \dots, \lambda_n\}})\ell}{k}\succeq 0.\]
   Thus, \eqref{indpos} holds for every $m\le n.$ Setting $m=n$ then yields
   \[\frac{L^n}{k}=\frac{(I-P_{n, \emptyset})\ell}{k}\succeq 0,\]
   as desired, where $P_{n, \emptyset}$ denotes the projection onto the span of $\{z^a: a\in F_n\}.$
 \end{proof}
 Next, we show that polynomials automatically yield bounded multipliers $\mathcal{H}_k\to\mathcal{H}_{\ell}$ whenever $(k, \ell)$ is a CP pair of diagonal holomorphic kernels. 
 
 \begin{proposition} \label{autopolyn}
    If $(k,\ell)$ is a CP pair of diagonal holomorphic kernels, then all $\mathbb{M}_N$-valued polynomials belong to $\text{Mult}(\mathcal{H}_k\otimes \mathbb{C}^N, \mathcal{H}_{\ell}\otimes \mathbb{C}^N)$, for all $N\ge 1.$ In particular, if $k$ is normalized, we have 
    \[||z^b||_{\text{Mult}(\mathcal{H}_{k}, \mathcal{H}_{\ell})}=||z^b||_{\mathcal{H}_{\ell}},\]
    for all $b\in\mathbb{N}^{\vg}.$
 \end{proposition}
 
 \begin{proof} Fix $n$ and set $b=a^n.$ 
   From Lemma~\ref{l:Lnkpos},  $\frac{L^n}{k}=g$ for some positive
    kernel $g.$ Since $L^n$ and $k$ are holomorphic diagonal kernels, so is $g.$ Writing 
   $L^n= g\, k,$ we find  $g_{a^j}=0$ for $j<n$ (that is $g_a=0$ whenever $a\prec b$ with respect to the lexicographic ordering) and 
\[
 \ell_b = \sum_{0\le u\le b} g_u k_{b-u}.
\]
Now, if $u\le b$ then $|u|\le |b|.$ If the inequality is strict, then $g_u=0.$ If $|u|=|b|$, then 
we must have $u=b.$  Hence, $\ell_b= g_b k_0$.  In particular, $g_b>0.$  Now, for $a\succeq  b,$ 
 \[
    \ell_a = \sum_{u\le a} g_u k_{a-u} =  g_b \, k_{a-b}   + \sum_{b\ne u \le a} g_u k_{a-u}
     \ge g_b \, k_{a-b}.
 \] 
 We conclude
 \[\frac{1}{g_b}\ell(z, w)-z^b\overline{w}^bk(z, w)\succeq 0.\]
 Hence, $||z^b||_{\text{Mult}(\mathcal{H}_{k}, \mathcal{H}_{\ell})}\le 1/\sqrt{g_b}$, and, consequently, multiplication by all matrix-valued polynomials is bounded. Finally, if $k$ is normalized, then $g_b=\ell_b$ and so we may write 
 \[||z^b||_{\mathcal{H}_{\ell}}=||z^b\cdot 1||_{\mathcal{H}_{\ell}}\le ||z^b||_{\text{Mult}(\mathcal{H}_{k}, \mathcal{H}_{\ell})}\le 1/\sqrt{\ell_b}=||z^b||_{\mathcal{H}_{\ell}},\]
which proves the second assertion.
 \end{proof}

The following lemma tells us that $(k, \ell)$ is CC if and only if it has the one-point extension property with respect to the lexicographic order. 
\begin{lemma}\label{lexiconvenient}
A pair   $(k, \ell)$  of diagonal holomorphic kernels
 has the CC property  if and only if for all positive integers $J$ and $n$  and collections $\{c_a  : a\in \{a^0, a^1, \dots, a^n\}\}\subseteq \mathbb{M}_J$ such that the block (upper-triangular) matrix $C$ indexed by $\{a^0, a^1, \dots, a^n\}\times \{a^0, a^1, \dots, a^n\}$ with block $J\times J$ entries,
\begin{equation}
  \label{e:Cablex}
   C_{a, b}=\begin{cases}
     c_{b-a}\sqrt{\cfrac{k_a}{\ell_b}}, \hspace{0.5 cm} b\ge a, \\
      \mathbf{0},  \hspace{1.05 cm} \text{otherwise,}
 \end{cases}
\end{equation}
is a contraction,  
the matrix $C^+$ indexed by $\{a^1, \dots, a^{n+1}\}\times \{a^1, \dots, a^{n+1}\}$, and given by 
\begin{equation}\label{onepointlex}
    C^+_{a, b}=\begin{cases}
     c_{b-a}\sqrt{\cfrac{k_a}{\ell_b}}, \hspace{0.5 cm} b\ge a, \\
      \mathbf{0},  \hspace{1.05 cm} \text{otherwise,}
 \end{cases} \end{equation}
 is also a contraction.
\end{lemma}
\begin{proof}
We need two observations: $\{a^0, a^1, \dots, a^n\}\subseteq\mathbb{N}^{\vg}$ is co-invariant for every $n$, and also $b-a\in \{a^0, \dots ,a^n\}$ whenever $a\le b$ with $a, b\in\{a^1, \dots, a^{n+1}\}$. We omit the rest of the proof, since it is essentially identical to that of Proposition~\ref{convenient}.
\end{proof}
We are now ready to give a direct proof that the CP property implies the CC property.
\begin{theorem} \label{directpassage}
    If the normalized diagonal holomorphic pair $(k, \ell)$ is CP, then it is also CC. 
\end{theorem}
\begin{proof}
First, we establish some notation. Given $m\ge 0 $ and $\Lambda\subseteq \Omega\setminus \{\mathbf{0}\}$, as before let  $P_{\Lambda}$ and $P_{m,\Lambda}$ denote the projections onto the spans of $\{\ell_{\lambda} : \lambda\in\Lambda\}$ and $\{z^{a^j}:0\le j\le m\}\cup \{\ell_{\lambda} : \lambda\in\Lambda\}$, respectively. Define $P^0_{\Lambda}$ and $P^0_{m,\Lambda}$ likewise, with $\ell$ replaced by $\ell-1.$ Finally, define $Q_{\Lambda}, Q_{m,\Lambda}, Q^0_{\Lambda}$ and $Q^0_{m,\Lambda}$ in an analogous manner, with $\ell$ replaced by $k$ throughout.

  Now, fix $n\ge 1$, choose $J\ge 1$ and  $\{c_a : a\in\{a^0, \dots, a^n\}\}\subseteq \mathbb{M}_J$. Set $F=\{a^0, \dots, a^n\}$ and $p(z)=\sum_{a\in F}c^*_a z^a$ and assume that the block matrix $C=\big(C_{a, b} \big)_{a, b\in F}$ as in \eqref{e:Cablex} is a contraction. Further, set $F^+=\{a^1, \dots, a^{n+1}\}$ and define the $F^+\times F^+$ block matrix $C^+$ as in \eqref{onepointlex}.  Our goal is to show that $C^+$  is also a contraction. By Lemma \ref{autopolyn}, $p$ yields a bounded operator $M_p: \mathcal{H}_k\otimes\mathbb{C}^J\to \mathcal{H}_{\ell}\otimes\mathbb{C}^J$. Setting $\mathcal{H}^0_k=\mathcal{H}_{k-1}$ and $\mathcal{H}^0_{\ell}=\mathcal{H}_{\ell-1}$, let $T_p: \mathcal{H}^0_k\otimes \mathbb{C}^J \to \mathcal{H}^0_{\ell}\otimes \mathbb{C}^J$ denote the restricted operator $M_p|_{\mathcal{H}^0_k\otimes \mathbb{C}^J}$.   Since $(k,\ell)$ has the CP property, Lemma~\ref{boringonepoint} implies that
  the inequality of equation~\eqref{initial} implies the inequality of equation~\eqref{e:boringonepoint}.
  Here, we replace $\{z_1, \dots, z_n\}$ by $\Lambda$ and $z_{n+1}$ by $0,$ so that $\ell^{z_{n+1}}$ in equation~\eqref{e:boringonepoint} becomes $\ell-1.$ Hence
  \begin{equation} \label{primalCP}
    ||(Q_{\Lambda}\otimes I_{J\times J}) M^*_p (P_{\Lambda}\otimes I_{J\times J})|| \ge ||(Q^0_{\Lambda}\otimes I_{J\times J})  T^*_p (P^0_{\Lambda}\otimes I_{J\times J})||,
  \end{equation}
for any $(n+1)$-point subset $\Lambda=\{\lambda_0, \lambda_1, \dots, \lambda_n\}\subseteq \Omega\setminus \{\mathbf{0}\}.$ Set $\Lambda^m=\Lambda\setminus \{\lambda_0, \dots, \lambda_m\},$ for any $0\le m\le n.$ By Lemma \ref{l:Pconverge:2}, there exists a continuous map $v:(0, \delta]\to\Omega$ such that, with $\Lambda^0_t=\Lambda^0\cup\{v(t)\}$, we have $P_{\Lambda^0_t}\to P_{0, \Lambda^0}, \,  P^0_{\Lambda^0_t}\to P^0_{0, \Lambda^0}, \, Q_{\Lambda^0_t}\to Q_{0, \Lambda^0} $ and $Q^0_{\Lambda^0_t}\to Q^0_{0, \Lambda^0}$ in operator norm as $t\to 0$. Thus, we may replace $\Lambda$ by $\Lambda^0_t$ in \eqref{primalCP} and let $t\to 0$ to obtain
\begin{equation*}  
   ||(Q_{0, \Lambda^0}\otimes I_{J\times J}) M^*_p (P_{0, \Lambda^0}\otimes I_{J\times J})|| \ge ||(Q^0_{0, \Lambda^0}\otimes I_{J\times J})  T^*_p (P^0_{0, \Lambda^0}\otimes I_{J\times J})||,  
\end{equation*}
for any  $\Lambda^0=\{\lambda_1, \dots, \lambda_n\}\subseteq \Omega\setminus\{\mathbf{0}\}.$ One may now proceed by induction; successive uses of Lemma \ref{l:Pconverge:2} yield
\begin{equation*}  
   ||(Q_{m, \Lambda^m}\otimes I_{J\times J}) M^*_p (P_{m, \Lambda^m}\otimes I_{J\times J})|| \ge ||(Q^0_{m, \Lambda^m}\otimes I_{J\times J})  T^*_p (P^0_{m, \Lambda^m}\otimes I_{J\times J})||,  
\end{equation*}
for every $0\le m \le n$ and any $\Lambda^m=\{\lambda_m, \dots, \lambda_n\}\subseteq \Omega\setminus\{\mathbf{0}\}.$ Setting $m=n$ then yields 
\begin{equation} \label{thatsit}
||(Q_{n}\otimes I_{J\times J}) M^*_p (P_{n}\otimes I_{J\times J})|| \ge ||(Q^0_{n}\otimes I_{J\times J})  T^*_p (P^0_{n}\otimes I_{J\times J})||,  
\end{equation}
where $P_n$ and $P^0_n$ are the projections onto $\textup{span}\{z^{a^j}: 0\le j\le n\}$ and $\textup{span}\{z^{a^j}: 1\le j\le n+1\},$ respectively, with $Q_n$ and $Q^0_n$ defined analogously. But $C$ being a contraction is equivalent to $||(Q_{n}\otimes I_{J\times J}) M^*_p (P_{n}\otimes I_{J\times J})|| \le 1$, which, in view of \eqref{thatsit}, yields $||(Q^0_{n}\otimes I_{J\times J})  T^*_p (P^0_{n}\otimes I_{J\times J})||\le 1$. Thus, $C^+$ is a contraction and our proof is complete. 
\end{proof}

\begin{remark}
We point out that our approach here is different from that of \cite{McCulloughcarath} and \cite{Mikethesis}; in those papers, the equivalence of the CC and CP properties was established by first showing that both conditions are characterized by the same positivity condition on the kernel $k$ (i.e. \eqref{singleCPnec}), while our proof is based on a direct passage from the complete Pick to the complete Carath\'eodory problem.
\end{remark}

\normalsize
\section{A Complete Characterization} \label{mainmain}
\normalsize
In this section, Theorems~\ref{introextmain} and \ref{CPisCC} are established as a consequence of Theorem~\ref{mainextmain}. 
 The main issue involves domains; that is,  passing from a master (formal) Shimorin certificate for a CC pair in Theorem~\ref{main:formal:CCP} to
 a strong Shimorin certificate.   The results are then illustrated with examples of Bergman-like kernels.

\subsection{The Shimorin certificate characterization of  diagonal holomorphic CP pairs}
Let $k, \ell$ be normalized diagonal holomorphic kernels in $\vg$ variables and recall the definitions of the domains of convergence
 $\Omega_\ell,\Omega_k,$ which, by assumption, are non-empty.
Suppose  $t(x)=\sum_{|a|>0}t_a x^a$ is a formal Shimorin certificate for $(k, \ell)$ and set 
\begin{align*}
\Omega_t^1=\{x\in\Omega_t: \sum_{|a|>0}t_a |x|^{2a}<1\}.
\end{align*}
In Section~\ref{theCCsection} we investigated how the existence of $t$ is related with the CC property for $(k, \ell).$  However, our calculations revolved exclusively around the matrix-completion version of the Carath\'eodory problem and no attention was payed to the domains where $k, \ell, t$ are actually defined. We shall settle this issue with Proposition~\ref{l:tlg} below.  We will write $(k, \ell, X)$ in place of $(k, \ell)$ to signify that the common domain of the kernels $k, \ell$ is $X$. We also update Definition~\ref{strcertdef} as follows.

\begin{definition}\label{updatestrcertdef}
 Assume $k, \ell$ are kernels on $X.$ A kernel $s$ on $X$ is a \df{strong Shimorin certificate for} 
 $(k, \ell, X)$ if $s$ is a complete Pick kernel on $X$ such that \eqref{intshimfact} holds with $\bh$ and $g$ defined on $X.$   
\end{definition}

\begin{proposition}
\label{l:tlg}
Suppose $t$ is a formal Shimorin certificate for $(k,\ell),$ a pair of normalized diagonal holomorphic kernels. 
Thus, by assumption,  $g=(1-t)\ell$ and $\hb=1-(1-t)k$ 
are formal (diagonal) power series with non-negative coefficients.   
 With notations as above,   
 \[
   \Omega_\ell \subseteq \Omega_t^1\subseteq \Omega_k,
     \hspace*{0.7 cm} \Omega_\ell \subseteq \Omega_g, \, \Omega_h.
\]
 Further, define $s$ as the formal power series reciprocal of $1-t.$
 Then, $\Omega_t^1=\Omega_s$ and the function $s:\Omega_t^1 \times\Omega_t^1 \to \CC$ given by
\[
   s(z, w)=s(z\overline{w})=\frac{1}{1-t(z\overline{w})}, 
\]
 is a strong Shimorin certificate for $(k, \ell, \Omega_{\ell})$  with 
 \[
   \ell(z,w)=g(z,w)\, s(z,w)  \hspace*{0.7 cm} k(z,w)=(1-\hb(z,w))\, s(z,w).
 \]
 as functions on $\Omega_\ell\times\Omega_\ell.$
 %
%
\end{proposition}

\begin{proof} 
 By assumption, $\Omega_k$  and $\Omega_\ell$ contain
 some open neighborhood of the origin. Moreover,  $t_a, k_a, \ell_a\ge 0$ for all $|a|\ge 0.$ 
 By definition of a formal Shimorin certificate, there exist formal power series $g(x)=\sum_a g_a x^a$ and  
 $\hb(x)=\sum_a \hb_a x^a$ with non-negative coefficients (and also $g_{\mathbf{0}}=1$ and $\hb_{\mathbf{0}}=0$) 
 such that $\ell(1-t)=g$ and $1-k(1-t)= \hb$ as formal power series. In particular,
\begin{equation}
\label{e:ltg:1}
 g_a +t_a = \ell_a-\sum_{0<u<a} t_u \ell_{a-u}  \le \ell_a.
\end{equation}
 It follows that
\[
 \{x\in\CC^\vg: \sum \ell_a |x|^{2a} <\infty\}\subseteq \{x\in\CC^\vg: \sum g_a |x|^{2a} <\infty\}
\]
 and similarly with $t$ in place of $g.$ Hence 
 $\Omega_\ell\subseteq \Omega_g,\Omega_t.$  
   \par
  According to Remark~\ref{r:ka<=la}, 
 $k_a\le \ell_a$ and hence, just as above,
 $\Omega_\ell\subseteq \Omega_k.$ Further, from equation~\eqref{e:ltg:1},
 \[
  g_a+ \sum_{0<u\le a} t_uk_{a-u} \le g_a+ \sum_{0<u\le a} t_u \ell_{a-u} =\ell_a.
 \]
 Hence, $\sum_{0<u\le a} t_u k_{a-u} \le \ell_a$ and therefore,
 \[
  \hb_a  \le (kt)_a =\sum_{0<u\le a} t_u k_{a-u} \le \ell_a.
 \]
 Thus, $\Omega_\ell\subseteq \Omega_h.$
 Consequently, all of  $k,\ell,t,g,h$ determine holomorphic kernel functions on at least $\Omega_\ell$ 
 and, as functions 
 on $\Omega_\ell\times\Omega_\ell,$ they satisfy $\ell(z,w)(1-t(z,w))=g(z,w)$ and $1-k(zw)(1-t(z,w))=\hb(z,w).$
\par 

  Let $z\in \Omega_\ell$ be given.
  Since $g(z,z)=\ell(z,z)(1-t(z,z))$  for $z\in\Omega_\ell\subset\Omega_t,$  and since
 both $g(z,z)$ and $\ell(z,z)$ are positive, it follows that $0\le t(z,z)<1.$ Thus, $\Omega_\ell\subseteq\Omega_t^1$ as
 claimed. Also, $t$ being PsD implies $|t(z, w)|<1$ for all $z, w\in\Omega^1_t,$ thus, by construction,  $s(z,w)=\frac{1}{1-t(z,w)}$ is a CP kernel on
 $\Omega_t^1$. We conclude that $s$ is a strong Shimorin certificate for $(k,\ell, \Omega_{\ell}).$
   \par

 By Lemma~\ref{Omegafsup1}, $\Omega_t^1$ is a domain containing $\mathbf{0}.$
 Since  $s$ is a diagonal holomorphic kernel on 
 $\Omega_t^1,$ Proposition~\ref{prop:Pringsheim} gives
 $\Omega_t^1\subseteq \Omega_s.$ 
 On the other hand, $t$ is a formal
 Shimorin certificate for the pair $(k,s)$ (with the same $h$ but $g=1$).
 Hence, $\Omega_s\subseteq \Omega_t^1,$ and so $\Omega_s=\Omega_t^1.$
   \par
 We will now show that $\Omega^1_t=\Omega_s \subseteq \Omega_k.$ Since  $\hb$ is positive on $\Omega_{\ell}\subset \Omega_s \cap \Omega_k\cap \Omega_h,$ there exists a Hilbert space $E$ and a holomorphic $\Phi: \Omega_{\ell}\to\mathcal{B}(E, \mathbb{C})$ such that $\hb=\Phi\Phi^*$ on $\Omega_{\ell}\times\Omega_{\ell}$. Since 
\[
 \frac{1-\Phi\Phi^*}{1-t}=(1-\hb)s=k\succeq 0 \hspace{0.1 cm}\text{ on } \hspace{0.1 cm} \Omega_{\ell}\times \Omega_{\ell},
 \]
$\Phi$ is a contractive multiplier $\mathcal{H}_s(\Omega_{\ell})\otimes E\to \mathcal{H}_s(\Omega_{\ell})$.  Combining the Identity Principle with results from \cite[Section 5.4]{PaulsenRagh} (the CP property of $s$ is not needed here), we obtain the existence of a (unique) holomorphic $\widetilde{\Phi}: \Omega^1_t\to\mathcal{B}(E, \mathbb{C})$ such that $\widetilde{\Phi}\equiv \Phi$ on $\Omega_{\ell}$ and $\widetilde{\Phi}$ is a contractive multiplier $\mathcal{H}_s\otimes E\to\mathcal{H}_s$.
Thus, $(1-\widetilde{\Phi}\widetilde{\Phi}^*)s$ is a holomorphic kernel on $\Omega_s$ that agrees with $k$ on $\Omega_{\ell}$. Applying Proposition~\ref{prop:Pringsheim} with $f=k$ and  $\tilde{f}=(1-\widetilde{\Phi}\widetilde{\Phi}^*)s$ on the domain  $\Omega=\Omega_s,$ we conclude $\Omega_s\subseteq \Omega_k,$ as desired.
\end{proof}

Specializing Proposition~\ref{l:tlg} to the case $t$ is the master certificate $\mt$ for $k$
gives the following corollary.

\begin{corollary}
\label{c:tlg}
  Suppose $k$ is a normalized diagonal holomorphic kernel and let $\mt$ denote
  its associated master certificate. Letting $s:\Omega_s\times\Omega_s$ denote the  
   diagonal holomorphic kernel $s=\frac{1}{1-\mt},$ we have
  \[
     \Omega_s=\Omega_\mt^1 \subseteq \Omega_k \subseteq \Omega_\mt
  \]
 and the function $\hb:\Omega_s\times\Omega_s\to \CC$ defined by $\hb(z,w)=1-k(z,w)(1-\mt(z,w))$
 is a diagonal holomorphic kernel.   
 
  Further, if $(k,\ell,\Omega)$ is a complete Carath\'eodory pair, then $\Omega\subseteq \Omega_\ell\subseteq \Omega_s$
  and $s$ is a strong Shimorin certificate for $(k,\ell,\Omega_\ell).$ 
\end{corollary}

\begin{proof}  
The inclusions $\Omega_k\subseteq \Omega_\mt$ and $\Omega\subseteq \Omega_{\ell}$ follow from Theorems~\ref{1-k(1-mt)} and ~\ref{prop:Pringsheim}, respectively.
 
  By Proposition~\ref{1-k(1-mt)}, the formal power series $\hb=1-k(1-\mt)$ has 
  non-negative coefficients. Evidently, $s(1-\mt)=g=1.$ 
   Hence $s$ is a formal Shimorin certificate for $(k,s)$ as witnessed
   by $g$ and $\hb.$  An application of  Proposition~\ref{l:tlg} with $\ell=s$ 
    gives  $\Omega_s =  \Omega_\mt^1\subseteq \Omega_k$ and
     says that $s,\hb,k$ all define kernels on $\Omega_s\times\Omega_s$
   and  $k(z,w)=(1-\hb(z,w))\, s(z,w)$  on $\Omega_s\times \Omega_s.$
     
     Finally, if $(k,\ell,\Omega)$ is a complete Carath\'eodory pair, then $\mt$ is a formal
    certificate for $(k,\ell)$ by Theorem~\ref{main:formal:CCP}. Hence, by Proposition~\ref{l:tlg},
     $\Omega_\ell\subseteq \Omega_\mt^1=\Omega_s$ 
      and  $s$ is a strong Shimorin certificate for $(k,\ell,\Omega_\ell).$
\end{proof}

All the results needed for the proof of Theorems~\ref{introextmain} and~\ref{CPisCC} are now in place.

\begin{theorem} \label{mainextmain}
 Let $(k, \ell)$ be a pair of normalized diagonal holomorphic kernels on a domain 
 $\mathbf{0}\in \Omega \subseteq \mathbb{C}^\vg.$ Let $\mt$ denote the master certificate associated 
 with $k,$ let $s$ denote the kernel
 \[
   s=\frac{1}{1-\mt}
 \]
 and set  $\Omega_{\mt}^1=\{z\in\CC^{\vg}: \mt(z\overline{z})<1 \}$. 
 If $(k,\ell)$ is a complete Pick pair, then 
\begin{equation}
\label{inclusions}
  \Omega \subseteq \Omega_\ell\subseteq
  \Omega_s = \Omega^1_\mt \subseteq  \Omega_k \subseteq \Omega_\mt.
\end{equation}
 Further,  the following assertions are equivalent, independent of the  domain\footnote{As always, it is assumed that $\mathbf{0}\in \Omega.$}  $\Omega\subseteq \Omega_\ell.$ 
 \begin{enumerate}[(i)]
     \item \label{i:mm:1} $(k, \ell, \Omega)$ is a complete Pick pair;
      \item \label{i:mm:3} $(k, \ell)$ is a complete Carath\'eodory pair;
    \item \label{i:mm:7} 
     $s$ is a strong Shimorin certificate for $(k, \ell, \Omega);$
     \item \label{i:mm:8} there exists a positive kernel $g$ on $\Omega$ such that 
      $\ell=g\, s$ on $\Omega;$
 \item \label{i:mm:6}  $(k, \ell, \Omega)$ has a diagonal holomorphic strong Shimorin certificate;
      \item \label{i:mm:5}  $(k, \ell, \Omega)$ has a strong Shimorin certificate;
     \item \label{i:mm:4}  $(k, \ell, \Omega)$ has a Shimorin certificate.
 \end{enumerate}
\end{theorem}
\begin{proof}   
Since, by assumption, $\ell$ is a kernel over the domain $\mathbf{0}\in \Omega,$ Proposition~\ref{prop:Pringsheim} implies $\Omega\subseteq \Omega_\ell.$  
Corollary~\ref{c:tlg} gives the remaining inclusions of equation~\eqref{inclusions}.
Assuming $(k,\ell,\Omega)$ is a complete Pick pair, Theorem~\ref{directpassage} implies that $(k,\ell)$ is a complete Carathéodory  pair. Here what is needed, beyond $(k,\ell,\Omega)$ being a complete Pick pair,  is only that $\Omega$ contains a neighborhood of $\mathbf{0}.$   Hence item~\eqref{i:mm:1} implies item~\eqref{i:mm:3}.
Corollary~\ref{c:tlg} gives the implication item~\eqref{i:mm:3} implies item~\eqref{i:mm:7}. 
By the definition of strong Shimorin certificate, item~\eqref{i:mm:7} implies
 item~\eqref{i:mm:8}.  The reverse implication is true too,  since by 
  Corollary~\ref{c:tlg}, $\hb(z,w)=1-k(z,w)(1-\mt(z,w))$ is a positive kernel on $\Omega_\ell$ and
 $k=(1-\hb)s$ in addition to $\ell=g\, s.$
 It is evident that item~\eqref{i:mm:7} implies item~\eqref{i:mm:6} implies item~\eqref{i:mm:5}.
The implication item~\eqref{i:mm:5} implies \eqref{i:mm:4} is a consequence of Proposition~\ref{newcerttt}. 
Finally, Theorem~\ref{generalsuff} says item~\eqref{i:mm:4} implies \eqref{i:mm:1}.
\end{proof}

Corollary~\ref{canoncorol} below, which is an immediate consequence of Theorem~\ref{mainextmain}, says $s(z, w)=\frac{1}{1-\mt(z\overline{w})}$ is the ``minimum" among all \df{diagonal holomorphic} strong Shimorin certificates for $(k, \ell),$ for any diagonal holomorphic $\ell$ such that $(k, \ell)$ is CP. In subsection~\ref{bergcerts}, we will see that this minimality property of $s$ does not survive if we also consider non-diagonal certificates.

\begin{corollary} \label{canoncorol}
Let $k$ be a normalized diagonal holomorphic kernel with master certificate $\mt.$ Assume $\Omega\subseteq \Omega_{\mt}^1$, set  
\[
   s(z, w)=\frac{1}{1-\mt(z\overline{w})}
\] 
and let $\ell$ be any normalized diagonal holomorphic kernel on $\Omega$ such that $(k, \ell, \Omega)$ is CP. 
If $\tilde{s}$ is a strong Shimorin certificate for $(k, \ell, \Omega)$, then there exists a kernel $g$ on $\Omega$ such that 
\[\tilde{s}=sg. \] 
\end{corollary}

\begin{proof}
Assume $(k, \ell, \Omega)$ and  $\tilde{s}$ are as above. 
Thus, by assumption, there exists a kernel $\widetilde{\bh}$ such that  $k=(1-\widetilde{\bh})\tilde{s}$ on $\Omega.$  
Since  $(k, \tilde{s}, \Omega)$ is a CP pair (having $\tilde{s}$ as a strong Shimorin certificate),
Theorem~\ref{mainextmain} item~\eqref{i:mm:5} implies item~\eqref{i:mm:8},  says there exists $g\succeq 0$ such that, on $\Omega\times\Omega,$ 
\[
  \tilde{s}(z, w)=\frac{g(z, w)}{1-\mt(z\overline{w})}=s(z, w)g(z, w). \qedhere
\]
\end{proof}

\subsection{Bergman-like kernel examples}
In this subsection, the master certificate for some Bergman-like kernels are computed, with
special attention focused on the domains of these certificates.

\begin{example}\label{finallyanex}
Fix $\vg \ge 1$ and suppose  $p_1, \dots, p_{\vg}$ are positive integers. Set \[\mathfrak{b}_{p}(z, w)=\prod_{i=1}^{\vg}\frac{1}{(1-z_i\overline{w}_i)^{p_i}}, \hspace{0.3 cm} z, w\in\mathbb{D}^{\vg}.\]    
Given a domain $\Omega\subseteq\mathbb{D}^{\vg}$ and a diagonal holomorphic kernel $\ell$ on $\Omega,$ when does $(\mathfrak{b}_{p}, \ell, \Omega)$ have the CP property? We will calculate the master certificate associated with $\mathfrak{b}_{p}$. The following lemma will be crucial.

\begin{lemma} \label{multcertlem}
$(\mathfrak{b}_p)_a\le p_j \, (\mathfrak{b}_p)_{a-e_j},$ for any $j\in\{1, \dots, \vg\}$ and $a\in\NN^{\vg}$ with $e_j\le a.$
\end{lemma}

\begin{proof}
 First, assume $\vg=1$ and set $p_1=p.$ The statement of the lemma is then equivalent to    
 \[(\mathfrak{b}_p)_{n+1}\le p (\mathfrak{b}_p)_{n}, \hspace{0.2 cm} n\ge 0.\]
This inequality clearly holds when $n=0$ (for any $p\ge 1),$ so we may assume $n\ge 1.$ We then have
\[
(\mathfrak{b}_p)_{n+1}-p (\mathfrak{b}_p)_{n}=
\begin{pmatrix}
n+p \\
p-1
\end{pmatrix}
-
p\begin{pmatrix}
n+p-1 \\
p-1
\end{pmatrix}  
=-n(p-1)\frac{(n+p-1)!}{(p-1)!(n+1)!}\le 0,
\]
as desired. Now, fix $m\ge 1$ and suppose that we have proved Lemma~\ref{multcertlem} for any $\vg\le m.$ We will show that it also holds for $\vg=m+1.$ Indeed, let $p_1, \dots, p_{m+1}$ be positive integers and choose $a\in\mathbb{N}^{\vg}$ and $j\in\{1, \dots, m+1\}$ with $e_j\le a$. It is not hard to see that $(\mathfrak{b}_p)_{e_j}=p_j=p_j(\mathfrak{b}_p)_{\mathbf{0}},$ for all $j,$ so we may assume $|a|\ge 2.$ Without loss of generality,  we assume  $j\le m.$ Let $\widetilde{p}=(p_1, \dots, p_m)$ and write $a=(a_1, \dots, a_m, 0)+ke_{m+1}=\widetilde{a}+ke_{m+1}$, so we have $e_j\le \widetilde{a}.$ Since 
\[\mathfrak{b}_p(z, w)=\frac{\mathfrak{b}_{\widetilde{p}}(z, w)}{(1-z_{m+1}\overline{w}_{m+1})^{p_{m+1}}},\]
an application of  our inductive hypothesis to $\mathfrak{b}_{\widetilde{p}}$ yields
\[(\mathfrak{b}_p)_a=
(\mathfrak{b}_{\widetilde{p}})_{\widetilde{a}}\begin{bmatrix}
   k+p_{m+1} -1  \\
   p_{m+1} -1
\end{bmatrix} \le p_j(\mathfrak{b}_{\widetilde{p}})_{\widetilde{a}-e_j} \begin{bmatrix}
   k+p_{m+1} -1  \\
   p_{m+1} -1
\end{bmatrix} =p_j (\mathfrak{b}_{p})_{a-e_j},\]
as desired.
\end{proof}
We now compute $\mt$. If $1\le j\le \vg,$ then
\[\mt_{e_j}=(\mathfrak{b}_{p})_{e_j}=p_j. \]
We will show that $\mt_{a}=0$ whenever $|a|\ge 2.$ Indeed, let $a\in \NN^{\vg}$ with $|a|\ge 2$ and choose $e_j$ with $e_j\le a.$ Applying Lemma~\ref{multcertlem}, we have
\begin{equation}
 \label{mt-mfb}
 (\mathfrak{b}_p)_a-\sum_{\substack{0<u<a}}\mt_{u}(\mathfrak{b}_p)_{a-u}\le   (\mathfrak{b}_p)_a-\mt_{e_j}(\mathfrak{b}_p)_{a-e_j}=(\mathfrak{b}_p)_a-p_j(\mathfrak{b}_p)_{a-e_j}\le 0.
\end{equation}
Since (see equation~\eqref{def:mtb}) $\theta_a$ is the larger of $0$ and the expression on the left hand side
 of equation~\eqref{mt-mfb}, it follows that $\theta_a=0$ for $|a|\ge 2.$ 
Thus, we obtain
\[\mt(x)=\sum_{i=1}^{\vg}p_ix_i.\]
Hence, $\Omega_{\mt}^1=\{z\in\mathbb{C}^{\vg} : \sum_ip_i|z_i|^2< 1\}$, and Theorem~\ref{mainextmain} tells us that $(\mathfrak{b}_{p}, \ell, \Omega)$ is CP if and only if it is CC if and only if $\Omega\subseteq \{z\in\mathbb{C}^{\vg} : \sum_ip_i|z_i|^2< 1\}$ and 
\[\big(1-\sum_ip_iz_i\overline{w}_i\big)\ell(z, w)\]
is a positive kernel on $\Omega.$ This last positivity condition is equivalent to saying that 
\[\Phi(z):=\begin{bmatrix}
    \sqrt{p_1}z_1 &
    \cdots &
     \sqrt{p_{\vg}}z_{\vg}
\end{bmatrix} \]
is a contractive multiplier from $\mathcal{H}_{\ell}\otimes\mathbb{C}^{\vg}$ to $\mathcal{H}_{\ell}$.
\end{example}
\begin{example} \label{finallyanex2}
Fix $\vg \ge 1$, set $\mathbb{B}_{\vg}=\{z\in\CC^{\vg}: \sum_i|z_i|^2<1\}$ and let $\alpha$ be a positive integer. Define 
\[\mathfrak{b}_{\alpha}(z, w)=\frac{1}{(1-\langle z, w\rangle)^{\alpha}}, \hspace{0.2 cm} z, w\in \mathbb{B}_{\vg}.\]
From Example~\ref{finallyanex} with $\vg=1$ and setting $x=z\overline{w},$ 
 there exists for positive integers $i,$ non-negative $b_i$ such that,  setting  $B_{\alpha}(x)=\sum_i b_ix^i$  for $x\in \frac{1}{{\alpha}}\DD$, 
\[\frac{1}{(1-x)^{\alpha}}=\frac{1-B_{\alpha}(x)}{1-\alpha x}, \hspace{0.2 cm} x \in\frac{1}{{\alpha}}\DD.\]
Setting $\widetilde{B}_{\alpha}(z, w)=\sum_i b_i\langle z, w\rangle^i$, we obtain that $\widetilde{B}_{\alpha}$ is a kernel such that 
\begin{equation}\label{Ba}
\frac{1}{(1-\langle z, w\rangle)^{\alpha}}=\frac{1-\widetilde{B}_{\alpha}(z, w)}{1-\alpha\langle z, w\rangle}, \hspace{0.2 cm} z, w\in \frac{1}{\sqrt{\alpha}}\mathbb{B}_{\vg}.
\end{equation}
The identity of equation~\eqref{Ba} is equivalent to 
\[1-\mathfrak{b}_{\alpha}(z, w)(1-\alpha\langle z, w\rangle) = \widetilde{B}_\alpha(z,w) \succeq 0,\]
which yields
\begin{equation}\label{Ba:2}
(\mathfrak{b}_{\alpha})_{u}\le \alpha \sum_{\substack{1\le i\le \vg, \\ e_i\le u}}(\mathfrak{b}_{\alpha})_{u-e_i},
\end{equation}
for every $u\in\NN^{\vg}$ with $|u|\ge 2.$ We can now compute the master certificate $\mt$ associated with $\mathfrak{b}_{\alpha}.$ If $1\le j\le \vg,$ then 
\[\mt_{e_j}=(\mathfrak{b}_{\alpha})_{e_j}=\alpha.\]
Next, assume $|u|=2.$ We have 
\[(\mathfrak{b}_{\alpha})_u-\sum_{\substack{1\le i\le \vg, \\ e_i\le u}}(\mathfrak{b}_{\alpha})_{u-e_i}\mt_{e_i}=(\mathfrak{b}_{\alpha})_u-\alpha\sum_{\substack{1\le i\le \vg, \\ e_i\le u}}(\mathfrak{b}_{\alpha})_{u-e_i}\le 0,\]
which yields $\mt_u=0.$ Proceeding by induction, we  obtain  $\mt_u=0$ for all $|u|\ge 2.$ 
Thus, 
\[\mt(x)=\alpha\sum_{i=1}^{\vg}x_i,\]
and so $\Omega^1_{\mt}=\frac{1}{\sqrt{\alpha}}\mathbb{B}_{\vg}$. Theorem~\ref{mainextmain} now tells us that $(\mathfrak{b}_{\alpha}, \ell, \Omega)$ is CP if and only if   $\Omega\subseteq \frac{1}{\sqrt{\alpha}}\mathbb{B}_{\vg}$ and 
\[(1-\alpha\langle z, w\rangle)\ell(z, w)\]
is a positive kernel on $\Omega,$ which is the same as saying that 
\[\Psi(z):=\sqrt{\alpha}\begin{bmatrix}
z_1 & \cdots & z_{\vg}
\end{bmatrix}\]
is a contractive multiplier from $\mathcal{H}_{\ell}\otimes\CC^{\vg}$ to $\mathcal{H}_{\ell}$. 
\end{example}

\normalsize

\section{General Necessary Conditions} 
  \label{generalnecsecti}
\normalsize
Theorem~\ref{Schurpreserves} and its corollary Theorem~\ref{generalnec} from the introduction along with several other necessary conditions for a pair to be a complete Pick pair are established in this section. We assume once more that all pairs $(k, \ell) $ consist of kernels that are non-vanishing along the diagonal.

\subsection{Zero-based restrictions} 
In this subsection, it is seen that analogs of the zero-based restriction results, Lemma~\ref{l:zeroslemma} and Proposition~\ref{Scertforhol}, hold for CP pairs $(k,\ell)$
without the assumption of a Shimorin certificate.
\begin{proposition} \label{CPsplit}
Assume $(k, \ell)$ is a CP pair on $X$ and $z,w,v\in X$ are distinct. If  $k(z, w)=0,$ then, $\ell(z, w)=0$ and at least one of the following assertions holds:
\begin{enumerate}[(i)]
    \item $\ell(z, v)=\ell(w, v)=0$;
    \item  either $k(z, v)=0$ or $k(w, v)=0$.
\end{enumerate}
\end{proposition}

\begin{proof}
Fix $z, w\in X$ and assume $k(z, w)=0$. Define $\psi: \{z\}\to \mathbb{C}$ by 
\[\psi(z)=\sqrt{\frac{\ell(z, z)}{k(z, z)}}.\]
Clearly, $[\ell(z,z)-|\psi(z)|^2k(z, z)]=[0]\succeq 0,$ thus, by the CP property, $\psi$ extends to a multiplier on $\{z, w\}$ satisfying 
\begin{equation*}
\begin{split}
 0 &\preceq \begin{bmatrix}
  \ell(z,z)-|\psi(z)|^2k(z, z)  & \ell(z,w)-\psi(z)\overline{\psi(w)}k(z, w) \\
  \ell(w,z)-\psi(w)\overline{\psi(z)}k(w, z)  & \ell(w,w)-|\psi(w)|^2k(w, w)
\end{bmatrix} \\ 
&=\begin{bmatrix}
  0 & \ell(z,w)\\
  \ell(w,z) & \ell(w,w)-|\psi(w)|^2k(w, w)
\end{bmatrix}.   
\end{split}
\end{equation*}
This gives us $\ell(z,w)=0$, as desired.  \par 
Now, let $v\in X\setminus\{z, w\}$ and assume $k(z, v)\ne 0 \ne k(w, v).$ We will show that $\ell(z, v)=\ell(w, v)=0.$ Define $\phi: \{z, w\}\to\mathbb{C}$ by 
$\phi(x)=e^{i\mt_x}\sqrt{\frac{\ell(x, x)}{k(x, x)}},$ 
where each $\mt_x\in \mathbb{R}$ is chosen arbitrarily with the dependency of $\phi$ on $\mt_z,\mt_w$  suppressed. 
Clearly, $\ell(x,x)-|\phi(x)|^2k(x, x)=0$ for $x\in \{z, w\}.$ Since we also have $\ell(z, w)=k(z, w)=0,$ the Pick matrix for $\phi$ over $\{z, w\}$ is identically zero. By the CP property, $\phi$ can then be extended to a multiplier, still denoted $\phi$  on $\{z, w, v\}\to \CC$ satisfying 
\[
0 \preceq \begin{bmatrix}
  0  & 0 & \ell(z,v)-\phi(z)\overline{\phi(v)}k(z, v)
  \\
  0  & 0 & \ell(w,v)-\phi(w)\overline{\phi(v)}k(w, v) \\ 
  * & * &  \ell(v,v)-\psi(v)\overline{\psi(v)}k(v, v) 
\end{bmatrix}.
\]
Thus, we obtain $\ell(z,v)-\psi(z)\overline{\psi(v)}k(z, v)=\ell(w,v)-\psi(w)\overline{\psi(v)}k(w, v)=0.$ Since $\ell(z, v)\neq 0,$ we also have $k(z, v)\neq 0$. Solving for $\overline{\psi(v)}$ twice, 
\begin{equation*}
\begin{split}
e^{-i\mt_z}\frac{\ell(z, v)}{k(z, v)}\frac{\sqrt{k(z, z)}}{\sqrt{\ell(z, z)}}=\overline{\psi(v)}=e^{-i\mt_w}\frac{\ell(w, v)}{k(w, v)}\frac{\sqrt{k(w, w)}}{\sqrt{\ell(w, w)}}.
\end{split}
\end{equation*}
Thus,
\[\ell(w, v)e^{-i(\mt_w-\mt_z)}=k(w,v)\frac{\ell(z, v)}{k(z, v)}\frac{\sqrt{k(z, z)}}{\sqrt{\ell(z, z)}}\frac{\sqrt{\ell(w, w)}}{\sqrt{k(w, w)}}.\]
Since $\mt_z, \mt_w\in\mathbb{R}$ are arbitrarily,  $\ell(z, v)=\ell(w, v)=0,$ as desired.
\end{proof}

In Proposition~\ref{Scertforhol}, we saw that, for a holomorphic pair $(k, \ell)$ over a connected domain to possess a Shimorin certificate, it is necessary that $k$ be non-vanishing.  This condition continues to be necessary even if we only require $(k, \ell)$ to be a CP pair. 
\begin{proposition}
If $(k, \ell)$ is a CP pair of holomorphic kernels on the connected domain $\Omega\subseteq\mathbb{C}^{\vg}$ that are not identically zero, then $k$ is non-vanishing.
\end{proposition}

\begin{proof}
The proof is completed by  arguing  as in the first part of the proof of Proposition~\ref{Scertforhol}, the only difference being the substitution of Proposition~\ref{CPsplit} for  Lemma~\ref{l:zeroslemma}. We omit the details.
\end{proof}

\subsection{Positivity conditions}
We will now show that taking Schur complements with respect to $\ell$ (i.e. replacing $\ell$ by $\ell^w$) preserves the CP property for $(k, \ell).$ 
Recall, for $Y\subseteq X$ is a finite set $\ell^Y$ denotes the kernel for the subpace of $\mathcal{H}_\ell$ consisting of those functions that
 vanish on $Y.$ 

\begin{theorem}\label{Schurpreserves}
 Let $(k, \ell)$ be a CP pair of kernels on a set $X.$ Given any finite $Y\subseteq X$, the pair $(k, \ell^Y)$ is a CP pair.    
\end{theorem}

\begin{proof}
It suffices to prove the statement with $Y=\{w\}$ a singleton and we will show that $(k, \ell^w)$ is a CP pair by using the definition. Accordingly, suppose $Z=\{x_{1}, \dots, x_n\}\subset X$ is a finite set of points,  $W_{1}, \dots, W_n\in\mathbb{M}_J$ and 
\begin{equation}\label{inter}
    \big[\ell^w(x_i, x_j)-k(x_i, x_j)W_iW_j^*\big]_{1\le i, j\le n}\succeq 0.
\end{equation}
We are required to establish the existence of a contractive multiplier $\Phi\in\Mult(\mathcal{H}_k\otimes\mathbb{C}^J,\mathcal{H}_{\ell^w}\otimes\mathbb{C}^J)$ such that $\Phi(x_i)=W_i$ for all $1\le i\le n.$ \par 
Assume first that $w\notin Z.$ Set $w=x_0$ and 
consider the matrix 
\begin{equation} \label{pickcond}
    \big[\ell(x_i, x_j)-k(x_i, x_j)W_iW_j^*\big]_{0\le i, j\le n},
\end{equation}
where we have set $W_0=\mathbf{0}.$ Taking the Schur complement with respect to the $(0, 0)$ entry, we obtain that \eqref{pickcond} is positive if and only if 
\begin{equation*} 
   \big[\ell^{x_0}(x_i, x_j)-k(x_i, x_j)W_iW_j^*\big]_{1\le i, j\le n} \succeq 0,
\end{equation*}
which is simply \eqref{inter}. Since $(k, \ell)$ is a CP pair, we obtain the existence of a contractive multiplier $\Phi:\mathcal{H}_k\otimes\mathbb{C}^J\to \mathcal{H}_{\ell}\otimes\mathbb{C}^J$ such that $\Phi(x_i)=W_i$ for all $1\le i\le n$ and $\Phi(x_0)=\mathbf{0}.$ Thus, the range of $\Phi$ is actually contained in \[\{f\in \mathcal{H}_{\ell}\otimes\mathbb{C}^J: f(x_0)=\mathbf{0} \}=\mathcal{H}_{\ell^w}\otimes\mathbb{C}^J,\] as desired. \par 
Assume now that $w\in Z.$ Without loss of generality,  $w=x_1.$ In this case, the $(1, 1)$ entry of  \eqref{inter} is $-W_1W_1^*k(x_1, x_1),$ which has to be non-negative. Since  $k(x_1, x_1)>0,$ we obtain $W_1=W_1W_1^*=\mathbf{0}$. Thus, the first row and column of \eqref{inter} are identically zero and so that positivity  condition is equivalent to 
\[\big[\ell^{x_1}(x_i, x_j)-k(x_i, x_j)W_iW_j^*\big]_{2\le i, j\le n}\succeq 0\]
The existence of an interpolating multiplier $\Phi$ can now be obtained via the CP property for $(k, \ell)$ as in the previous case.
\end{proof}

We can now prove Theorem~\ref{generalnec}.

\begin{proof}[Proof of Theorem~\ref{generalnec}]
Combine Theorem~\ref{Schurpreserves} with Theorem~\ref{Shimnec}.
\end{proof}
 
Let $k, \ell, \tilde{s}$ be kernels on $X$ such that $k$ is non-vanishing, $\tilde{s}$ is a non-vanishing CP kernel and $\frac{\tilde{s}}{k}, \frac{\ell}{\tilde{s}}$ are both positive kernels. Further, let $M$ be any closed $\Mult(\mathcal{H}_{\ell})$-invariant subspace of $\mathcal{H}_{\ell}$ and denote its kernel by $\ell_M.$ By Lemma~\ref{invfactors}, we have 
\[\frac{\ell_M}{\tilde{s}}\succeq 0.\]
Taking the Schur product with $\frac{\tilde{s}}{k}$ gives
\begin{equation} \label{somewhatbetternec}
   \frac{\ell_M}{k}\succeq 0. 
\end{equation}
Thus, the existence of a CP kernel $\tilde{s}$ such that both $\frac{\tilde{s}}{k}, \frac{\ell}{\tilde{s}}$ are positive kernels forces $(k, \ell)$ to satisfy a much stronger positivity condition than the one obtained in the conclusion of equation~\eqref{generalnecpos} in  Theorem~\ref{generalnec}.  This stronger condition does not, however, imply $\tilde{s}$ is a strong Shimorin certificate as Proposition~\ref{necbutnotsuf} below shows. The obstruction is that,  to be a strong Shimorin certificate for $(k, \ell)$, the equality $k=(1-h)\tilde{s}$ requires (for $k, \tilde{s}$ non-vanishing)
\[\frac{\tilde{s}}{k}=\frac{1}{1-h},\]
and so $\frac{\tilde{s}}{k}$ needs to be not just a kernel, but a CP kernel as well. 

\begin{proposition}\label{necbutnotsuf}
There exist diagonal holomorphic kernels $k,\ell,\tilde{s}$ such that $\tilde{s}$ is a CP kernel and $\frac{\tilde{s}}{k}, \frac{\ell}{\tilde{s}}$ are positive kernels, but $(k,\ell)$ is not a CP pair.
\end{proposition}

\begin{proof}
 For convenience, let $x=z\overline{w}.$ 
Set
\[
 k(z, w)= \frac{1}{(1+x+4x^2)(1-3x)} = \frac{1}{1-2x+x^2-12x^3} =\frac{1}{(1-x)^2[1-\frac{12x^3}{(1-x)^2}]} \succeq 0.
\]
 The first four coefficients of the power series expansion for $k$ are $1,2,3,16,$ starting with $k_0.$  Thus, 
 computing the master certificate,  $\mt_1=k_1=2;$ since $k_2-k_1\mt_1=-1,$ we have $\mt_2=0;$ and next, since
 $k_3-k_2\mt_1 = 16-6=10>0,$ we have $\mt_3=10.$ Thus $\mt = 2x+10x^3 + \dots \, .$ 
Set $s=\frac{1}{1-\mt}$ and 
\[
 \tilde{s}(z, w)=\frac{1}{1-3x}.
\]
 Thus,
\[
 \frac{\tilde{s}}{k} = 1+x+4x^2 \succeq 0.
\]
On the other hand, the coefficient of the cubic term in the power series expansion of 
\[
f(z,w)= \frac{\tilde{s}(z, w)}{s(z, w)}=\frac{1-2x-10x^3 -\dots}{1-3x}
\]
 is $-1$ and thus $f$ is not positive.  Hence, by Theorem~\ref{mainextmain}, $(k, \tilde{s})$ is not a CP pair, even though the CP kernel $\tilde{s}$ has $k$ as a factor. 
\end{proof}

\normalsize
\section{A Closer Look at the Bergman Kernel} \label{closerl}
\normalsize
This section investigates pairs of kernel $(\mathfrak{b},\ell),$ where $\mathfrak{b}$ is the Bergman kernel and $\ell$ is not necessarily diagonal.
It  also contains a proof of Theorem~\ref{bergnotmin}.

\subsection{Non-diagonal certificates} \label{bergcerts}
Let $\mathfrak{b}(z, w)=(1-z\overline{w})^{-2}$ denote the Bergman kernel on $\DD.$ As seen in Example~\ref{finallyanex}, the master certificate $\mt$ associated with $\mathfrak{b}$ is
\[s(z, w)=\frac{1}{1-\mt(z\overline{w})}=\frac{1}{1-2z\overline{w}},\hspace{0.2 cm} z, w\in\Omega^1_{\mt}=\frac{1}{\sqrt{2}}\DD.\]
This kernel  $s$ satisfies $\mathfrak{b}=(1-h)s,$ where $h(z, w)=(z\overline{w})^2\mathfrak{b}(z, w).$ By Theorem~\ref{mainextmain}, given any diagonal holomorphic kernel $\ell$ on a domain $\Omega\subseteq \CC,$ the pair $(\mathfrak{b}, \ell, \Omega)$ is CP if and only if $\Omega\subseteq\frac{1}{\sqrt{2}}\DD$ and $\frac{\ell}{s}\succeq 0$. In this subsection, we will compute further examples of CP kernels $\tilde{s}$ such that $\mathfrak{b}=(1-\tilde{h})\tilde{s}$ for some $\tilde{h}.$ In particular, we will utilize these examples to show that Theorem~\ref{mainextmain} fails if we do not assume $\ell$ is diagonal.

\begin{example}\rm
\label{eg:0} Let $\lambda\in\CC$ with $0<|\lambda|<2.$ Define $g:\DD\to\CC$ by
\[
\begin{split}
 g_\lambda(z)= g(z) = &\frac{1-\frac{1}{\mathfrak{b}(z,\lambda)}}{\sqrt{1-\frac{1}{\mathfrak{b}(\lambda,\lambda)}}}
\\  = & \frac{2z\overline{\lambda} -(z\overline{\lambda})^2}{\sqrt{2\lambda\overline{\lambda} -(\lambda\overline{\lambda})^2}}
\\ = & \frac{\overline{\lambda} z}{|\lambda|} \frac{(2-z\overline{\lambda})}{\sqrt{2 -\lambda\overline{\lambda}}}.
\end{split}
\]
We have $|g|\le \frac{2+|\lambda|}{\sqrt{2-|\lambda|^2.}}$ on $\DD.$ We want $g$ to be a contractive multiplier, so restrict $\lambda$ to $0<|\lambda|<1$ and set \[\tOmega_{\lambda}=\{z\in \CC: |g_{\lambda}(z)|<1\}.\]
Notice that $\frac{1}{3}\DD\subseteq\tOmega_{\lambda}$ for every $0<|\lambda|<1$. Now, for $z, w\in\tOmega_{\lambda,}$ we have
\[
\begin{split}
 (g(z)\overline{g(w)}-1)\mathfrak{b}(z,w)+1
 = & \bigg [\frac{z\overline{w} (2 -z\overline{\lambda})\, (2-\lambda \overline{w})}
  {2 -\lambda\overline{\lambda}} -1\bigg ] \mathfrak{b}(z,w)+1.
\end{split}
\]
The common denominator above is 
\[
 (1-z\overline{w})^2 \,  (2 -\lambda\overline{\lambda}).
\]
The numerator is,
\[
z\overline{w}(2-z\overline{\lambda})(2-\lambda\overline{w}) -  (2-\lambda\overline{\lambda})  +  (2-\lambda\overline{\lambda})(1-z\overline{w})^2 = 2z\overline{w}(z-\lambda)\overline{(w-\lambda)}.
\]
Thus,
\begin{equation}
 \label{e:eg0:1}
(g(z)\overline{g(w)}-1)\mathfrak{b}(z,w)+1
  = 2 z\overline{w}
  \frac{(z-\lambda)\overline{(w-\lambda)}}{(1-z\overline{w})^2\, (2 -\lambda\overline{\lambda})}=h_\lambda(z,w),
\end{equation}
 which is positive. In other words, for every $0<|\lambda|<1,$ there exists a positive kernel $h_{\lambda}$ on $D_{\lambda}$ such that
 \begin{equation}\label{anothereg}
     \mathfrak{b}(z, w)=\frac{1-h_{\lambda}(z, w)}{1-g_{\lambda}(z)\overline{g_{\lambda}(w)}}=(1-h_{\lambda}(z, w))s_{\lambda}(z, w),
 \end{equation}
where $s_{\lambda}(z, w)=(1-g_{\lambda}(z)\overline{g_{\lambda}(w)})^{-1}$ defines a CP kernel on $D_{\lambda}$.
\end{example}

\begin{remark}\rm
\label{non-inclusion}
   Choosing $0<\lambda<1$ in Example \ref{eg:0}, it can be verified that $|g_{\lambda}(\frac{1}{\sqrt{2}})|<1$. Thus, $\frac{1}{\sqrt{2}}\DD$ does not contain  
   $\tOmega_\lambda$, even though $(\mathfrak{b}, s_\lambda, D_{\lambda})$ is a CP pair.  Comparing with 
    the inclusions \eqref{inclusions}, we see that the conclusion of Theorem~\ref{mainextmain} does not apply to $(\mathfrak{b}, s_\lambda, D_{\lambda})$. At the same time, observe that given $0\neq \lambda=e^{i\theta}|\lambda|,$ we have $|g_{\lambda}(-e^{i\theta}\frac{1}{\sqrt{2}})|>1,$ hence $D_{\lambda}$ can never contain $\frac{1}{\sqrt{2}}\DD$.
\qed
\end{remark}

We now prove Theorem~\ref{bergnotmin}, which we restate here for the reader's convenience. 

\begin{theorem}\label{bergnotminrest}
Let  $\mathfrak{b}$ denote the Bergman kernel on $\DD$. There exists an $0<r<1$, and a one-parameter family of CP kernels $\{s_{\lambda}\}_{\lambda\in\Lambda}$
on $r\DD$ with the following properties:
\begin{enumerate}[(i)]
    \item \label{i:bergman:1} For every $\lambda\in\Lambda,$ there exists a kernel $h_{\lambda}$ on $r\mathbb{D}$ such that 
\[\mathfrak{b}=(1-h_{\lambda})s_{\lambda};\]

\item \label{i:bergman:2}  There do not exist kernels $s,h,f_\lambda$ on $r\mathbb{D}$  such that $s$ is a CP kernel, $\mathfrak{b}=(1-h)s$  and,
 for every $\lambda\in \Lambda,$ 
\begin{equation}\label{namexx}
 s_{\lambda}=sf_{\lambda}.
\end{equation}
\end{enumerate}
\end{theorem} 
\begin{proof}
For  $\lambda\in\CC$ with $0<|\lambda|<\frac13,$ define $g_{\lambda}, s_{\lambda}, h_{\lambda}$ as in Example~\ref{eg:0}.  Further, set $\Lambda=\{0<|\lambda|<\frac13\}$ and $r=\frac{1}{3}.$ It is straightforward to verify that $\frac13\DD\subseteq D_\lambda.$ Thus $\{s_{\lambda}\}_{\lambda\in\Lambda}$ is a family of CP kernels defined on $\frac{1}{3}\DD$ and such that \eqref{anothereg} holds, giving item~\eqref{i:bergman:1}. 
    \par
 Now suppose there is a kernel $s$ defined on  $\frac{1}{3}\DD$ such that 
\begin{equation} \label{outofnames}
\mathfrak{b}=(1-h)s
\end{equation}
and, for each $\lambda,$ there exists a (PsD) kernel $f_\lambda$ on $\frac{1}{3}\DD$ such that $s_\lambda = sf_\lambda.$
Fix $\lambda\in \Lambda.$ 
Since $\mathfrak{b}$ is non-vanishing, so is $s.$ Since $b=(1-h_\lambda)s_\lambda,$ it follows that
\[
0\preceq  f_\lambda = \frac{s_\lambda}{s} =\frac{1-h}{1-h_\lambda}.
\]
 From $h_\lambda(0,0)=h_\lambda(0,\lambda)=h_\lambda(\lambda,\lambda)=0,$ it follows that 
\begin{equation}\label{somat}
0\preceq   \begin{pmatrix} f_\lambda(0,0) & f_\lambda(\lambda,0)\\f_\lambda(0,\lambda) & f_\lambda(\lambda,\lambda)
   \end{pmatrix} = \begin{pmatrix} 1-h(0,0) & 1-h(\lambda,0)\\ 1-h(0,\lambda) & 1-h(\lambda,\lambda)\end{pmatrix}.
\end{equation}
But $h$ is a kernel (PsD), so the resulting inequality
\[
 \begin{pmatrix} h(0,0) & h(\lambda,0) \\ h(0,\lambda)  & h(\lambda,\lambda) \end{pmatrix}
   \preceq \begin{pmatrix} 1\\ 1 \end{pmatrix} \, \begin{pmatrix} 1 & 1 \end{pmatrix}
\]
implies $h(0,0)=h(\lambda,0)=h(0,\lambda)=h(\lambda,\lambda).$
  Since $\lambda$ is arbitrary in $\Lambda$, it follows that $h(\lambda,\mu)=c$ for all $\lambda,\mu\in \frac{1}{3}\DD,$ where $0\le c=B(0,0)<1.$  Thus
$\mathfrak{b}=(1-c)s$ is a complete Pick kernel over $\frac{1}{3}\DD$, which is easily seen to be false. Thus, item~\eqref{i:bergman:2} holds and the proof is complete.
\end{proof}
Thus, \eqref{inclusions} and the implication (i)$\Rightarrow$ (iv) in Theorem \ref{mainextmain} might both fail if $\ell$ is non-diagonal. 

Example~\ref{eg:0} can also be used to show that Corollary~\ref{canoncorol} fails if we consider general (possibly non-diagonal) strong Shimorin certificates.
\begin{corollary} \label{nonmin}
Letting  \[\ell(z, w)=\frac{1}{1-9z\overline{w}},\hspace{0.3 cm} z, w\in\frac{1}{3}\mathbb{D}, \]
$(\mathfrak{b}, \ell, \frac{1}{3}\DD)$ is a CP pair, but there does not exist a strong Shimorin certificate $s$ for $(\mathfrak{b}, \ell,  \frac{1}{3}\DD)$ such that, if $\tilde{s}$ is any strong Shimorin certificate for  $(\mathfrak{b}, \ell,  \frac{1}{3}\DD)$, then $s$ is a factor of $\tilde{s}$.
\end{corollary}
\begin{proof}
Define $g_{\lambda}:\frac{1}{3}\DD\to\CC$ as in Example~\ref{eg:0}, with $0<|\lambda|<\frac13$. 
Thus
\[s_{\lambda}(z, w)=\frac{1}{1-g_{\lambda}(z)\overline{g_{\lambda}(w)}}\]
defines a CP kernel on $\frac{1}{3}\mathbb{D}$, for any $0<|\lambda|<\frac13.$ Further, note that
\[ \frac{\ell(\frac{1}{3}z, \frac{1}{3}w)}{s_{\lambda}(\frac{1}{3}z, \frac{1}{3}w)}=\frac{1-g_{\lambda}(\frac{1}{3}z)\overline{g_{\lambda}(\frac{1}{3}w)}}{1-z\overline{w}}\]
is a kernel on $\mathbb{D}.$ Hence, $\frac{\ell}{s_{\lambda}}$ is a kernel on $\frac{1}{3}\mathbb{D}$, for any $0<|\lambda|<\frac13.$ Combined with \eqref{anothereg}, it follows that $s_{\lambda}$ is a strong Shimorin certificate for $(\mathfrak{b}, \ell,  \frac{1}{3}\DD)$. Assuming that there exists a strong Shimorin certificate $s$ as described in the statement of the corollary, we obtain $\mathfrak{b}=(1-h)s$ and that $s$ is a factor of $s_{\lambda}$ for every $0<|\lambda|<\frac{1}{3}.$ Theorem~\ref{bergnotminrest} then yields a contradiction.
\end{proof}

We conclude this subsection with another example of a (necessarily non-diagonal) certificate for $\mathfrak{b}$ that violates the conclusions of Theorem~\ref{introextmain}.
\begin{example}\rm \label{lastex}

 Let 
\[
\begin{split}
 P(z,w) & = z\overline{w} \left [ 3-2z-2\overline{w} + 2 z\overline{w} \right ] \, 
   + \,  8 \frac{\zw^3}{1-z\overline{w}}
\\ & =  z\overline{w} [1 + 2(1-z)(1-\overline{w})] 
   \, 
   + \,  8 \frac{(z\overline{w})^3}{1-z\overline{w}} \succeq 0.
\end{split}
\]
Set
\[
 s=\frac{1}{1-P}.
\]
 Thus, $s$ is a CP kernel with $s(z,0)=1.$ 
 We claim
\begin{equation}
\label{e:funnyP}
 1-\frac{k}{s} \succeq 0,
\end{equation}
but 
\[
  \frac{1}{k}-\frac{1}{s}\not\succeq 0,
\]
 in contrast to the kernels $G_\lambda=g_{\lambda}\overline{g_{\lambda}}$ from Example~\ref{eg:0},
 where $s_\lambda=(1-G_\lambda)^{-1}$ satisfies 
 $s_\lambda(z,0)=1$ and, by equation~\eqref{e:eg0:1}, 
\[
 \frac{1}{k}- \frac{1}{s_\lambda} \succeq 0,
\]
 (which, of course, implies $1-\frac{k}{s_\lambda}\succeq 0$).
 Observe,
\[
 \frac{1}{k}-\frac{1}{s} =  z\overline{w} [ 1 -2z -2\overline{w} +3 z\overline{w} ] 
   \, 
   + \,  8 \frac{\zw^3}{1-z\overline{w}}  \not\succeq 0.
\]
 To establish the inequality of equation~\eqref{e:funnyP}, 
 compute,  letting $x=z\overline{w},$
\[
\begin{split}
 1- \frac{k}{s} &  = \frac{(1-x)^2 - (1-P)}{(1-x)^2}
\\ & = (P-2x+x^2) \, [1+2x+3x^2 +\cdots]
\\ & = x (1-2z-2\overline{w} + 3x + 8x^2+8x^3 +\cdots) \, [1+2x+3x^2 +\cdots]
\\ & = x(1+3x+8x^2+ 8x^3 + \cdots) \,  [1+2x+3x^2 +\cdots]
 \, - \, x (2z+2\overline{w}) [1+2x+3x^2 +\cdots]
\\ & = x \sum_{n=0}^\infty [(n+1)+3n + 8((n-1)+(n-2)+\dots 1)] x^n  \, - \,
   2 x\sum_{n=0}^\infty (z+\overline{w}) (n+1)x^n
\\ & = x \sum_{n=0}^\infty (4n^2+1) x^{n} + \, - \,
   2 x\sum_{n=1}^\infty (z+\overline{w})n x^{n-1} 
\\ & =  x \left [ \sum_{n=1}^\infty [4n^2 x^n + x^{n-1}] \right ] 
 \, - \,  2 x\sum_{n=1}^\infty (z+\overline{w})n x^{n-1} 
\\ & =   x \left [ \sum_{n=1}^\infty   (1-2nz)(1-2n\overline{w}) \zw^{n-1}  \right ] \succeq 0.
\end{split} 
\]
 
 To see that  $(1-2z\overline{w})^{-1}$ is not a strong Shimorin certificate for $(k,s),$ 
 note that the coefficient matrix of  $(1-2z\overline{w})s$ up to degree $3$ is given by 
\begin{equation*}  
 \begin{blockarray}{ccccc}
 & 1 & z & z^2 & z^3\\ 
\begin{block}{c(cccc)}
1 & 1 & 0 & 0 & 0 \\
\overline{w} & 0 & 1 & -2 & 0 \\ 
\overline{w}^2 & 0 & -2 & 5 & -8 \\ 
\overline{w}^3 & 0 & 0 & -8 & 33 \\ 
\end{block}
\end{blockarray} \, , 
\end{equation*}
which is not positive.
\qed
\end{example}

\subsection{The largest domain disc for the Bergman kernel} \label{maxdom}
If  $\ell$ is a diagonal holomorphic kernel and  $(\mathfrak{b}, \ell, \Omega)$ is a CP pair, then  $\Omega\subseteq \frac{1}{\sqrt{2}}\mathbb{D}$ by Theorem~\ref{mainextmain}.  In this subsection, we prove that $\frac{1}{\sqrt{2}}\mathbb{D}$ is actually the maximal disc that can serve as the domain of \textit{any} holomorphic (but not necessarily diagonal) $\ell$ such that $(\mathfrak{b}, \ell)$ is a CP pair. Compare with Remark~\ref{non-inclusion}. For convenience, we will impose a few mild regularity assumptions on $\ell.$

\begin{definition}\label{regdef}
Let $\ell(z, w)=\sum_{i, j\ge 0}\ell_{ij}z^i\overline{w}^j$ be a holomorphic kernel defined on an open neighborhood of the origin. We  will say $\ell$ is \df{regular} if 
\begin{enumerate}[(i)]
    \item  \label{i:reg:1} every collection of kernel functions $\{\ell_{z_1}, \dots, \ell_{z_n}\}$ with $z_1, \dots, z_n$ distinct and sufficiently close to $0$ forms a linearly independent set and
    \item  \label{i:reg:2} all principal determinants of the coefficient matrix $\big(\ell_{ij}\big)^{\infty}_{i, j\ge 0}$ are non-zero. 
\end{enumerate}
\end{definition}
We point out that all diagonal holomorphic kernels with non-zero coefficients are regular.

 \begin{theorem}
\label{t:max:disc}
Let $\mathfrak{b}$ denote the Bergman kernel on $\DD$ and fix $0<r\le 1$. If $\ell$ is a holomorphic kernel on $r\DD$ such that $(\mathfrak{b}, \ell) $ is CP, then
\[r\le \frac{1}{\sqrt{2}}.\]
 \end{theorem}

 We need a few preliminary lemmas. Let $\ell(z, w)=\sum_{i, j\ge 0} \ell_{ij}z^i\overline{w}^j$ be any regular holomorphic kernel on $r\DD$. Given $n\ge 1$ and $0<u_1<u_2<\dots<u_n$ sufficiently close to $0$, define the kernels
$$\ell_{[1]}(z, w):=\ell(z, w)-\frac{\ell(z, u_1)\ell(u_1, w)}{\ell(u_1, u_1)} $$
and $$\ell_{[1, \dots, n]}(z, w):=\ell_{[1, \dots, n-1]}(z, w)-\frac{\ell_{[1, \dots, n-1]}(z, u_n) \ell_{[1, \dots, n-1]}(u_n, w)}{\ell_{[1, \dots, n-1]}(u_n, u_n)},$$
for every $n\ge 2.$ Item~\ref{i:reg:1} of Definition~\ref{regdef} ensures that all of the above denominators are non-vanishing.  We also define, for every $n \ge 0,$ the functions
\[\ell^{(n)}(z, w)=\sum_{i, j\ge n}\ell^{(n)}_{ij}z^i\overline{w}^j \]
recursively by setting  $\ell^{(0)}\equiv \ell$ and, for $n\ge 1$, 
\[\ell^{(n)}_{ij}=\ell^{(n-1)}_{ij}-\frac{\ell^{(n-1)}_{i(n-1)}\ell^{(n-1)}_{(n-1)j}}{\ell^{(n-1)}_{(n-1)(n-1)}}.\]
Note that the denominators are non-vanishing because of item~\eqref{i:reg:2} from Definition~\ref{regdef}. \par 
Observe that the kernels $\ell_{[1, \dots, n]}$ are generated by considering Schur complements of matrices of the form $\big[\ell(z_i, z_j)\big]$, while each $\ell^{(n)}$ is generated by taking Schur complements of the coefficient matrix $\big(\ell_{ij}\big)$. The connection between these two types of Schur complements is given by the following lemma.
\begin{lemma} \label{lemma1}
 $\lim_{u_n\to 0}\cdots \lim_{u_1\to 0}\ell_{[1, \dots, n]}(z, w)=\ell^{(n)}(z, w)$ pointwise on $r\DD\times r\DD.$
\end{lemma}

\begin{proof}
We proceed by induction. The base case $n=1$ follows from observing that 
\[\lim_{u_1\to 0}\frac{\ell(z, u_1)\ell(u_1, w)}{\ell(u_1, u_1)}=\sum_{i, j\ge 0} \frac{\ell_{i0}\ell_{0j}}{\ell_{00}}z^i\overline{w}^j.\]
Now, assume that the statement of the lemma holds for $n\ge 1.$ We then have 
\begin{align*}
  &\lim_{u_{n+1}\to 0}\lim_{u_n\to 0}\cdots \lim_{u_1\to 0}\ell_{[1, \dots, n+1]}(z, w) \\  
=&\lim_{u_{n+1}\to 0}\lim_{u_n\to 0}\cdots \lim_{u_1\to 0}\bigg(\ell_{[1, \dots, n]}(z, w)-\frac{\ell_{[1, \dots, n]}(z, u_{n+1}) \ell_{[1, \dots, n]}(u_{n+1}, w)}{\ell_{[1, \dots, n]}(u_{n+1}, u_{n+1})} \bigg) \\
=&\lim_{u_{n+1}\to 0}\bigg(\ell^{(n)}(z, w)-\frac{\ell^{(n)}(z, u_{n+1}) \ell^{(n)}(u_{n+1}, w)}{\ell^{(n)}(u_{n+1}, u_{n+1})} \bigg)  \\
=&\sum_{i, j\ge n}\bigg(\ell^{(n)}_{ij}-\frac{\ell^{(n)}_{in}\ell^{(n)}_{nj}}{\ell^{(n)}_{nn}} \bigg)z^i\overline{w}^j \\
=&\sum_{i, j\ge n+1}\bigg(\ell^{(n)}_{ij}-\frac{\ell^{(n)}_{in}\ell^{(n)}_{nj}}{\ell^{(n)}_{nn}} \bigg)z^i\overline{w}^j \\
=&~\ell^{(n+1)}(z, w). \qedhere
\end{align*}
\end{proof}
We now combine Lemma~\ref{lemma1} with Theorem~\ref{generalnec} to obtain new necessary conditions for a pair of kernels $(k, \ell),$ with $\ell$ holomorphic, to be CP.

 \begin{lemma}\label{lemma2}
    Assume $k, \ell$ are two kernels such that $\ell$ is regular holomorphic on $r\DD.$ If  $(k, \ell, r\DD)$ is CP, then, with notation as above and $n$ a non-negative integer,
\[
 \frac{\ell^{(n)}}{k}\succeq  0. 
\]
 \end{lemma}
\begin{proof}
The case $n=0$ follows immediately from Theorem~\ref{Shimnec}. Now, assume $n\ge 1$ and let $0<u_1<\dots<u_n.$  By Theorem~\ref{generalnec}, we know that $(k, \ell)$ being CP implies that $(k, \ell^z)$ is also CP, for every $z\in r\DD$. Thus, $(k, \ell^{u_1})=(k, \ell_{[1]})$ must be CP as well. Continuing inductively, we obtain that $(k, \ell_{[1, \dots, n]})$ is CP, hence, by Theorem~\ref{Shimnec}, 
 \[\frac{\ell_{[1, \dots, n]}}{k}\succeq  0.\]
Since pointwise limits preserve positive kernels, this last condition yields (in combination with Lemma~\ref{lemma1})
\[\frac{\ell^{(n)}}{k}=\lim_{u_n\to 0}\cdots\lim_{u_1\to 0} \frac{\ell_{[1, \dots, n]}}{k}\succeq 0.\]
\end{proof}
We will now apply Lemma~\ref{lemma2} to pairs of the form $(\mathfrak{b}, \ell).$
\begin{lemma} \label{lemma3}
If  $(\mathfrak{b}, \ell, r\DD)$ is a CP pair and  $\ell$ regular holomorphic,  then, with notation as above and $n$ a positive integer,
 \[
\ell^{(n)}_{nn}\ge 2\ell^{(n-1)}_{(n-1)(n-1)}. 
\]
\end{lemma}
\begin{proof}
 Let $n\ge 1$ be given. 
 The coefficients of $z^n\overline{w}^n, z^n\overline{w}^{n+1}, z^{n+1}\overline{w}^n, z^{n+1}\overline{w}^{n+1} $ in the expansion of $\frac{\ell^{(n-1)}}{\mathfrak{b}}$ form the $2\times 2$ matrix
 \[\begin{bmatrix}
    \ell^{(n-1)}_{(n-1)(n-1)} & \ell^{(n-1)}_{(n-1)n} \\ \ell^{(n-1)}_{n(n-1)} & \ell^{(n-1)}_{nn}-2 \ell^{(n-1)}_{(n-1)(n-1)}
\end{bmatrix} \]
Since, by Lemma~\ref{lemma2},  $\frac{\ell^{(n-1)}}{\mathfrak{b}}$  is positive, the same must be true for the above $2\times 2$ minor. Thus,$$\big(\ell^{(n-1)}_{nn}-2 \ell^{(n-1)}_{(n-1)(n-1)}\big) \ell^{(n-1)}_{(n-1)(n-1)}\ge \ell^{(n-1)}_{n(n-1)}\ell^{(n)}_{(n-1)n},$$
which is precisely what we want.
\end{proof}

We can now prove the main result of this subsection. 
\begin{proof}[Proof of Theorem~\ref{t:max:disc}]
     Assume $\ell$ is regular holomorphic on $r\DD$ with $r>\frac{1}{\sqrt{2}}$ and $(k, \ell)$ is a CP pair. By Lemma~\ref{lemma3}, we know that 
     \[\ell^{(n+1)}_{(n+1)(n+1)}\ge 2\ell^{(n)}_{nn}\ge \dots \ge 2^{n+1} \ell_{00},\]
     for all $n\ge 1,$ where $\ell_{00}>0.$ 
It is evident that, for all $n$, \[\ell^{(n+1)}_{(n+1)(n+1)}\le \ell^{(n)}_{(n+1)(n+1)}\le \dots \le \ell_{(n+1)(n+1)}.\]
     Combining this with the previous inequalities, we obtain
     \begin{equation}\label{badpowerseries}
         \ell_{nn}\ge 2^n\ell_{00},
     \end{equation}
     for all $n\ge 1.$ Now, since $r>\frac{1}{\sqrt{2}}$, we obtain  the continuous function $F: [0, 2\pi)\to\mathbb{C}$ given by 
          \[F(\theta)=\ell\big(\frac{1}{\sqrt{2}}e^{i\theta}, \frac{1}{\sqrt{2}}e^{i\theta}\big)=\sum_{r, p\ge 0}\ell_{rp}\frac{e^{i(r-p)\theta}}{(\sqrt{2})^{r+p}}.\] Integrating term-wise with respect to $\theta$  annihilates the  off-diagonal terms and thus 
     \[\int_{0}^{2\pi}F(\theta)\ d\theta=\sum_{i\ge 0}\frac{\ell_{ii}}{2^i}\]
     is finite, contradicting (\ref{badpowerseries}).
\end{proof}

\begin{remark} \label{twokeys}
There are two key points in the proof of Theorem~\ref{t:max:disc}. 
 First, given a CP pair $(\mathfrak{b}, \ell)$ with $\ell$ diagonal holomorphic, it is possible to arrive at the conclusion that $s=(1-\mt)^{-1}$ is a factor of $\ell$ by only considering conditions of the form \[\frac{\ell_M}{\mathfrak{b}}\succeq 0,\] where $M$ is a subspace of $\mathcal{H}_{\ell}$ determined by the vanishing of all derivatives up to some order. In particular, the necessary conditions of Lemma \ref{lemma2} coincide with \eqref{nicc} if we assume that $\ell$ is diagonal holomorphic. 
 Second, the disc is circularly symmetric, so the off diagonal terms of $\ell$
 can be annihilated by averaging.  

\end{remark}

\begin{remark} \label{verygen}
Theorem~\ref{t:max:disc} extends to more general Bergman-like kernels. In particular, given a holomorphic CP pair $(b, \ell, \Omega),$ where $b$ is as in Examples~\ref{finallyanex}, \ref{finallyanex2}, one obtains that
$\Omega$ does not properly contain the Reinhardt domain $\Omega_\mt,$
where $\mt$ is, as usual, the master certificate associated with $b$. To get there, we observe again that the conclusion ``$s=(1-\mt)^{-1}$ is a factor of $\ell$" can be reached by only considering conditions of the form $\ell_M/b\succeq 0,$  where $M$ is a subspace of $\mathcal{H}_{\ell}$ determined by the vanishing of all Taylor coefficients up to some order. Thus, the two observations given in Remark \ref{twokeys} carry over to the $(b, \ell, \Omega)$ setting and the proof method of Theorem~\ref{t:max:disc} generalizes accordingly. We omit the details. 
\end{remark}

\normalsize
\section{A Reformulation of Theorem~\ref{introextmain}} \label{reform}
\normalsize

The algorithm from the proof of Theorem~\ref{almostCCPtoShim} gives us the following set of necessary and sufficient conditions for 
 $(k, \ell)$ to be a CP pair that make no explicit reference to the master certificate associated with $k$. It is our hope that this version of the theorem will lend itself more naturally to generalization.
For simplicity, we treat the single variable case; that is $\vg=1.$

\begin{theorem} \label{sometheorem}
     A pair $(k,\ell)$ of normalized diagonal holomorphic kernels on an open disc centered at $0$, so that
    \[k(z, w)=1+\sum_{j\ge 1}k_j(z\overline{w})^j, \hspace*{0.3 cm} \ell(z, w)=1+\sum_{j\ge 1}\ell_j(z\overline{w})^j,\]
    is a CP pair if and only if \[\frac{\ell}{k}\succeq 0\] and also the following conditions are satisfied: 
    For every strictly increasing (infinite) sequence $0\le m_0<m_1<m_2<\dots$ and for all $j\ge 0,$ 
    \[
      \frac{\ell_{(m_0, m_1, \dots, m_j)}}{k}\succeq 0,  
    \]
    where the kernels $\ell_{(m_0, m_1, \dots, m_j)}$ are defined inductively as follows:
    \begin{enumerate}[(i)]
        \item First, set $$\ell_{(m_0)}(z, w):=\ell(z, w)-\bigg(\frac{\ell}{k}\bigg)_{m_0}(z\overline{w})^{m_0};$$
        \item  Assuming $\ell_{(m_0, m_1, \dots, m_j)}$ has been defined, put 
        $$\ell_{(m_0, m_1, \dots, m_{j+1})}(z, w):=\ell_{(m_0, m_1, \dots, m_j)}(z, w)-\bigg(\frac{\ell_{(m_0, m_1, \dots, m_j)}}{k}\bigg)_{m_{j+1}}(z\overline{w})^{m_{j+1}}.$$
    \end{enumerate}
\end{theorem}

\begin{proof}
Define the sequence $\{t_n\}_{n\ge 1}$ as follows: choose either $t_1=0$ or $t_1=k_1$ and then define, inductively,
\begin{equation}\label{maybeneg}
\text{ either}\hspace*{0.2 cm} t_n=k_n-\sum_{j=1}^{n-1}t_j k_{n-j} \hspace*{0.2 cm} \text{ or} \hspace*{0.2 cm} t_n=0.    
\end{equation}
Thus, we have a family of possible sequences here, one of which is the ``distinguished" one originating from the master certificate $\mt$, as in Definition~\ref{mastercert}. In the general case, we might have negative $t_i$'s too. 
\par 
Now, fix such a sequence $\{t_j\}_{j\ge 1}^{n+1}$, where $n\ge 1$. Set $t_0=1$ and define $v_m$ for $0\le m\le n+1$ inductively as follows. First, set
$$v_0= \begin{cases}
    0, \hspace*{1 cm} \text{ if }t_{n+1}=0 \\ 
    1, \hspace*{1 cm} \text{ if }t_{n+1}\neq 0 
\end{cases} $$
and then, for $1\le m\le n+1,$ put

$$v_m= \begin{cases}
    0, \hspace*{4 cm} \text{ if }t_{n+1-m}=0 \\ 
    \ell_m -\sum_{j=0}^{m-1}v_jk_{m-j}, \hspace*{1.2 cm} \text{ if }t_{n+1-m}\neq 0. 
\end{cases} $$
Finally, we also set 
$$a_m= \begin{cases}
    0, \hspace*{1 cm} \text{ if }t_{n+1-m}=0 \\ 
    1, \hspace*{1 cm} \text{ if }t_{n+1-m}\neq 0 
\end{cases} $$
and $\tilde{a}_m=1-a_m$.
Thus, 
$$v_m=a_m\bigg(\ell_m-\sum_{j=0}^{m-1}v_jk_{m-j} \bigg).$$
An application of Lemma~\ref{l:someare0} (with $d=n+1$ and $S=\{0\le k\le n+1 : t_{n+1-k}=0\}$) then yields $v_0, v_1, \dots, v_{n+1}\ge 0$ and, in particular, 

\begin{equation}\label{fund}
    v_{n+1}=\ell_{n+1}-\sum_{m=0}^nv_m k_{n+1-m}\ge 0.
\end{equation}
We now establish new notation. Set 
$$\ell_{[0]}=\ell -\tilde{a}_0\bigg(\frac{\ell}{k}\bigg)_0 $$
and then, inductively,
\begin{equation}\label{newnot}
 \ell_{[0, 1, \dots, m]}(z, w)=\ell_{[0, 1, \dots, m-1]}(z, w)-\tilde{a}_m \bigg(\frac{\ell_{[0, 1, \dots, m-1]}}{k}\bigg)_m (z\overline{w})^m,    
\end{equation}
for all $0\le m\le n$ (remember that $n$ is fixed). Our goal will be to show that 
\begin{equation}\label{coeff}
 \bigg(\frac{\ell_{[0, 1, \dots, n]}}{k}\bigg)_{n+1}=\ell_{n+1}-\sum_{m=0}^nv_m k_{n+1-m}.  
\end{equation}

\par 
We proceed with complete induction. First, take $n=0$. If $t_1=0,$ we have 
$$\bigg(\frac{\ell_{[0]}}{k}\bigg)_1=\bigg(\frac{\ell-\big(\frac{\ell}{k}\big)_0}{k}\bigg)_1=\bigg(\frac{\ell-1}{k}\bigg)_1=\ell_1=\ell_1-v_0k_1.$$
Similarly, if $t_1\neq 0$, we have 
$$\bigg(\frac{\ell_{[0]}}{k}\bigg)_1=\bigg(\frac{\ell}{k}\bigg)_1=\ell_1-k_1=\ell_1-v_0k_1,$$
as desired. Now, for the inductive step, assume we have showed
\begin{equation}\label{indstep}
\bigg(\frac{\ell_{[0, 1, \dots, q]}}{k}\bigg)_{q+1}=\ell_{q+1}-\sum_{m=0}^qv_m k_{q+1-m}, \hspace*{0.2 cm} \forall\text{ } 0\le q\le \rho,
\end{equation}
Setting \[\frac{1}{k(z, w)}=1+\sum_{j\ge 1}b_j (z\overline{w})^j,\] we  compute 
\begin{align*}
 \bigg(\frac{\ell_{[0, 1, \dots, \rho+1]}}{k}\bigg)_{\rho+2}=&\Bigg(\frac{\ell_{[0, 1, \dots, \rho]}-\tilde{a}_{\rho+1}\Big(\cfrac{\ell_{[0, 1, \dots, \rho]}}{k} \Big)_{\rho+1}(z\overline{w})^{\rho+1}}{k} \Bigg)_{\rho+2} \\ 
=&\big(\ell_{[0, 1, \dots, \rho]}\big)_{\rho+2}+b_1\bigg(\big(\ell_{[0, 1, \dots, \rho]}\big)_{\rho+1}-\tilde{a}_{\rho+1}\bigg(\cfrac{\ell_{[0, 1, \dots, \rho]}}{k} \bigg)_{\rho+1} \bigg) \\ &+b_2\big(\ell_{[0, 1, \dots, \rho]}\big)_{\rho}+b_3\big(\ell_{[0, 1, \dots, \rho]}\big)_{\rho-1}+\dots +b_{\rho+2}\big(\ell_{[0, 1, \dots, \rho]}\big)_{0} \\
    =&\big(\ell_{[0, 1, \dots, \rho]}\big)_{\rho+2}+b_1\bigg(\big(\ell_{[0, 1, \dots, \rho]}\big)_{\rho+1}-\tilde{a}_{\rho+1}\bigg(\cfrac{\ell_{[0, 1, \dots, \rho]}}{k} \bigg)_{\rho+1} \bigg)\\ 
        &+b_2\bigg(\big(\ell_{[0, 1, \dots, \rho-1]}\big)_{\rho}-\tilde{a}_{\rho}\bigg(\cfrac{\ell_{[0, 1, \dots, \rho-1]}}{k} \bigg)_{\rho} \bigg)\\
    &+b_3\big(\ell_{[0, 1, \dots, \rho-1]}\big)_{\rho-1}+\dots+b_{\rho+1}\big(\ell_{[0, 1]}\big)_{1}+b_{\rho+2}\big(\ell_{[0]}\big)_{0}     \\
=&\big(\ell_{[0, 1, \dots, \rho]}\big)_{\rho+2}+b_1\bigg(\big(\ell_{[0, 1, \dots, \rho]}\big)_{\rho+1}-\tilde{a}_{\rho+1}\bigg(\cfrac{\ell_{[0, 1, \dots, \rho]}}{k} \bigg)_{\rho+1} \bigg)\\
&+b_2\bigg(\big(\ell_{[0, 1, \dots, \rho-1]}\big)_{\rho}-\tilde{a}_{\rho}\bigg(\cfrac{\ell_{[0, 1, \dots, \rho-1]}}{k} \bigg)_{\rho} \bigg)\\
    &+b_3\bigg(\big(\ell_{[0, 1, \dots, \rho-2]}\big)_{\rho-1}-\tilde{a}_{\rho-1}\bigg(\cfrac{\ell_{[0, 1, \dots, \rho-2]}}{k} \bigg)_{\rho-1} \bigg)\\
    &+\dots
    +b_{\rho+1}\bigg(\big(\ell_{[0]}\big)_{1}-\tilde{a}_{1}\bigg(\cfrac{\ell_{[0]}}{k} \bigg)_{1} \bigg) 
    +b_{\rho+2}(1-\tilde{a}_0)\\
=&\ell_{\rho+2}+b_1\bigg(\ell_{\rho+1}-\tilde{a}_{\rho+1}\bigg(\ell_{\rho+1}-\sum_{m=0}^{\rho} v_m k_{\rho+1-m}\bigg) \bigg)\\ 
&+b_2\bigg(\ell_{\rho}-\tilde{a}_{\rho}\bigg(\ell_{\rho}-\sum_{m=0}^{\rho-1} v_m k_{\rho-m} \bigg)  \bigg)\\
    &+b_3\bigg(\ell_{\rho-1}-\tilde{a}_{\rho-1}\bigg(\ell_{\rho-1}-\sum_{m=0}^{\rho-2}v_m k_{\rho-1-m} \bigg)  \bigg)\\
    &+\dots \\
    &+b_{\rho+1}\big(\ell_{1}-\tilde{a}_{1}\big(\ell_1-v_0k_1\big) \big) +b_{\rho+2}v_0.
\end{align*}
\noindent
Since $b_{\rho+2}+k_1b_{\rho+1}+\dots+ b_1k_{\rho+1} +k_{\rho+2}=0,$ this last equality yields
\begin{align*}
 \bigg(\frac{\ell_{[0, 1, \dots, \rho+1]}}{k}\bigg)_{\rho+2}=&\ell_{\rho+2}-v_0k_{\rho+2}+b_1\bigg(\ell_{\rho+1}-v_0k_{\rho+1}-\tilde{a}_{\rho+1}\bigg(\ell_{\rho+1}-\sum_{m=0}^{\rho} v_m k_{\rho+1-m}\bigg) \bigg) \\
&+b_2\bigg(\ell_{\rho}-v_0k_{\rho}-\tilde{a}_{\rho}\bigg(\ell_{\rho}-\sum_{m=0}^{\rho-1} v_m k_{\rho-m} \bigg)  \bigg) \\
&+b_3\bigg(\ell_{\rho-1}-v_0k_{\rho-1}-\tilde{a}_{\rho-1}\bigg(\ell_{\rho-1}-\sum_{m=0}^{\rho-2}v_m k_{\rho-1-m} \bigg)  \bigg) \\
&+\dots \\ 
 &+b_{\rho+1}\big(\ell_{1}-v_0k_1-\tilde{a}_{1}\big(\ell_1-v_0k_1\big) \big).
\end{align*}
\noindent
Thus,
\begin{align*}
\bigg(\frac{\ell_{[0, 1, \dots, \rho+1]}}{k}\bigg)_{\rho+2}=&\ell_{\rho+2}-v_0k_{\rho+2}+b_1\bigg(a_{\rho+1}\big(\ell_{\rho+1}-v_0k_{\rho+1}\big)+\tilde{a}_{\rho+1}\sum_{m=1}^{\rho} v_m k_{\rho+1-m}  \bigg) \\
&+b_2\bigg(a_{\rho}\big(\ell_{\rho}-v_0k_{\rho}\big)+\tilde{a}_{\rho}\sum_{m=1}^{\rho-1} v_m k_{\rho-m}   \bigg) \\
&+\dots+b_{\rho+1}v_1.    
\end{align*}
\noindent
Proceeding by induction, assume we have showed for $0\le t\le \rho-1,$
\begin{align*}
\bigg(\frac{\ell_{[0, 1, \dots, \rho+1]}}{k}\bigg)_{\rho+2}=&\ell_{\rho+2}-\sum_{m=0}^tv_mk_{\rho+2-m}  \\
 &+b_1\bigg(a_{\rho+1}\bigg(\ell_{\rho+1}-\sum_{m=0}^t v_m k_{\rho+1-m}\bigg)+\tilde{a}_{\rho+1}\sum_{m=t+1}^{\rho} v_m k_{\rho+1-m}  \bigg) \\
&+b_2\bigg(a_{\rho}\bigg(\ell_{\rho}-\sum_{m=0}^t v_m k_{\rho-m}\bigg)+\tilde{a}_{\rho}\sum_{m=t+1}^{\rho-1} v_m k_{\rho-m}   \bigg) \\
    &+\dots \\
 &+b_{\rho-t}\bigg(a_{t+2}\bigg(\ell_{t+2}-\sum_{m=0}^t v_m k_{t+2-m}\bigg)+\tilde{a}_{t+2}v_{t+1}k_1  \bigg) \\
 &+b_{\rho+1-t}v_{t+1}.    
\end{align*}
\noindent
Since $b_{\rho+1-t}+k_1b_{\rho-t}+\dots+k_{\rho+1-t}=0,$ this last equality allows us to write 

\begin{align*}
&\bigg(\frac{\ell_{[0, 1, \dots, \rho+1]}}{k}\bigg)_{\rho+2} \\ 
 =&\ell_{\rho+2}-\sum_{m=0}^{t+1}v_mk_{\rho+2-m} \\
 &+b_1\bigg(a_{\rho+1}\bigg(\ell_{\rho+1}-\sum_{m=0}^{t} v_m k_{\rho+1-m}\bigg)-v_{t+1}k_{\rho-t}+\tilde{a}_{\rho+1}\sum_{m=t+1}^{\rho} v_m k_{\rho+1-m}  \bigg) \\
&+b_2\bigg(a_{\rho}\bigg(\ell_{\rho}-\sum_{m=0}^t v_m k_{\rho-m}\bigg)-v_{t+1}k_{\rho-1-t}+\tilde{a}_{\rho}\sum_{m=t+1}^{\rho-1} v_m k_{\rho-m}   \bigg) \\
    &+\dots \\
 &+b_{\rho-t}\bigg(a_{t+2}\bigg(\ell_{t+2}-\sum_{m=0}^t v_m k_{t+2-m}\bigg)-v_{t+1}k_1+\tilde{a}_{t+2}v_{t+1}k_1  \bigg)\\
 =&\ell_{\rho+2}-\sum_{m=0}^{t+1}v_mk_{\rho+2-m}\\
 &+b_1\bigg(a_{\rho+1}\bigg(\ell_{\rho+1}-\sum_{m=0}^{t+1} v_m k_{\rho+1-m}\bigg)+\tilde{a}_{\rho+1}\sum_{m=t+2}^{\rho} v_m k_{\rho+1-m}  \bigg) \\
&+b_2\bigg(a_{\rho}\bigg(\ell_{\rho}-\sum_{m=0}^{t+1} v_m k_{\rho-m}\bigg)+\tilde{a}_{\rho}\sum_{m=t+2}^{\rho-1} v_m k_{\rho-m}   \bigg)\\
    &+\dots \\
 &+b_{\rho-t}v_{t+2}.   
\end{align*}
\noindent
 Setting $t=\rho-1$ yields 
\begin{align*}
 \bigg(\frac{\ell_{[0, 1, \dots, \rho+1]}}{k}\bigg)_{\rho+2} 
&=\ell_{\rho+2}-\sum_{m=0}^{\rho}v_mk_{\rho+2-m}
 +b_1v_{\rho+1} \\
 &=\ell_{\rho+2}-\sum_{m=0}^{\rho}v_mk_{\rho+2-m}
 -k_1v_{\rho+1} \\
 &=\ell_{\rho+2}-\sum_{m=0}^{\rho+1}v_mk_{\rho+2-m},\end{align*}
 which is \eqref{indstep} for $q=\rho+1.$ Thus, \eqref{indstep} holds for all $0\le q\le n.$ Setting $q=n$ then yields \eqref{coeff}, which, in view of \eqref{fund}, implies 
 \begin{equation}\label{runout}
 \bigg(\frac{\ell_{[0, 1, \dots, n]}}{k}\bigg)_{n+1}\ge 0,
     \end{equation}
for all $n\ge 0.$ \par 
We are now ready to conclude our proof. Let $0\le m_0<m_1<\dots<m_j$ and fix $n\ge 0.$ We want to show 
\begin{equation*}  
    \bigg(\frac{\ell_{(m_0, \dots, m_j)}}{k}\bigg)_n\ge 0.
\end{equation*}
If $n<m_0,$ then 
\[\bigg(\frac{\ell_{(m_0, \dots, m_j)}}{k}\bigg)_n= \bigg(\frac{\ell}{k} \bigg)_n\ge 0.\]
Otherwise, let $m_i$ denote the largest integer in $\{m_0, \dots, m_j\}$ such that $m_i\le n.$ Since 
\[   
 \bigg(\frac{\ell_{(m_0, \dots, m_j)}}{k}\bigg)_n=\bigg(\frac{\ell_{(m_0, \dots, m_i)}}{k}\bigg)_n,
\]
it suffices to show 
\begin{equation} \label{secgoal}
    \bigg(\frac{\ell_{(m_0, \dots, m_i)}}{k}\bigg)_n\ge 0.
\end{equation}
If $m_i=n,$ one obtains 
\[\bigg(\frac{\ell_{(m_0, \dots, n)}}{k}\bigg)_n=\bigg(\frac{\ell_{(m_0, \dots, m_{i-1})}}{k} -
\frac{1}{k}\bigg(\frac{\ell_{(m_0, \dots, m_{i-1})}}{k}\bigg)_n(z\overline{w})^n\bigg)_n=0,
\]
so \eqref{secgoal} is satisfied. Next, assume $m_i\le n-1.$  Define the sequence $\{t_{\rho}\}^n_{\rho=1}$ as in \eqref{maybeneg} and in such a way that $\tilde{a}_{\rho}=1$ if $\rho\in\{m_0, \dots, m_i\}$ and $0$ otherwise. In particular, we set \[t_{n-\rho}=k_{n-\rho}-\sum_{r=1}^{n-\rho-1}t_rk_{n-\rho-r}\]
if $\rho\notin\{m_0, \dots, m_j\}$ and $t_{n-\rho}=0$ otherwise. This way, we obtain  (using notation as in \eqref{newnot})
\[\frac{\ell_{(m_0, \dots, m_i)}}{k}=\frac{\ell_{[0, 1, \dots, n-1]}}{k}.\]
Equation~\eqref{runout} now yields
\[\bigg(\frac{\ell_{(m_0, \dots, m_i)}}{k}\bigg)_{n}=\bigg(\frac{\ell_{[0, 1, \dots, n-1]}}{k}\bigg)_n\ge 0,\]
as desired.
\end{proof}

\begin{remark} \label{remark81}
Assume $(k, \ell)$ is a normalized diagonal holomorphic CP pair. Choosing $m_k=k$, for all $k$, in Theorem~\ref{sometheorem}, one obtains 
    \begin{equation}\label{nicc}
        \frac{\ell_{[0, 1, \dots, j]}}{k}=\frac{\ell-\sum_{k=0}^j\ell_k (z\overline{w})^k}{k}\succeq 0,
    \end{equation}
    for all $j\ge 0.$ Since \eqref{nicc} is  of the form $\ell_M/k\succeq 0,$ where $M$ is the $\Mult(\mathcal{H}_{\ell})$-invariant subspace of $\mathcal{H}_{\ell}$ that is determined by the vanishing of the first $j$ derivatives, we actually have the stronger conclusion that \begin{equation} \label{pair}
      \big(k, \ell-\sum_{k=0}^j\ell_k (z\overline{w})^k\big)  
    \end{equation}
    is a CP pair for every $j\ge 0.$ Indeed, letting $s$ denote any strong Shimorin certificate for $(k, \ell)$, Lemma~\ref{invfactors} tells us that $s$ is a strong Shimorin certificate for \eqref{pair} as well. Note that \eqref{pair} being CP implies \eqref{nicc} because of Theorem~\ref{Shimnec}.
\end{remark}
\begin{remark}  \label{remark82}
Unfortunately, the conclusion of Remark~\ref{remark81} no longer holds for more general conditions of the form
\[\frac{\ell_{(m_0, \dots, m_j)}}{k}\succeq 0. \]
Indeed, consider the Bergman kernel $\mathfrak{b}$ and set $\ell(z, w)=(1-\mt(z\overline{w}))^{-1}=(1-2z\overline{w})^{-1}$. Since $(\mathfrak{b}, \ell)$ is a CP pair,  Theorem~\ref{sometheorem} tells us that \[\frac{\ell_{(2)}}{\mathfrak{b}}\succeq 0,\] where  
\begin{align*}
\ell_{(2)}(z, w)&=\ell(z, w)-\bigg(\frac{\ell}{\mathfrak{b}}\bigg)_2(z\overline{w})^2 \\ 
&=\ell(z, w)-(\ell_2-\ell_1\mathfrak{b}_1+(\mathfrak{b}^2_1-\mathfrak{b}_2))(z\overline{w})^2 \\
&=1+2(z\overline{w})+3(z\overline{w})^2+\dots.
\end{align*}
    However, $(\mathfrak{b}, \ell_{(2)})$ is not a CP pair; if it were, then we would have 
    $$\frac{ \ell_{(2)}-1}{\mathfrak{b}}\succeq 0,$$
    which is easily seen to be false. 
\end{remark}
\begin{remark} \label{remark83}
It is important that the sequence $m_0, m_1, m_2, \dots$ be strictly increasing. Indeed, choosing $k=\mathfrak{b}$ and $\ell$ as in Remark~\ref{remark82}, we have:
  \[\ell_{(2, 0)}=\ell_{(2)}-\bigg(\frac{\ell_{(2)}}{\mathfrak{b}}\bigg)_0=2(z\overline{w})+3(z\overline{w})^2+\dots,\] 
  but $\frac{\ell_{(2, 0)}}{\mathfrak{b}}$ is not positive.
\end{remark}

\normalsize
\section{Open Questions} \label{openq}
\normalsize
We have seen that if a pair $(k, \ell)$ possesses a Shimorin certificate, then it is a CP pair (Theorem~\ref{generalsuff}). Does the converse hold?
\begin{question} \label{q1}
Let $(k, \ell)$ be a CP pair of kernels on a set $X.$ Does there exist a Shimorin certificate for $(k, \ell)$?
\qed \end{question}
Specializing to holomorphic pairs over connected domains, Question~\ref{q1} becomes (in view of Proposition~\ref{Scertforhol}) the following.
\begin{question} \label{q25}
Let $(k, \ell)$ be a CP pair of holomorphic  kernels on a connected domain $\Omega.$ Does there exist a strong Shimorin certificate for $(k, \ell)$?
\qed \end{question}
Instead of tackling Questions \ref{q1}-\ref{q25} directly, it might be easier to first attempt to address certain consequences of the existence of Shimorin certificates. For instance, if the pair $(k, \ell)$ has a strong Shimorin certificate, the same holds for $(k, \ell_M),$ for any multiplier-invariant subspace $M$ of $\mathcal{H}_{\ell}$ (recall that $\ell_M$ denotes the reproducing kernel of $M$). As such, we may ask:
\begin{question}\label{extraque}
Let $(k, \ell)$ be a CP pair of kernels on $X$. Given a $\text{Mult}(\mathcal{H}_{\ell})$-invariant subspace $M$ of $\mathcal{H}_{\ell}$, is it necessary that $(k, \ell_M)$ is also CP? 
\qed  \end{question}
\noindent One could also observe that, if the pair $(k, \ell)$ has a (strong) Shimorin certificate, the same must be true for $(k, \ell h)$, for any kernel $h$, which leads to:
\begin{question} \label{q2}
Let $(k, \ell)$ be a CP pair of kernels on $X$. Given an arbitrary kernel $h$ on $X$, is it necessary that $(k, \ell h)$ is also CP? 
\qed \end{question}

\begin{question}
\label{q45}
Let $(k, \ell, \Omega)$ be a CP pair of holomorphic kernels with $\mathbf{0}\in\Omega$ and such that $k$ is also (normalized) diagonal. 
Let $\mt$ denote the master certificate associated with $k.$ Can $\Omega$ contain the closure of $\Omega_\mt^1?$  
\qed
\end{question} If no restriction is placed on $\ell$ (e.g. analyticity), then $\Omega$ can be anything; see Example \ref{anykernel}. We also point out that it is possible to have $\Omega\not\subset \Omega_\mt^1$ in the above setting, see Remark \ref{non-inclusion}.  However, we believe that the answer to the first (and thus also to the second) part of Question~\ref{q45} is no.  This is indeed the case when $k$ is Bergman-like (as in Examples \ref{finallyanex}, \ref{finallyanex2}); see Remark \ref{verygen}. The main obstacle in generalizing the proof of Theorem \ref{t:max:disc} to arbitrary (holomorphic) pairs $(k, \ell)$ lies in the fact that the necessary conditions of Lemma \ref{lemma2} will not, in general, be sufficient for the CP property, even if both kernels are diagonal (Proposition \ref{necbutnotsuf}).

We can also formulate a version of Question \ref{q45} where the assumption of the CP property is replaced by the (a priori stronger) assumption of the existence of a Shimorin certificate. Our motivation stems from the observation that all non-diagonal certificates associated with the Bergman kernel in Section \ref{closerl} have domains that do not contain $\Omega^1_{\mt}=\frac{1}{\sqrt{2}}\mathbb{D}$; see Remark \ref{non-inclusion} and Example \ref{lastex}. 
\begin{question}
\label{q:475}
 Suppose $k$ is a normalized diagonal holomorphic kernel 
 with master certificate $\mt.$ If $k=(1-h)s$ for a kernel $h$ and a holomorphic CP kernel $s$ on $\mathbf{0}\in\Omega$ with $\Omega^1_{\mt}\subseteq \Omega$ that is also normalized at $\mathbf{0}$, must $s$ be diagonal (in which case $s=(1-\mt)^{-1}$, by definition of $\mt$, and $\Omega\equiv\Omega^1_{\mt}$)?  What if $s$ is not holomorphic? 
 \qed
\end{question}  A pair $(k,\ell)$ of non-vanishing holomorphic kernels on a domain $\mathbf{0}\in \Omega$
  has a Shimorin certificate if and only if there is a kernel $p$ such that 
   $\ell^{\mathbf{0}} \succeq p\, \ell$ and $k^{\mathbf{0}}\preceq p\, k$, by Proposition~\ref{betweencerts} and Proposition~\ref{Scertforhol}.    Motivated by this, we offer the following definition: given a (non-vanishing) holomorphic kernel $k$ and another kernel $p$ on $\Omega,$ we will say that $p$ is a \textit{pre-certificate for $k$} if $k^{\mathbf{0}}\preceq p\, k$. Observe that if $p, q$ are both pre-certificates for $k$ with  $q\preceq p$ and $p$ is also a Shimorin certificate for $(k, \ell)$, then $q$ has to be a Shimorin certificate as well. We will say that the pre-certificate $p$ for $k$ is \textit{minimal} if, whenever  $q$ is a pre-certificate with $q\preceq p,$ then $q=p$.     
In the case $k=\mathfrak{b}$ is the Bergman kernel, the kernels $p_\lambda=g_\lambda(z)\overline{g_\lambda(w)}$ are minimal pre-certificates, where we set $p_0(z, w)=2z\overline{w}$.  Indeed, if $q$ is any pre-certificate for $\mathfrak{b}$ satisfying $q\preceq p_\lambda$ for some $\lambda$, then $p_{\lambda}$ being rank $1$ implies the existence of $0\le c\le1$ such that $q=cp_{\lambda}$. It is easy to see that $2c^2z\overline{w}=c^2p_0(z, w)$ cannot satisfy $\mathfrak{b}=(1-h)/(1-c^2p_0)$ if $c<1$. Also, for $\lambda\neq 0$ and $c<1$, we have
\[(c^2p_{\lambda}(\lambda, \lambda)-1)\mathfrak{b}(\lambda, \lambda)+1<0,\]
 and minimality has been established. 
   
 \begin{question}
\label{q:6}
Do there exist minimal pre-certificates for the Bergman kernel $\mathfrak{b}$ that are not of the form $p_\lambda=g_\lambda(z)\overline{g_\lambda(w)}$?
  \end{question}  
   \begin{question}
\label{q:65}
 As a variation on Question~\ref{q:6} and still with $\mathfrak{b}$ the Bergman kernel, if $\ell$ is holomorphic and $(\mathfrak{b},\ell)$ has a strong Shimorin certificate, then does it have a rank one certificate; that is,  a certificate  of the form $\frac{1}{1-g(z)\overline{g(w)}}$ for a function $g:\Omega\to\mathbb{C}$? 
 \qed
\end{question}
There is a competing notion of minimal certificate. Given holomorphic kernels $k$ and $s$ on $\mathbf{0}\in\Omega$  with $s$ CP, we could define $s$ to be a minimal certificate for $k$ over $\Omega$ if $k=(1-h)s$ for some kernel $h$ and also the following is satisfied: whenever $\tilde{s}$ is a holomorphic CP kernel on $\Omega$ satisfying $k=(1-\tilde{h})\tilde{s}$ for some kernel $\tilde{h}$ and also  $s/\tilde{s}\succeq 0$, we have that $s$ and $\tilde{s}$ are equal up to a dyad. We point out that, given any sub-domain $\mathbf{0}\in\Omega'\subset\Omega$, $s$ is a minimal certificate for $k$ over $\Omega'$ if and only if it is one over $\Omega.$ This can be proved using standard reproducing kernel arguments; see e.g. the proof of the inclusion $\Omega^1_{\mt}\subset \Omega_k$ in Proposition \ref{l:tlg}.

\begin{question} \label{q:7}
Are the certificates $(1-g_{\lambda}(z)\overline{g_{\lambda}(w)})^{-1}$ (where $|\lambda|<1$) minimal for $\mathfrak{b}$ over $\frac{1}{3}\DD$ with respect to this competing definition of minimality?
\end{question}
So, we can find minimal certificates, but what about minimum ones? Given a kernel $k$ and a CP kernel $s$ on a set $X$ with $k=(1-h)s$ over $X,$ we will say that $s$ is a \textit{minimum} certificate for $k$ over $X$ if, whenever $\tilde{s}$ is another CP kernel on $X$ with $k=(1-\tilde{h})\tilde{s}$, then $\tilde{s}/s\succeq 0.$  It is not hard to see that such a certificate will be unique up to a dyad. Now, if $k$ is CP, then, for any kernel $\ell,$ $(k, \ell)$ is CP if and only if $\ell/k\succeq 0$ (this follows from Theorems \ref{mainShim} and \ref{Shimnec}). Thus, $k$ is a minimum certificate for itself, and this continues to hold over any sub-domain $X'\subset X.$ This motivates the following question.
\begin{question} \label{strongcP}
 Let $k$ be a kernel on a non-empty set $X$. Suppose there exists a CP kernel $s$ on $X$ with $k=(1-h)s$ and such that, for any non-empty $X'\subset X,$ the kernel $s|_{X'\times X'}$ is a minimum certificate for $k|_{X'\times X'}$. Does that imply that $k$ is CP (in which case $k=s$ up to a dyad)?
 \qed
\end{question}
Now, Corollary \ref{canoncorol} tells us that no minimum certificate exists for $k=\mathfrak{b}$ over $\frac{1}{3}\mathbb{D}.$ What if we insist on a larger domain?
\begin{question} \label{num}
Does a minimum certificate exist for $\mathfrak{b}$ over $\Omega^1_{\mt}=\frac{1}{\sqrt{2}}\DD$?  
\qed
\end{question}

\par\textit{Acknowledgements}. The second author is grateful to his advisor, John M\raise.5ex\hbox{c}Carthy, for his advice and support. He would also like to thank Michael Hartz and Greg Knese for helpful discussions.

\newpage 
\appendix 
\section{One-point Pick extensions} \label{1stap}

\begin{proof}[Proof of Lemma~\ref{boringonepoint}]

\label{a:boringponepoint}
First, define three projection operators $P, L\in\mathcal{B}(\mathcal{M}_{n+1}^{\ell})$ and  $Q\in\mathcal{B}(\mathcal{M}^k_{n+1})$ that satisfy
\begin{equation*}
  \begin{split}
  \text{ran} P &=\mathcal{M}^{\ell}_{n} \\
 \text{ran} Q &=\{{k_{z_{n+1}}}\}^{\perp} \\
 \text{ran} L &=\{{\ell_{z_{n+1}}}\}^{\perp}
  \end{split}  
\end{equation*}
 and decompose $R_W$ as 
 \begin{equation} \label{RWdecomp}
\begin{blockarray}{ccc}
 & P & P^{\perp} \\ 
\begin{block}{c(cc)}
Q & QR_WP & QR_WP^{\perp} \\
Q^{\perp} & Q^{\perp}R_WP & Q^{\perp} R_WP^{\perp} \\ 
\end{block}
\end{blockarray}\ .
 \end{equation}
Further, let $\{\xi_i\}$ be a dual basis to $\{\ell_{z_i}\}$, so
\[\langle \ell_{z_i}, \xi_j\rangle =\delta_{ij}, \hspace{0.4 cm} i, j=1,\dots,n+1.\]
Write $\xi_{n+1}=\sum_{i=1}^{n+1}c_i\ell_{z_i}.$ Then,
\[R_WP^{\perp}\big(\xi_{n+1}\otimes u^{\alpha}\big)=R_W\big(\xi_{n+1}\otimes u^{\alpha}\big)=\sum_{i=1}^{n}c_ik_{z_i}\otimes W_i^*u^{\alpha}+c_{n+1}k_{z_{n+1}}\otimes W^*u^{\alpha},\]
so 
\begin{equation} \label{QperpRPperp}
Q^{\perp}R_WP^{\perp}\big(\xi_{n+1}\otimes u^{\alpha}\big)=k_{z_{n+1}}\otimes \bigg[\sum_{i=1}^nc_i\frac{k(z_{n+1}, z_i)}{k(z_{n+1}, z_{n+1})}W_i^*+c_{n+1}W^*\bigg]u^{\alpha}.
\end{equation}
Now, observe that, since $\ell_{z_1}, \dots, \ell_{z_{n+1}}$ are linearly independent, $c_{n+1}$ must be non-zero. As we are free to choose $W,$ this means that the expression in brackets in \eqref{QperpRPperp} can be made equal to whatever we want. Thus, the $(2, 2)$ entry of \eqref{RWdecomp} can be chosen arbitrarily. \par 
Further, we see that 
\[QR_WP^{\perp}\big(\xi_{n+1}\otimes u^{\alpha}\big)=Q\bigg(\sum_{i=1}^nc_ik_{z_i}\otimes W_i^*u^{\alpha}\bigg).\]
Thus $QR_WP^{\perp}$ does not depend on $W.$ Since the same obviously holds for $QR_WP$ and $Q^{\perp}R_WP$, we conclude that the smallest norm of \eqref{RWdecomp} coincides with the smallest norm of a matrix completion of 
\begin{equation} \label{RWcompl}
 \begin{blockarray}{ccc}
 & P & P^{\perp} \\ 
\begin{block}{c(cc)}
Q & QR_WP & QR_WP^{\perp} \\ 
Q^{\perp} & Q^{\perp}R_WP & * \\ 
\end{block}
\end{blockarray}\ .
\end{equation}
By Parrott's Lemma, we obtain
\begin{equation}
 \inf_{W\in M_{N}(\mathbb{C})}\|R_W\|=\max\big\{{\|R_WP\|, \|QR_W\|}\big\}=\max\big\{\|R\|, \|QR_WL\|\big\},   
\end{equation} 
where the last equality holds because $R_WP=R$ and $QR_WL^{\perp}=R_WL^{\perp}-Q^{\perp}R_WL^{\perp}=0,$ thus 
$QR_W=QR_WL.$ \par 
Since we have $\|R\|\le 1$ by assumption, all that remains is to show $\|QR_WL\|\le 1,$ which is equivalent to 
\[L-LR_W^*QQR_WL \succeq 0.\]
Thus, we have to verify that this $nN\times nN$ matrix is positive. We calculate
\[
\begin{split}
\big\langle \big[L&-LR_W^*QQR_WL\big]\ell_{z_j}\otimes u^{\beta}, \ell_{z_i}\otimes u^{\alpha} \big\rangle 
=\Big\langle \Big(\ell_{z_j}-\frac{\langle \ell_{z_j}, \ell_{z_{n+1}}\rangle}{\big|\big|\ell_{z_{n+1}}\big|\big|^2}\ell_{z_{n+1}}\Big)\otimes u^{\beta}, \ell_{z_i}\otimes u^{\alpha}\Big\rangle  \\
&-\Big\langle LR^*_WQQ\Big(k_{z_j}\otimes W^*_ju^{\beta}
-\frac{\langle \ell_{z_j}, \ell_{z_{n+1}}\rangle}{\big|\big|\ell_{z_{n+1}}\big|\big|^2}k_{z_{n+1}} \otimes W^*_{n+1}u^{\beta}\Big), \ell_{z_i}\otimes u^{\alpha}\Big\rangle \\ 
=&\Big\langle \Big(\ell_{z_j}-\frac{\langle \ell_{z_j}, \ell_{z_{n+1}}\rangle}{\big|\big|\ell_{z_{n+1}}\big|\big|^2}\ell_{z_{n+1}}\Big)\otimes u^{\beta}, \ell_{z_i}\otimes u^{\alpha}\Big\rangle  -\Big\langle QQ\Big(k_{z_j}\otimes W^*_ju^{\beta}
\Big), R_WL\big(\ell_{z_i}\otimes u^{\alpha}\big)\Big\rangle \\
=&\Big\langle \Big(\ell_{z_j}-\frac{\langle \ell_{z_j}, \ell_{z_{n+1}}\rangle}{\big|\big|\ell_{z_{n+1}}\big|\big|^2}\ell_{z_{n+1}}\Big)\otimes u^{\beta}, \ell_{z_i}\otimes u^{\alpha}\Big\rangle  \\
&-\Big\langle  QQ\Big(k_{z_j}\otimes W^*_ju^{\beta}
\Big), k_{z_i}\otimes W^*_iu^{\alpha}
-\frac{\langle \ell_{z_i}, \ell_{z_{n+1}}\rangle}{\big|\big|\ell_{z_{n+1}}\big|\big|^2}k_{z_{n+1}} \otimes W^*_{n+1}u^{\alpha}\Big\rangle \\
=&\Big\langle \Big(\ell_{z_j}-\frac{\langle \ell_{z_j}, \ell_{z_{n+1}}\rangle}{\big|\big|\ell_{z_{n+1}}\big|\big|^2}\ell_{z_{n+1}}\Big)\otimes u^{\beta}, \ell_{z_i}\otimes u^{\alpha}\Big\rangle \\ &-\Big\langle  \Big(k_{z_j}-\frac{\langle k_{z_j}, k_{z_{n+1}}\rangle}{\big|\big|k_{z_{n+1}}\big|\big|^2}k_{z_{n+1}}\Big)\otimes W^*_ju^{\beta}, k_{z_i}\otimes W^*_iu^{\alpha}
\Big\rangle \\ 
=&\ \ell^{z_{n+1}}(z_i, z_j)\langle u^{\beta}, u^{\alpha}\rangle-k^{z_{n+1}}(z_i, z_j)\big\langle W_iW^*_ju^{\beta}, u^{\alpha}\big\rangle.
\end{split}
\]
Thus, $L-LR_W^*QQR_WL$ is positive if and only if 
\begin{equation*}  
\big[\ell^{z_{n+1}}(z_i, z_j) I_{N\times N}-k^{z_{n+1}}(z_i, z_j)W_iW_j^* \big]_{i, j=1}^n\succeq 0,
\end{equation*}
as desired.

\end{proof}

\section{More on Shimorin certificates and zero sets}
 Further consequences of  the existence of a Shimorin certificate for $(k,\ell)$ appear in this appendix.

 \begin{remark}
    In Section \ref{ShimnstrShim}, we saw examples of Shimorin certificates with the property that, for any $w, v\in X$ with $w\neq v,$ the sets $X^w_0, X^v_0$ are always disjoint. Unfortunately, this will not be true in general. To see why, observe that, given any two Shimorin certificates $\{p[z]\}$ and $\{q[z]\}$ and an arbitrary decomposition $X=X_1\cup X_2,$ one can always consider the collection $\{p[z]\}_{z\in X_1}\cup\{q[z]\}_{z\in X_2}$, which will be a Shimorin certificate with potential overlap between $X^w_0$ and $X^v_0$ . 
\end{remark} 

The next two lemmas will  shed further light on the relation between $X^w_0, X^w_1$ appearing in equation~\eqref{certdecomp} and $(k, \ell)$ and so might be of independent interest. 
\begin{lemma} \label{noname:1}
Let $w\in X.$ For any $z, v\in X^w_1,$ we either have $\ell(z, v)=0$ or $p[w](z, v)=1.$
\end{lemma}
\begin{proof}
Note that, since $w\in X^w_0$, Lemma~\ref{ellsplit} implies that all diagonal entries of \[0\preceq \big(\ell^w-p[w]\ell\big)_{X^w_1\times X^w_1}=\big(\ell-p[w]\ell\big)_{X^w_1\times X^w_1} \] are  equal to zero. Thus, every other entry has to be zero as well, so  \[\ell(z, v)-p[w](z,v)\ell(z, v)=0\] for every $z,v \in X^w_1,$ and the conclusion follows. 
\end{proof}

\begin{lemma}\label{noname:2}
Let $w\in X$, $v\in X^w_1$ and $z\in X^w_0$. If $k(v, z)=0,$ then $\big(p[w](v,u)-1\big)k(v, u)=0$ for all $u\in X, u\neq v.$ In particular, $k(v, u)=0$ for all $u\in X^w_0.$
\end{lemma}
\begin{proof}
Let $w\in X,$ with $z\neq w$ (not necessarily in $X^w_0$) and $v\in X^w_1$, $v\neq z$. Since $k^w\preceq p[w] k,$ we have
\[\begin{bmatrix}
k^w(z, z) & k^w(z, v) \\ 
k^w(v, z) & k^w(v, v)
\end{bmatrix}  \preceq  \begin{bmatrix}
p[w](z, z)k(z, z) & p[w](z, v)k(z, v) \\ 
p[w](v, z)k(v, z) & k(v, v)
\end{bmatrix},\]
which, in turn, implies 
\[\frac{|k(v, w)|^2}{k(w, w)}k(z, z)\big(p[w](z, z)-1\big)\]
\begin{equation} \label{pw}
     \ge \big|k(z, v)\big(p[w](z, v)-1\big)\big|^2+2\Re\bigg(\big(p[w](z, v)-1\big)\frac{k(z, v)k(w, z)k(v, w)}{k(w, w)}\bigg).
\end{equation}
Assume now that $z\in X^w_0$. From this last inequality, and since $|p[w](z, v)|\le p[w](z, z)<1$, we easily obtain that $k(z, v)=0$ if and only if $k(w, v)=0.$ Thus, if we assume $k(z, v)=0,$ then we may replace $z$ by $u\in X$ (with $u\neq w, v$) in \eqref{pw} to obtain  $k(u, v)\big(p[w](u, v)-1\big)=0$. Clearly, this equality continues to hold if we set $u=w$ (since $v\in X^w_1$). 
\end{proof}

 \printbibliography

 \printindex

\end{document}